\documentclass[11pt,twoside,a4paper]{article} 
\usepackage[T1]{fontenc}
\usepackage[english]{babel}
\usepackage[latin1]{inputenc}
\usepackage[a4paper]{geometry}
\usepackage{mathrsfs}
\usepackage{amsthm}
\usepackage{amsmath}
\usepackage{amssymb}
\usepackage{graphicx}
\usepackage{caption}
\usepackage{color}
\usepackage{enumerate}
\usepackage{esint}
\usepackage{fancyhdr}
\usepackage{amssymb,latexsym,amsmath,amsthm}     
\usepackage{bbm}
\usepackage{eqnarray}
\usepackage{color}
\usepackage[toc,page]{appendix}
\usepackage{scalerel,stackengine}
\usepackage{titling}
\usepackage{hyperref}
\hypersetup{
    colorlinks=true,
    linkcolor=blue,
    filecolor=blue,      
    urlcolor=blue,
    citecolor=blue,
    linktocpage=true
}
\makeatletter
\theoremstyle{plain}
\newtheorem{thm}{Theorem}[section]
\newtheorem{prop}[thm]{Proposition}
\newtheorem{defn}[thm]{Definition}
\newtheorem{lemma}[thm]{Lemma}
\newtheorem{corollary}[thm]{Corollary}
\newtheorem{assumption}[thm]{Assumption}
\theoremstyle{remark}
\newtheorem{remark}[thm]{Remark}
\newtheorem{example}[thm]{Example}
\newcommand{\eps}{{\varepsilon}} 
\newcommand{\e}{_{\varepsilon}} 
\newcommand{\R}{\mathbb{R}}
\newcommand{\N}{\mathbb{N}}

\newcommand{\cE}{{\mathcal{E}}}
\newcommand{\cB}{{\mathcal{B}}}
\newcommand{\sB}{{\mathscr{B}}}
\newcommand{\sD}{{\mathscr D}}
\newcommand{\sA}{{\mathscr A}}
\newcommand{\sM}{{\mathscr M}}
\newcommand{\tsq}[1]{{\stackrel{2s}{\rightharpoonup}_{#1}}}
\newcommand{\ts}{{\stackrel{2s}{\rightharpoonup}}}
\newcommand{\weakto}{\rightharpoonup}
\newcommand{\tomega}{\tilde{\omega}}
\newcommand{\CL}{{\mathscr{C\!P}}}
\newcommand{\p}{\varphi} 
\newcommand{\pinv}{P_{{\textsf{inv}}}} 
\newcommand{\ex}[1]{\left\langle {#1} \right\rangle} 
\newcommand{\unf}{\mathcal{T}_{\varepsilon}} 
\newcommand{\brac}[1]{\left({#1}\right) } 
\newcommand{\cb}[1]{\left\lbrace {#1} \right\rbrace}
\newcommand{\pot}{{\textsf{pot}}}
\newcommand{\inv}{{\textsf{inv}}}
\newcommand{\ns}[1]{\|{#1}\|_{L^p(\Omega\times Q)}} 
\newcommand{\ltp}{L^p(\Omega \times Q)} 
\newcommand{\ltq}{L^q(\Omega \times Q)}
\newcommand{\wt}{\overset{2s}{\rightharpoonup}} 
\newcommand{\st}{\overset{2s}{\rightarrow}} 
\newcommand{\re}[1]{\mathbb{R}^{#1}}
\newcommand{\sob}{W^{1,p}}

\makeatother

\usepackage{authblk}
\title{Stochastic unfolding and homogenization}
\author[1]{Martin Heida\thanks{martin.heida@wias-berlin.de}} 
\author[2]{Stefan Neukamm\thanks{stefan.neukamm@tu-dresden.de}}
\author[2]{Mario Varga\thanks{mario.varga@tu-dresden.de}}
\affil[1]{Weierstrass Institute for Applied Analysis and Stochastics, Berlin}
\affil[2]{Faculty of Mathematics, Technische Universit\"at Dresden}
\begin{document}
\maketitle
\begin{abstract}
The notion of periodic two-scale convergence and the method of periodic unfolding are prominent and useful tools in multiscale modeling and analysis of PDEs with rapidly oscillating periodic coefficients. In this paper we are interested in the theory of stochastic homogenization for continuum mechanical models in form of PDEs with random coefficients, describing random heterogeneous materials. The notion of periodic two-scale convergence has been extended in different ways to the stochastic case. In this work we introduce a stochastic unfolding method that features many similarities to periodic unfolding. In particular it allows to characterize the notion of stochastic two-scale convergence in the mean by mere weak convergence in an extended space. We illustrate the method on the (classical) example of stochastic homogenization of convex integral functionals, and prove a new result on stochastic homogenization for a non-convex evolution equation of  Allen-Cahn type. Moreover, we discuss the relation of stochastic unfolding to previously introduced notions of (quenched and mean) stochastic two-scale convergence. The method described in the present paper extends to the continuum setting the notion of discrete stochastic unfolding, as recently introduced by the second and third author in the context of discrete-to-continuum transition.
\medskip

\noindent
{\bf Keywords:} stochastic homogenization, unfolding, two-scale convergence, gradient systems
\end{abstract}

\setlength{\parindent}{0pt}
\tableofcontents

\section{Introduction}\label{Intro}
Homogenization theory deals with the derivation of effective, macroscopic models for problems that involve two or more length-scales. Typical examples are continuum mechanical models for microstructured materials that give rise to boundary value problems or evolutionary problems for partial differential equations with coefficients that feature rapid, spatial oscillations. The first results in homogenization theory were motivated by a mechanics problem which was about the determination of the macroscopic behavior of linearly elastic composites with periodic microstructure, see Hill \cite{Hill1963}. In the mathematical community early contributions in the 70s came from the \textit{French school} (e.g.~see \cite{Bensoussan1978} for an early standard reference, and \cite{Tartar1977,Murat1997} for Tartar and Murat's notion of $H$-convergence), the \textit{Russian school} (e.g.~Zhikov, Kozlov and Oleinik, see \cite{Zhikov1979,jikov2012homogenization}), and from the \textit{Italian school} for variational problems (e.g., Marcellini \cite{Marcellini1978}, Spagnolo \cite{Spagnolo1976} for $G$-convergence, and De Giorgi and Franzoni for $\Gamma$-convergence \cite{DeGiorgi1975}). In the 80s and later, homogenization was intensively studied for a variety of models from continuum mechanics including non-convex integral functionals and applications to non-linear elasticity (e.g.~M\"uller \cite{Mueller1987, Geymonat1993} and Braides \cite{Braides1985}),  or the topic of effective flow through porous media (e.g.~see Hornung et al. \cite{arbogast1990derivation,Hornung1991,Hornung2012} and Allaire \cite{Allaire1989}). Most results in homogenization theory discuss problems with periodic microstructure, and specific analytic tools for periodic homogenization of linear (or monotone) operators are developed, including the notions of two-scale convergence  and  periodic unfolding \cite{nguetseng1989general, allaire1992homogenization, Visintin2006, cioranescu2002periodic}, which by now are standard tools in  multiscale modeling and analysis. In the last decade considerable interest in applied mathematics emerged in understanding random heterogeneous materials, i.e.~materials whose properties on a small length-scale are only described on a statistical level, such as polycrystalline composites, foams, or biological tissues, see \cite{Torquato2013} for a standard reference. Although the first results in stochastic homogenization were already obtained in the 70s and 80s for linear elliptic equations and convex minimization problems, see \cite{Papanicolaou1979, Kozlov1979, DalMaso1985,DalMaso1986}, the theory in the stochastic case is still less developed as in the periodic case and object of various recent studies, e.g.~regarding error estimates and regularity properties (see \cite{GO11,GO12, GNO1, GNO2, GNOreg, armstrong2016quantitative, Armstrong2017}, or modeling of random heterogeneous materials~\cite{Zhikov2006,Alicandro2011, Cicalese2017, Hornung2017, Heida2017b, HeidaNesenenko2017monotone, Berlyand2017, neukamm2017stochastic}. With the present paper we contribute to the latter. In particular, we introduce a \textit{stochastic unfolding method} that shares many similarities to periodic unfolding and two-scale convergence with the intention to systematize and simplify the process of lifting results from periodic homogenization to the stochastic case. We illustrate this  by reconsidering stochastic homogenization of convex integral functionals and by proving a new stochastic homogenization result for semilinear gradient flows of Allen-Cahn type. In order to put the notion into perspective, in the following we recall the concepts of two-scale convergence and periodic unfolding.
\medskip

For problems with periodic coefficients, the notion of (periodic) two-scale convergence was introduced in \cite{nguetseng1989general} and further developed in \cite{allaire1992homogenization,lukkassen2002two}. Two-scale convergence refines weak convergence in $L^p$-spaces: The two-scale limit captures not only the averaged behavior of an oscillating sequence (as opposed to the weak limit), but also oscillations on a prescribed small scale $\varepsilon$. In particular, let $Q\subset \R^d$ and $\Box=[0,1)^d$, a sequence $(u_{\varepsilon})\subset L^p(Q)$ two-scale converges to $u\in L^p(Q\times \Box)$ (as $\varepsilon\to 0$) if 
\begin{equation*}
\lim_{\varepsilon \to 0}\int_{Q}u_{\varepsilon}(x)\varphi\brac{x,\frac{x}{\varepsilon}}dx =  \int_{Q}\int_{\Box}u(x,y)\varphi(x,y)dy dx,
\end{equation*}   
for all $\varphi \in L^q(Q; C_{\#}(\Box))$. Here $C_{\#}(\Box)$ denotes the space of continuous and $\Box$-periodic functions and $p,q\in (1,\infty)$ are dual exponents.
\medskip

In \cite{arbogast1990derivation} in the specific context of homogenization of flow through porous media Arbogast et~al.~introduced a dilation operator to resolve oscillations on a prescribed scale of weakly converging sequences; it turned out that the latter yields a characterization of two-scale convergence (see \cite[Proposition~4.6]{bourgeat1996convergence}). In a similar spirit, Cioranescu et al.~introduced in \cite{cioranescu2002periodic, cioranescu2008periodic} the periodic unfolding method as a systematic approach to homogenization. The key object of this method is a linear isometry $\mathcal{T}_{\varepsilon}^{p}:L^p(Q)\to L^p(Q\times \Box)$ (the periodic unfolding operator) which invokes a change of scales and allows (at the expense of doubling the dimension) to use standard weak and strong convergence theorems in $L^p$-spaces to capture the microscopic behavior of oscillatory sequences. It turned out that the method is well-suited for periodic multiscale problems, e.g.~see~\cite{cioranescu2012periodic,griso2004error,mielke2007two, Visintin2006, neukamm2010homogenization, hanke2011homogenization, liero2015homogenization}. Moreover, the unfolding method allows to rephrase two-scale convergence: Applied to an oscillatory sequence $(u_{\varepsilon})\subset L^p(Q)$, the unfolded sequence $(\mathcal{T}_{\varepsilon}^{p}u_{\varepsilon})$ weakly converges in $L^p(Q \times \Box)$ if and only if $(u_{\varepsilon})$ two-scale converges, and the corresponding limits are the same. We refer to \cite{mielke2007two} where this perspective on two-scale convergence is investigated and applied in the context of  evolutionary problems.
\medskip

Motivated by the idea of (periodic) two-scale convergence, in \cite{bourgeat1994stochastic} the notion of stochastic two-scale convergence in the mean was introduced suited for homogenization problems that invoke random  coefficients, see also \cite{andrews1998stochastic}. In stochastic homogenization typically random coefficients of the form $a(\omega,x)=a_0(\tau_x\omega)$ (for $x\in\R^d$) are considered where $\omega$ stands for a ``random configuration'' and $a_0$ is defined on a probability space $(\Omega,\mathcal F,P)$ that is equipped with a measure preserving action $\tau_x:\Omega\to\Omega$, see Section \ref{S_Desc}. A sequence $(u_{\varepsilon})\subset L^p(\Omega\times Q)$ (where $Q\subset\R^d$ denotes a continuum domain) is said to two-scale converge in the mean to some $u\in L^p(\Omega\times Q)$ if 
\begin{equation*}
\lim_{\varepsilon\to 0} \int_{\Omega}\int_{Q}u_{\varepsilon}(\omega,x)\varphi(\tau_{\frac{x}{\varepsilon}}\omega,x)dx dP(\omega)= \int_{\Omega}\int_{Q}u(\omega,x)\varphi(\omega,x)dx dP(\omega)
\end{equation*} 
for all $\varphi \in L^q(\Omega\times Q)$ satisfying suitable measurability conditions.
\medskip

Motivated by the concept of the periodic unfolding method, in \cite{neukamm2017stochastic} the second and third author developed a stochastic unfolding method for a discrete-to-continuum analysis of discrete models of random heterogeneous materials. In the present work, we extend the concept to problems defined on continuum domains $Q\subset\R^d$. In particular, we introduce a stochastic unfolding operator $\mathcal{T}_{\varepsilon}: L^p(\Omega\times Q)\to L^p(\Omega\times Q)$ which is an isometric isomorphism (see Section \ref{S_Stoch_1}).  It displays similar properties as the periodic unfolding operator; in particular, weak convergence of the unfolded sequence $(\mathcal{T}_{\varepsilon}u_{\varepsilon})$ is equivalent to stochastic two-scale convergence in the mean, and -- as in the periodic case -- we recover a compactness statement for two-scale limits of gradients.
\medskip

A first example that we treat via stochastic unfolding is the classical problem of stochastic homogenization of convex integral functionals. As in the periodic case, the proof of the homogenization theorem via unfolding is merely based on elementary properties of the unfolding operator and on (semi-)continuity of convex functionals (with suitable growth assumptions). The second example we consider is homogenization for gradient flows driven by $\lambda$-convex energies. In particular, we consider an Allen-Cahn type equation with random and oscillating coefficients. The argument follows an abstract strategy for evolutionary $\Gamma$-convergence of gradient systems, see \cite{mielke2016evolutionary} and the references therein (we provide more references in Section 3). The homogenization results that we obtain via stochastic unfolding establish convergence in the \textit{mean} (i.e.~in a statistically averaged sense). This is in contrast to \textit{quenched} homogenization, where a finer topology is considered -- namely convergence for almost every random realization. Although homogenization in the mean (via unfolding) is easier to prove than homogenization in a quenched sense (which in most cases relies on a subadditive ergodic theorem), typically the homogenization limits in both cases are the same, see  Section~\ref{Section:3:3}. One thus might view stochastic unfolding as a convenient and easy tool to rigorously identify homogenized models.
\medskip

The alternative, \textit{quenched} notion of stochastic two-scale convergence was introduced by Zhikov and Piatnitski in \cite{Zhikov2006}. In a very general setting, they introduced  two-scale convergence on random measures as a generalization of periodic two-scale convergence as presented in \cite{Zhikov2000}. In this work, we restrict to the simplest case where the random measure is the Lebesgue measure. The concept of stochastic two-scale convergence in \cite{Zhikov2006} is based on Birkhoff's ergodic theorem. 
Although the definition of (quenched) stochastic two-scale convergence,  which we recall in Section~\ref{Section_4}, and two-scale convergence in the mean look quite similar, it is non-trivial to relate both notions. In this paper we investigate this issue and provide some tools that allow to draw conclusions on quenched homogenization from mean homogenization. As an example we treat convex integral functionals. For the analysis, we appeal to Young measures generated by stochastically two-scale convergent sequences in the mean and in particular establish a compactness result (see Theorem \ref{thm:Balder-Thm-two-scale} and Lemma \ref{lem:Balder-Lem-two-scale}). Moreover, we exploit a lower semicontinuity result of convex integral functionals w.r.t.~quenched stochastic two-scale convergence that has been recently obtained by the first author and Nesenenko in \cite{HeidaNesenenko2017monotone}.

\medskip

\textbf{Structure of the paper.} In Section \ref{S_Stoch} we introduce the standard setting for stochastic homogenization, introduce the notion of stochastic unfolding and derive the most significant properties of the unfolding operator. In the following Section \ref{Section_Applications} two examples of the homogenization procedure via stochastic unfolding are presented. Namely, Section \ref{Section_Convex} is dedicated to homogenization of convex functionals and in Section \ref{Section_3.2} homogenization for Allen-Cahn type gradient flows is provided. In Section \ref{Section_4} we discuss the relations of stochastic unfolding and quenched stochastic two-scale convergence. Section~\ref{S_Stoch} and ~\ref{Section_Convex}, which contain the basic concepts and the application to convex homogenization, are self-contained and require only basic input from functional analysis. Section~\ref{Section_3.2} and Section~\ref{Section_4} require some advanced tools from analysis and measure theory.

\section{Stochastic unfolding and properties}\label{S_Stoch}
\subsection{Description of random media - a functional analytic framework}\label{S_Desc}
To fix ideas we consider for a moment the setup of Papanicolaou and Varadhan \cite{Papanicolaou1979} for homogenization of elliptic operators of the form $-\nabla\cdot a(\frac{x}{\varepsilon})\nabla$ with a coefficient field $a:\R^d\to\R^{d\times d}$. In the stochastic case the coefficients are assumed to be random and thus $a$ can be viewed as a family of random variables $\{a(x)\}_{x\in\R^d}$. A minimal requirement for stochastic homogenization of such operators is that the distribution of the coefficient field is \textit{stationary} and \textit{ergodic}. Stationarity means that the coefficients are statistically homogeneous (i.e.~for any finite set of points $x_1,\ldots,x_n\in\R^d$ the joint distribution of the \textit{shifted} random variables $a(x_1+z),\ldots,a(x_n+z)$ is independent of $z\in\R^d$), while \textit{ergodicity} (see below for the precise definition) is an assumption that ensures a separation of scales in the sense that long-range correlations of the coefficients become negligible in the large scale limit, e.g.~$\mbox{cov}[\fint_{B+z}a,\fint_Ba]\to 0$ as $z\to\infty$. In \cite{Papanicolaou1979}, Papanicolaou and Varadhan introduced a (by now standard) setup that allows to phrase these conditions in the following functional analytic framework (see also \cite{jikov2012homogenization}):
\begin{assumption} \label{Assumption_2_1}
Let $\brac{\Omega,\mathcal{F},P}$ denote a probability space with a countably generated $\sigma$-algebra, and let $\tau=\cb{\tau_x}_{x\in \re{d}}$ denote a group of measurable, bijections $\tau_x:\Omega\to \Omega$ such that
\begin{enumerate}[(i)]
\item (group property). $\tau_0=Id$ and $\tau_{x+y}=\tau_x\circ \tau_y$ for all $x,y\in \re{d}$,
\item (measure preserving). $P(\tau_x A)=P(A)$ for all $A\in \mathcal{F}$ and $x\in \re{d}$,
\item (measurability). $(\omega,x)\mapsto \tau_{x}\omega$ is $\mathcal{F}\otimes \mathcal{L}$-measurable ($\mathcal{L}$ denotes the Lebesgue-$\sigma$-algebra on $\R^d$).
\end{enumerate}
\end{assumption}
From now on we assume that $(\Omega,\mathcal F,P,\tau)$ satisfies these assumptions and we write $\langle\cdot\rangle:=\int_\Omega\cdot\, dP$ as a shorthand for the expectation. 
\medskip

In the functional analytic setting, a random coefficient field is described by a map $a:\Omega\times \R^d\to\R^{d\times d}$ with the interpretation that $a(\omega,\cdot):\R^d\to\R^{d\times d}$ with $\omega\in\Omega$ sampled according to $P$ yields  a realization of the random coefficient field.  Likewise, solutions to an associated PDE with physical domain $Q\subset\R^d$ might be considered as \textit{random} functions, i.e.~quantities defined on the product $\Omega\times Q$. In this paper we denote by $L^p(\Omega)$ and $L^p(Q)$ (with $Q\subset\R^d$ open) the usual Banach spaces of $p$-integrable functions defined on $(\Omega,\mathcal F,P)$ and $Q$, respectively. We introduce function spaces for functions defined on $\Omega\times Q$ as follows: For closed subspaces $X\subset L^p(\Omega)$ and $Y\subset L^p(Q)$ (resp. $Y\subset W^{1,p}(Q)$) we denote by $X\otimes Y$ the closure of $$X\overset{a}{\otimes}Y:=\cb{\sum_{i=1}^{n}\varphi_i \eta_i:  \varphi_i \in X, \eta_i\in Y, n\in \mathbb{N}}$$ in $L^p(\Omega;L^p(Q))$ (resp. $L^p(\Omega;W^{1,p}(Q))$, with a slight abuse of notation we use ``$X\otimes Y$'' for both type of spaces). Since the probability space is countably generated, $L^p(\Omega)$ (with $1\leq p<\infty$) is separable, and thus we have $L^p(\Omega)\otimes L^p(Q)=L^p(\Omega\times Q)=L^p(\Omega;L^p(Q))$ up to isometric isomorphisms. We therefore simply write $L^p(\Omega\times Q)$ instead of $L^p(\Omega)\otimes L^p(Q)$.
\smallskip

In the functional analytic setting and in view of the measure preserving property of $\tau$, the requirement of stationarity can be rephrased as the assumption that the coefficient field can be written in the form $a(\omega,x)=a_0(\tau_x\omega)$ for some measurable map $a_0:\Omega\to\R^{d\times d}$. The transition from $a_0$ to $a$ conserves measurability. As usual we denote by $\mathcal B(Q)$ (resp. $\mathcal L(Q)$) the Borel (resp. Lebesgue)-$\sigma$-algebra on $Q\subset\R^d$. The proof of the following lemma is obvious and therefore we do not present it.
\begin{lemma}[Stationary extension]\label{L:stat}
  Let $\varphi:\Omega\to\R$ be $\mathcal F$-measurable. Then $S\varphi:\Omega\times Q\to\R$, $S\varphi(\omega,x):=\varphi(\tau_x\omega)$ defines an $\mathcal F\otimes\mathcal L(Q)$-measurable function -- called the stationary extension of $\varphi$. Moreover, if $Q$ is bounded, for all $1\leq p<\infty$ the map $S:L^p(\Omega)\to L^p(\Omega\times Q)$ is a linear injection satisfying
  \begin{equation*}
    \|S\varphi\|_{L^p(\Omega\times Q)}=|Q|^\frac{1}{p}\|\varphi\|_{L^p(\Omega)}.
  \end{equation*}
\end{lemma}
The assumption of ergodicity can be phrased as follows:  We say $(\Omega,\mathcal F,P,\tau)$ is \textit{ergodic} (we also say $\ex{\cdot}$ is ergodic), if
\begin{align*}
  \text{ every shift invariant }A\in \mathcal{F} \text{ (i.e.~}\tau_x A=A \text{ for all }x\in \re d)\text{ satisfies } P(A)\in \cb{0,1}. 
\end{align*}
In this case the celebrated  ergodic theorem of Birkhoff applies, which we recall in the following form:
\begin{thm}[{Birkhoff's ergodic Theorem \cite[Theorem 10.2.II]{Daley1988}}]
\label{thm:ergodic-thm} Let $\ex{\cdot}$ be ergodic and  $\varphi:\Omega\to\R$ be integrable. Then for $P$-a.e.~$\omega\in\Omega$ it holds: $ S\varphi(\omega,\cdot)$ is locally integrable and for all open, bounded sets $Q\subset\R^d$ we have
\begin{equation}
\lim_{\eps\rightarrow0}\int_{Q}S\varphi(\omega,\tfrac{x}{\varepsilon})\,dx=|Q|\langle\varphi\rangle\,.\label{eq:ergodic-thm}
\end{equation}
Furthermore, if $\varphi\in L^p(\Omega)$ with $1\leq p\leq\infty$, then for $P$-a.e.~$\omega\in\Omega$ it holds: $S\varphi(\omega,\cdot)\in L_{loc}^{p}(\R^d)$, and provided $p<\infty$ it holds $S\varphi(\omega,\frac{\cdot}{\varepsilon})\weakto \langle\varphi\rangle$ weakly in $L_{loc}^{p}(\R^d)$ as $\eps\rightarrow0$.
\end{thm}
\begin{example}\label{example:1}
Basic examples for stationary and ergodic systems include the random checkerboard (e.g.~see \cite[Example 2.12]{Neukamm_lecture}), Gaussian random fields (e.g.~see \cite[Example 2.13]{Neukamm_lecture}).
We remark that the setting for periodic homogenization fits as well into this framework. In particular, $\Omega = \Box$ equipped with the Lebesque-$\sigma$-algebra and the Lebesgue measure, and the shift $\tau_{x}y = y + x \mod 1$ defines a system satisfying Assumption 2.1 and ergodicity. We refer to \cite{duerinckx2017weighted} for further examples of stationary and ergodic systems. 
\end{example}

\subsection{Stochastic unfolding operator and two-scale convergence in the mean}\label{S_Stoch_1}
In the following we introduce the stochastic unfolding operator, which is a key object in this paper. It is a linear, $\varepsilon$-parametrized, isometric isomorphism $\mathcal{T}_{\varepsilon}$ on $L^p(\Omega\times Q)$ where $Q\subset\R^d$ denotes an open set which we think of as the domain of a PDE.
\begin{lemma}\label{L:unf}
  Let $\varepsilon>0$, $1<p<\infty$, $q:=\frac{p}{p-1}$, and $Q\subset\R^d$ be open. There exists a unique linear isometric isomorphism
  \begin{equation*}
    \unf: \ltp\rightarrow \ltp
  \end{equation*}
  such that 
  \begin{equation*}
    \forall u\in L^p(\Omega)\overset{a}{\otimes} L^p(Q)\,:\qquad (\unf u)(\omega,x)=u(\tau_{-\frac{x}{\varepsilon}}\omega,x)\qquad \text{a.e. in }\Omega\times Q.
  \end{equation*}
  Moreover, its adjoint is the unique linear isometric isomorphism $\unf^{*}:\ltq\to\ltq$ that satisfies $(\unf^{*}u)(\omega,x)=u(\tau_{\frac{x}{\varepsilon}}\omega,x)$ a.e.~in $\Omega\times Q$ for all $u\in L^q(\Omega)\overset{a}{\otimes}L^q(Q)$.
  
(For the proof see Section \ref{S_Proofs}.)
\end{lemma}
\begin{defn}[Unfolding operator and two-scale convergence in the mean]\label{def46}
  The operator $\unf:\ltp\to\ltp$ of Lemma~\ref{L:unf} is called the stochastic unfolding operator. We say that a sequence $(u\e) \subset \ltp$ weakly (strongly) two-scale converges in the mean in $\ltp$ to $u\in \ltp$ if (as $\varepsilon\to 0$)
  \begin{equation*}
    \unf u\e \rightarrow u \quad \text{ weakly (strongly) in }\ltp.
  \end{equation*} 
  In this case we write $u\e\wt u$ (resp. $u_{\varepsilon} \st u$) in $L^p(\Omega\times Q)$.
\end{defn}

See Remark \ref{R_Stoch_1} for an explanation of the origin of the term \textit{weak/strong stochastic two-scale convergence in the mean} used for the above notion of convergence.

To motivate the definition, let $u\e\in H^1_0(Q)$ denote a (distributional) solution to $-\nabla\cdot a\e(x)\nabla u\e=f$ in $Q$, where $a\e$ is a family of uniformly elliptic, random coefficient fields of the form $a\e(\omega, x)=a_0(\tau_{\frac{x}{\varepsilon}}\omega)$. The main difficulty in homogenization of this PDE is the passage to the limit $\varepsilon\to 0$ in the product $a\e \nabla u\e$, since both factors in general only weakly converge. The stochastic unfolding operator $\unf$ turns this expression into a product of a strongly and a weakly convergent sequence in $L^2(\Omega\times Q)$: Indeed, we have $\unf(a\e\nabla u\e)=a_0(\unf\nabla u\e)$ and thus it remains to characterize the weak limit of $\unf\nabla u\e$, as will be done in the next section.

Since $\unf$ is an isometry, we obtain the following properties (which resemble the key properties of the periodic unfolding method). The below lemma directly follows from the isometry property of $\unf$ and the usual properties of weak and strong convergence in $L^p(\Omega\times Q)$; therefore, we do not present its proof.
\begin{lemma}[Basic properties]\label{lemma_basics} Let $p\in (1,\infty)$, $q=\frac{p}{p-1}$ and $Q\subset \R^d$ be open.
Consider sequences $(u\e) \subset \ltp$ and $(v\e) \subset \ltq$.
\begin{enumerate}[(i)]
\item (Boundedness and lower-semicontinuity of the norm). If $u\e \wt u$, then \\ $\sup_{\varepsilon\in (0,1)}\ns{u\e}< \infty$ and $\ns{u}\leq \liminf_{\varepsilon\to 0}\ns{u\e}$.
\item (Compactness of bounded sequences). If $\limsup_{\varepsilon\rightarrow 0}\ns{u\e}<\infty$, then there exists a subsequence $\varepsilon'$ and $u\in \ltp$ such that $u_{\varepsilon'} \wt u$ in $\ltp$.
\item (Characterization of strong two-scale convergence). $u\e \st u$ if and only if $u\e \wt u$ in $L^p(\Omega\times Q)$ and $\ns{u\e}\to \ns{u}$.
\item (Products of strongly and weakly two-scale convergent sequences). If $u\e \wt u$ in $\ltp$ and $v\e \st v$ in $\ltq$, then
\begin{equation*}
\ex{\int_Qu\e(\omega,x) v\e (\omega,x) dx}\rightarrow \ex{\int_Qu(\omega,x)v(\omega,x) dx}.
\end{equation*}
\end{enumerate}
\end{lemma}
\begin{remark}
The stochastic unfolding operator enjoys many similarities to the periodic unfolding operator, however we would like to point out one considerable difference. Namely, in the periodic case if a sequence $(u_{\varepsilon})\subset L^p(Q)$ satisfies $u_{\varepsilon}\to u$ strongly in $L^p(Q)$, it follows that $\unf^p u_{\varepsilon}\to u$ strongly in $L^p(Q\times \Box)$ (see e.g. \cite[Proposition 2.4]{mielke2007two}). In the stochastic case, this does not hold in general, specifically even for a fixed function $u \in L^p(\Omega\times Q)$, in general it does not hold $\unf u \weakto u$. However, if $\ex{\cdot}$ is ergodic, using Proposition \ref{prop1} below, it follows that for a sequence $(u\e) \subset L^{p}(\Omega)\otimes W^{1,p}(Q)$ such that $u_{\varepsilon}\weakto u$ weakly in $L^p(\Omega\times Q)$ it holds that $u_{\varepsilon}\overset{2}{\weakto} \ex{u}$. In this respect, stochastic two-scale convergence might be viewed as an ergodic theorem for weakly convergent sequences.
\end{remark}
\begin{remark}
The choice $\Omega=\Box =[0,1)^d$ (see Example \ref{example:1}), provides us with a tool for periodic homogenization (we might call it \textit{periodic unfolding in the mean}). However, we remark that the convergence notion we obtain using the unfolding operator (in the mean) slightly differs from the notion obtained in standard periodic homogenization results (e.g. using the usual periodic unfolding operator). Namely, we consider a standard convex periodic homogenization problem: For $V: \Box \times \R^d \to \R$, $V$ being convex in its second variable (see Section \ref{Section_Convex} for precise assumptions), for $y\in \Box$, let $u\e(y,\cdot) \in H^1_0(Q)$ be the minimizer of the functional
\begin{equation*}
\mathcal{E}\e^{y}: H^1_0(Q)\to \R, \quad \cE\e^{y}(u)= \int_{Q}V(\tau_{\frac{x}{\varepsilon}}y,\nabla u(x))dx.
\end{equation*}
Let $u\in H^1_0(Q)$ be the minimizer of the homogenized functional $\cE_{\mathsf{hom}}: H^1_0(Q)\to \R$, $\cE_{\mathsf{hom}} (u)= \int_{Q}V_{hom}(\nabla u(x))dx$ (see Section \ref{Section_Convex} for the formula for $V_{\mathsf{hom}}$).
Theorem \ref{thm2} below implies that $\int_{\Box}u\e(y,\cdot)dy \to u(\cdot)$ in $L^2(Q)$, whereas classical periodic results (e.g. \cite{cioranescu2004homogenization}) include the convergence $u\e(y)\to u$ in $L^2(Q)$ for all $y\in \Box$. Note that, in the case of a strongly convex integrand (see Proposition \ref{prop3}), we might recover the convergence $u\e(y)\to u$ in $L^2(Q)$, but merely for a.e. $y\in \Box$ and for a subsequence (which might depend on the choice of the exceptional set in $\Box$). 
\end{remark}
For homogenization of variational problems (in particular, convex integral functionals) the following transformation and (lower semi-)continuity properties are useful.
%
\begin{prop}\label{P_Cont_1}
Let $p\in (1,\infty)$ and $Q\subset \R^d$ be open and bounded. Let $V: \Omega \times Q \times \R^{m}\to \R$ be such that $V(\cdot,\cdot,F)$ is $\mathcal{F} \otimes \mathcal{L}(Q)$-measurable for all $F\in \R^m$ and $V(\omega,x,\cdot)$ is continuous for a.e. $(\omega,x)\in \Omega \times Q$. Also, we assume that there exists $C>0$ such that for a.e. $(\omega,x)\in \Omega\times Q$
\begin{equation*}
|V(\omega, x, F)|\leq C(1+|F|^p), \quad \text{for all }F\in \R^m.
\end{equation*}
\begin{enumerate}[(i)]
\item We have
\begin{equation}\label{intform}
\forall u\in \ltp^{m}\quad \ex{\int_Q V(\tau_{\frac{x}{\varepsilon}}\omega, x,u(\omega,x))dx}=\ex{\int_Q V(\omega, x, \unf u(\omega,x))dx} \,.
\end{equation}
\item 
If $u\e \overset{2s}{\to} u$ in $L^p(\Omega\times Q)^m$, then 
\begin{equation*}
\lim_{\varepsilon\to 0}\ex{\int_{Q}V(\tau_{\frac{x}{\varepsilon}}\omega,x, u\e(\omega,x))dx} = \ex{\int_{Q} V(\omega, x, u(\omega,x))dx}.
\end{equation*} 
\item We additionally assume that for a.e. $(\omega,x)\in \Omega\times Q$, $V(\omega, x,\cdot)$ is convex. Then, if $u\e \overset{2s}{\weakto} u$ in $L^p(\Omega\times Q)^m$,
\begin{equation*}
\liminf_{\varepsilon\to 0}\ex{\int_{Q} V(\tau_{\frac{x}{\varepsilon}} \omega,x,u\e(\omega,x)) dx}\geq \ex{\int_{Q} V(\omega, x, u(\omega,x))dx}.
\end{equation*}
\end{enumerate}
\end{prop}
\textit{(For the proof see Section \ref{S_Proofs}.)}
\begin{remark}[A technical remark about measurability]
The stochastic unfolding operator $\unf$ is defined as a linear operator on the Banach space $L^p(\Omega\times Q)$, which is convenient since this prevents us from (fruitless) discussions on measurability properties. The elements of $L^p(\Omega\times Q)$ are strictly speaking not functions but equivalence classes of functions that coincide a.e.~in $\Omega\times Q$. Thus, a representative function $\tilde u$ in $L^p(\Omega\times Q)$ is measurable w.r.t.~the completion of the product $\sigma$-algebra $\mathcal F\otimes\mathcal L(Q)$, and thus the map  $(\omega,x)\mapsto \tilde u(\tau_{x}\omega,x)$ might not be measurable. However, if $\tilde u$ is $\mathcal F\otimes\mathcal L(Q)$-measurable (e.g. if $\tilde u\in L^p(\Omega)\overset{a}{\otimes}L^p(Q)$), then $\tilde u\e(\omega,x):=\tilde u(\tau_{\frac{x}{\varepsilon}}\omega,x)$ is $\mathcal F\otimes \mathcal L(Q)$-measurable. In particular, since $L^p(\Omega)\overset{a}{\otimes}L^p(Q)$ is dense in $L^p(\Omega\times Q)$, for any $u\in L^p(\Omega\times Q)$ we can find a representative-$\mathcal F\otimes\mathcal L(Q)$ measurable function $\tilde u:\Omega\times Q\to\R$ and we have $\unf u=\tilde u\e$ a.e.~in $\Omega\times Q$.
\end{remark}
\begin{remark}[{Comparison to the notion of} \cite{bourgeat1994stochastic}]\label{R_Stoch_1}
The notion of weak two-scale convergence in the mean of Definition~\ref{def46}, i.e.~the weak convergence of the unfolded sequence, coincides with the convergence introduced in \cite{bourgeat1994stochastic} (see also \cite{andrews1998stochastic}). More precisely, for a bounded sequence $(u_{\varepsilon})\subset L^p(\Omega\times Q)$ we have $u\e\wt u$ in $L^p(\Omega\times Q)$ (in the sense of Definition~\ref{def46}) if and only if $u\e$ stochastically 2-scale converges in the mean to $u$ in the sense of \cite{bourgeat1994stochastic}, i.e. 
\begin{equation}\label{eq:1}
\lim_{\varepsilon\rightarrow 0}\ex{\int_{Q}u_{\varepsilon}(\omega,x)\varphi(\tau_{\frac{x}{\varepsilon}}\omega,x)dx}=\ex{ \int_{Q}u(\omega,x)\varphi(\omega,x)dx},
\end{equation}
for any $\varphi\in L^q(\Omega\times Q)$ that is admissible (in the sense that the transformation $(\omega,x)\mapsto \varphi(\tau_{\frac{x}{\varepsilon}}\omega,x)$ is well-defined). Indeed, with help of $\unf$ (and its adjoint) we might rephrase the integral on the left-hand side in \eqref{eq:1} as
\begin{equation}\label{eq:1234}
\ex{\int_{Q}u_{\varepsilon}(\unf^{*}\varphi)\, dx}=\ex{\int_{Q}(\unf u_{\varepsilon})\varphi dx},
\end{equation}
which proves the equivalence. For the reason of this equivalence, we use the terms \textit{weak} and \textit{strong stochastic two-scale convergence in the mean} instead of talking about \textit{weak} or \textit{strong convergence of unfolded sequences}. 
\end{remark}
\subsection{Two-scale limits of gradients}\label{S_Two}
\def\per{{\sf per}}
As for periodic homogenization via periodic unfolding or two-scale convergence, also in the stochastic case it is important to understand the interplay of the unfolding operator and the gradient operator and to characterize two-scale limits of gradient fields. As a motivation we first recall the periodic case. A standard result states that for any bounded sequence in $W^{1,p}(Q)$ we can extract a subsequence such that $u\e$ weakly converges in $W^{1,p}(Q)$ to a single scale function $u\in W^{1,p}(Q)$ and $\nabla u\e$ weakly two-scale converges to a two-scale limit of the form $\nabla u(x)+\chi(x,y)$, where $\chi$ is a vector field in $L^p(Q)\otimes L^p_{\per}(\Box)^d$ and $L^p_{\per}(\Box)$ denotes the space of locally $p$-integrable, $\Box$-periodic functions on $\R^d$, and $\chi$ is mean-free and curl-free w.r.t. $y\in \Box:=[0,1)^d$. Since $\Box$ is compact, such vector fields can be represented with help of a periodic potential field, i.e.~there exists $\varphi\in L^p(Q,W^{1,p}_{\per}(\Box))$ s.t. $\chi(x,y)=\nabla_y\varphi(x,y)$ for a.e.~$(x,y)$. A helpful example to have in mind is the following $u\e(x):=\varepsilon\varphi(\frac{x}{\varepsilon})\eta(x)$ with $\eta\in W^{1,p}(Q)$ and $\varphi\in C^\infty_{\per}(\Box)$. Then a direct calculation shows that $\nabla u\e(x)=\nabla_y\varphi(\tfrac{x}{\varepsilon})\eta(x)+O(\varepsilon)$, which obviously two-scale converges to $\nabla_y\varphi(y)\eta(x)$.
\medskip

In the stochastic case the torus of the periodic case (which is above represented by $\Box$) is replaced by the probability space $\Omega$ and periodic functions (e.g.~$\varphi$ above) are conceptually replaced by stationary functions, i.e.~functions of the form $S\varphi(\omega,x)=\varphi(\tau_{x}\omega)$ with $\varphi:\Omega\to\R$ measurable. To proceed further, we need to introduce an analogue of the gradient $\nabla_{y}$ and its domain $W^{1,p}_{\per}(\Box)$  in the stochastic setting. As illustrated below, the shift-group $\tau$ together with standard concepts from functional analysis lead to a horizontal gradient $D$ and the space $W^{1,p}(\Omega)$. With help of these objects we prove, as in the  periodic case, that  any bounded sequence in $L^p(\Omega)\otimes W^{1,p}(Q)$ admits (up to extraction of a subsequence) a weak two-scale limit $u$ and the sequence of gradients converges weakly two-scale to a limit of the form $\nabla u+\chi$ where $\chi$ is $D$-curl-free w.r.t.~$\omega$. A difference to the periodic case to be pointed out is that $\chi$ in general does not admit a representation by means of a stationary potential.
\medskip

In order to implement the above philosophy we require some input from functional analysis, which we recall from the original work by Papanicolaou and Varadhan \cite{Papanicolaou1979} (see also \cite{jikov2012homogenization}). We consider the group of isometric operators $\cb{U_x:x\in \re{d}}$ on $L^p(\Omega)$ defined by $U_x \varphi(\omega)=\varphi(\tau_{x}\omega)$. This group is strongly continuous (see \cite[Section 7.1]{jikov2012homogenization}). For $i=1,...,d$, we consider the 1-parameter group of operators $\cb{U_{h e_i}:h\in \re{}}$ ($\cb{e_i}$ being the usual basis of $\re d$) and its infinitesimal generator $D_i:\mathcal{D}_i\subset L^p(\Omega)\rightarrow L^p(\Omega)$
\begin{equation*}
D_i \varphi=\lim_{h\rightarrow 0} \frac{U_{he_i}\varphi-\varphi}{h},
\end{equation*}
which we refer to as \textit{horizontal} derivative. $D_i$ is a linear and closed operator and the associated domain $\mathcal{D}_i$ is dense in $L^p(\Omega)$. We set $W^{1,p}(\Omega)=\cap_{i=1}^{d}\mathcal{D}_i$ and define for $\varphi\in W^{1,p}(\Omega)$ the horizontal gradient as $D\varphi=(D_1 \varphi,...,D_d \varphi)$. In this manner, we obtain  a linear, closed and densely defined operator $D:W^{1,p}(\Omega)\rightarrow L^p(\Omega)^d$, and we denote by
\begin{equation*}
L^p_{\pot}(\Omega):=\overline{\mathcal R(D)}\subset L^p(\Omega)^d
\end{equation*}
the closure of the range of $D$ in $L^p(\Omega)^d$. We denote the adjoint of $D$ by $D^*:\mathcal{D}^*\subset{L^q(\Omega)^d}\rightarrow L^q(\Omega)$ which is a linear, closed and densely defined operator ($\mathcal{D}^*$ is the domain of $D^*$). 
Note that $W^{1,q}(\Omega)^d\subset \mathcal{D}^*$ and for all $\varphi\in W^{1,p}(\Omega)$ and $\psi\in W^{1,q}(\Omega)$ ($i=1,...,d$) we have the integration by parts formula
\begin{equation*}
\ex{D_i \varphi \psi}=-\ex{\varphi D_i \psi},
\end{equation*}
and thus $D^*\psi=-\sum_{i=1}^d D_i \psi_i$ for $\psi\in W^{1,q}(\Omega)^d$. We define the subspace of shift invariant functions in $L^p(\Omega)$ by 
\begin{equation*}
L^p_{{\inv}}(\Omega)=\cb{\varphi\in L^p(\Omega):U_x \varphi=\varphi \quad \text{for all }x \in \re{d}},
\end{equation*}
and denote by $ P_{\inv}:L^p(\Omega) \to L^p_{\inv}(\Omega)$ the conditional expectation with respect to the $\sigma$-algebra of shift invariant sets $\cb{ A \in \mathcal{F} : \tau_x A = A \text{ for all } x\in \re d}$. It is a contractive projection and for $p=2$ it coincides with the orthogonal projection onto $L^2_{\inv}(\Omega)$.
\begin{prop}[Compactness]\label{prop1}
Let $p\in (1,\infty)$ and $Q\subset \R^d$ be open. Let $(u\e)$ be a bounded sequence in $L^p(\Omega)\otimes \sob(Q)$. Then, there exist $u\in L^p_{{\inv}}(\Omega)\otimes \sob(Q)$ and $\chi\in L^p_{\pot}(\Omega)\otimes L^p(Q)$ such that (up to a subsequence)
\begin{equation}\label{equation1}
u\e \wt u  \text{ in }\ltp, \quad \nabla u\e \wt \nabla u +\chi  \text{ in }\ltp^d.
\end{equation}
If, additionally, $\ex{\cdot}$ is ergodic, then $u=P_{\mathsf{{\inv}}} u=\ex{u} \in \sob(Q)$ and $\ex{u_{\varepsilon}}\weakto u$ weakly in $\sob(Q)$.
\end{prop}
We remark that the above result is already established in \cite{bourgeat1994stochastic} in the context of two-scale convergence in the mean in the $L^2$-space setting. We recapitulate its (short) proof from the perspective of stochastic unfolding, see section \ref{S_Proofs}.
\begin{remark}\label{R_Two_1}
Since closed, convex subsets of a Banach space are also weakly closed, for any sequence $(u\e)$ that satisfies the assumption of Proposition \ref{prop1} and $\unf u\e \in X$ where $X\subset L^p(\Omega\times Q)$ is closed and convex, the two-scale limit from Proposition \ref{prop1} satisfies $u\in X$. This is useful to study problems with boundary conditions.
\end{remark}
\begin{lemma}[Nonlinear recovery sequence]\label{Nonlinear_recovery} Let $p \in (1,\infty)$ and $Q\subset \R^d$ be open. For $\chi\in L^p_{\pot}(\Omega)\otimes L^p(Q)$ and $\delta>0$, there exists a sequence $g_{\delta,\varepsilon}(\chi) \in L^p(\Omega)\otimes W^{1,p}_0(Q)$ such that
\begin{equation*}
\|g_{\delta,\varepsilon}(\chi)\|_{L^{p}(\Omega\times Q)} \leq \varepsilon C(\delta), \quad \limsup_{\varepsilon\to 0}\|\mathcal{T}_{\varepsilon}\nabla g_{\delta, \varepsilon}(\chi)-\chi\|_{L^p(\Omega\times Q)^d}\leq \delta.
\end{equation*}
(For the proof see Section \ref{S_Proofs}.)
\end{lemma} 
\begin{prop}[Linear recovery sequence]\label{prop2}
Let $p\in (1,\infty)$ and $Q\subset \R^d$ be open, bounded and $C^1$. For $\varepsilon>0$ there exists a linear operator $\mathcal{G}\e: L^p_{\pot}(\Omega)\otimes L^p(Q) \to L^p(\Omega)\otimes W^{1,p}_0(Q)$, that is uniformly bounded in $\eps$, with the property that for any $\chi\in L^p_{\pot}(\Omega)\otimes L^p(Q)$
\begin{equation*}
\mathcal{G}\e \chi \st 0 \text{ in } \ltp, \quad  \nabla \mathcal{G}\e \chi \st \chi \text{ in }\ltp^d.
\end{equation*}
\end{prop}
\textit{(For the proof see Section \ref{S_Proofs}.)}
\begin{remark}\label{rem14} 
If $Q\subset \R^d$ is open, bounded and $C^1$, using Proposition \ref{prop2}, we obtain a mapping 
$$\brac{L^p_{{\inv}}(\Omega)\otimes \sob(Q)} \times \brac{L^p_{\pot}(\Omega)\otimes L^p(Q)}\ni(u,\chi)\mapsto u_{\varepsilon}(u,\chi):= u+ \mathcal{G}_{\varepsilon} \chi \in L^p(\Omega)\otimes\sob(Q)$$ 
which is linear, uniformly bounded in $\eps$ and it satisfies (for all $(u,\chi)$)
\begin{equation}\label{cor22_eq1}
u_{\varepsilon}(u,\chi) \overset{2s}{\to} u \text{ in }L^p(\Omega\times Q), \quad \nabla u_{\varepsilon}(u,\chi) \overset{2s}{\to} \nabla u + \chi \text{ in }L^p(\Omega\times Q).
\end{equation}
In the case that $Q$ is merely open, we can use the nonlinear construction from Lemma \ref{Nonlinear_recovery}. Specifically, for $(u,\chi)\in \brac{L^p_{{\inv}}(\Omega)\otimes \sob(Q) }\times \brac{L^p_{\pot}(\Omega)\otimes L^p(Q)}$ we define $u_{\delta,\varepsilon}(u,\chi)= u + g_{\delta,\varepsilon}(\chi)$. Using Attouch's diagonal argument, we find a sequence $u_{\varepsilon}(u,\chi)=u_{\delta(\varepsilon),\varepsilon}$ which satisfies (\ref{cor22_eq1}). We remark that in both cases, the recovery sequence $u_{\varepsilon}$ matches the boundary conditions of the function $u$ (see constructions in Section \ref{S_Proofs}).
\end{remark}

We conclude this section with some basic facts from functional analysis used in the proof of Proposition~\ref{prop1}.
\begin{remark}
  Let $p \in (1,\infty)$ and $q= \frac{p}{p-1}$.
  \begin{enumerate}[(i)]
  \item $\ex{\cdot}$ is ergodic $\Leftrightarrow$ $L^p_{\inv}(\Omega)\simeq \re{}$ $\Leftrightarrow$ $\pinv f=\ex{f}$. 
  \item The following orthogonality relations hold (for a proof see \cite[Section 2.6]{brezis2010functional}): Identify the dual space $L^p(\Omega)^*$ with $L^q(\Omega)$, and define for a set $A\subset L^q(\Omega)$ its orthogonal complement $A^{\bot}\subset L^p(\Omega)$ as $A^{\bot}=\cb{\varphi\in L^p(\Omega):\ex{\varphi,\psi}_{L^p,L^q}=0 \text{ for all }\psi\in A}$. Then
    \begin{equation}\label{orth}
      \mathcal N(D)=\mathcal{R}(D^*)^{\bot}, \quad 
      L^p_{\pot}(\Omega)= \overline{\mathcal{R}(D)}=\mathcal N(D^*)^{\bot}.
    \end{equation}
    Above, $\mathcal{N}(\cdot)$ denotes the kernel and $\mathcal{R}(\cdot)$ the range of an operator.
  \end{enumerate}
\end{remark}

\subsection{Proofs}\label{S_Proofs}
\begin{proof}[Proof of Lemma \ref{L:unf}]
  
  We first define $\unf$ on $\sA:=\{\psi(\omega,x)=\varphi(\omega)\eta(x)\,:\,\varphi\in L^p(\Omega),\,\eta\in L^p(Q)\,\}\subset L^p(\Omega\times Q)$ by setting $(\unf\psi)(\omega,x)=(S\varphi)(\omega,-\tfrac{x}{\eps})\eta(x)$ for all $\psi=\varphi\eta\in\sA$. In view of Lemma~\ref{L:stat} $(\unf\psi)$ is $\mathcal F\otimes\mathcal L(Q)$-measurable, and 
  \begin{equation*}
    \ex{\int_Q|\unf\psi|^p\,dx}=\int_Q\big(\int_\Omega|S\varphi(\omega,-\tfrac{x}{\eps})|^p\,dP(\omega)\big)|\eta(x)|^p\,dx=\|\varphi\|_{L^p(\Omega)}^p\|\eta\|_{L^p(Q)}^p=\|\psi\|_{L^p(\Omega\times Q)}^p.
  \end{equation*}
  Since $\mbox{span}(\sA)$ is dense in $L^p(\Omega\times Q)$, $\unf$ extends to a linear isometry from $L^p(\Omega\times Q)$ to $L^p(\Omega\times Q)$. We define a linear isometry $\mathcal{T}_{-\varepsilon}: L^q(\Omega\times Q)\to L^q(\Omega\times Q)$ analogously as $\mathcal{T}_{\varepsilon}$ with $\varepsilon$ replaced by $-\varepsilon$. Then for any $\varphi\in L^p(\Omega)\stackrel{a}{\otimes} L^p(Q)$ and $\psi\in L^q(\Omega)\stackrel{a}{\otimes} L^q(Q)$ we have (thanks to the measure preserving property of $\tau$):
  \begin{eqnarray*}
    \ex{\int_Q(\unf\varphi)\psi\,dx}&=&\int_Q\int_\Omega\varphi(\tau_{-\frac{x}{\eps}}\omega,x)\psi(\omega,x)\,dP(\omega)\,dx\\
  &=&\int_Q\int_\Omega\varphi(\omega,x)\psi(\tau_{\frac{x}{\eps}}\omega,x)\,dP(\omega)\,dx=\ex{\int_Q\varphi(\mathcal T_{-\eps}\psi)}\,dx.
  \end{eqnarray*}
  Since  $L^p(\Omega)\stackrel{a}{\otimes} L^p(Q)$ and $L^q(\Omega)\stackrel{a}{\otimes} L^q(Q)$ are dense in $L^p(\Omega\times Q)$ and $L^q(\Omega\times Q)$, respectively, we conclude that $\unf^*=\mathcal T_{-\eps}$.  

It remains to argue that $\unf$ and $\mathcal{T}_{\varepsilon}^*$ are surjective. Since $\mathcal{T}_{\varepsilon}^*$ is an isometry, it follows that $\unf$ is surjective (see \cite[Theorem 2.20]{brezis2010functional}). Analogously, $\mathcal{T}_{\varepsilon}^*$ is also surjective.
\end{proof}
\begin{proof}[Proof of Proposition \ref{P_Cont_1}]
We first note that $V$ is a \textit{Charath{\'e}odory integrand} (which is defined as a function satisfying the measurability and continuity assumptions given in the statement of the proposition) and therefore it follows that $V$ is $\overline{\mathcal{F}\otimes \mathcal{L}(Q)}\otimes \mathcal{B}(\R^{m})$-measurable (see \cite{rockafellar1971} Proposition 1 and the remarks following it). For fixed $\varepsilon>0$, the mapping $(\omega,x)\mapsto (\tau_{\frac{x}{\varepsilon}}\omega,x)$ is $\mathcal{F}\otimes \mathcal{L}(Q)$-$\mathcal{F}\otimes \mathcal{L}(Q)$-measurable and therefore $(\omega,x,F)\mapsto V(\tau_{\frac{x}{\varepsilon}}\omega,x,F)$ defines as well a Charath{\'e}odory integrand (with same measurability as $V$). As a result of these facts, for any function $u\in L^p(\Omega\times Q)^m$ it follows that $(\omega,x)\mapsto V(\omega,x,u(\omega,x))$ and $(\omega,x)\mapsto V(\tau_{\frac{x}{\varepsilon}}\omega, x, u(\omega,x))$ define measurable functions with respect to the completion of $\mathcal{F}\otimes \mathcal{L}(Q)$. Additionally, these functions are integrable thanks to the growth assumptions on $V$. Thus all the integrals in the statement of the proposition are well-defined.
\medskip

(i) We first argue that it suffices to prove that
\begin{equation}\label{eq1_prop2.8}
\ex{\int_{Q}V(\tau_{\frac{x}{\varepsilon}}\omega,x,u(\omega,x))dx}=\ex{\int_{Q}V(\omega,x,\unf u(\omega,x))dx} \quad \text{for all }u\in L^p(\Omega)\overset{a}{\otimes} L^p(Q)^m.
\end{equation}
Indeed, for any $u\in L^p(\Omega\times Q)^m$ we can find a sequence $u_k \in L^p(\Omega)\overset{a}{\otimes} L^p(Q)^m$ such that $u_k \to u$ strongly in $L^p(\Omega\times Q)^m$, and by passing to a subsequence (not relabeled) we may additionally assume that $u_k\to u$ pointwise a.e.~in $\Omega\times Q$. 
By continuity of $V$ in its last variable, we thus have $V(\tau_{\frac{x}{\varepsilon}}\omega,x, u_{k}(\omega,x)) \to V(\tau_{\frac{x}{\varepsilon}}\omega,x, u(\omega,x))$ for a.e.~$(\omega,x)\in \Omega\times Q$. Since $|V(\tau_{\frac{x}{\varepsilon}}\omega,x, u_{k}(\omega,x))|\leq C(1+|u_{k}(\omega,x)|^p)$ a.e.~in $\Omega\times Q$, the dominated convergence theorem by Vitali implies that $\lim_{k\to \infty}\ex{\int_{Q}V(\tau_{\frac{x}{\varepsilon}}\omega,x,u_{k}(\omega,x))dx} = \ex{\int_{Q}V(\tau_{\frac{x}{\varepsilon}}\omega,x,u(\omega,x))dx}$. In the same way we conclude that 
$$\lim_{k\to \infty}\ex{\int_{Q}V(\omega,x,\unf u_k(\omega,x))dx} = \ex{\int_{Q}V(\omega,x,\unf u(\omega,x))dx}\,,$$ 
and thus (\ref{eq1_prop2.8}) extends to general $u\in L^p(\Omega\times Q)^m$.

It is left to show (\ref{eq1_prop2.8}). Let $u\in L^p(\Omega)\overset{a}{\otimes} L^p(Q)^m$. By Fubini's theorem, the measure preserving property of $\tau$, and the transformation $\omega\mapsto \tau_{-\frac{x}{\varepsilon}}\omega$ in the second equality below, it follows
\begin{equation*}
\ex{\int_{Q}V(\tau_{\frac{x}{\varepsilon}}\omega,x,u(\omega,x))dx}=\int_{Q}\ex{V(\tau_{\frac{x}{\varepsilon}}\omega,x,u(\omega,x))}dx = \int_{Q}\ex{V(\omega,x,u(\tau_{-\frac{x}{\varepsilon}}\omega,x))}dx.
\end{equation*}
Since $u\in L^p(\Omega)\stackrel{a}{\otimes}L^p(Q)$, we have $u(\tau_{-\frac{x}{\varepsilon}}\omega,x)=\unf u(\omega,x)$, and thus the right-hand side equals $\ex{\int_{Q}V(\omega,x,\unf u(\omega,x))dx}$, which completes the proof of (i).
\medskip

(ii) By part (i) we get $\ex{\int_Q V(\tau_{\frac{x}{\varepsilon}}\omega, x,u_{\varepsilon}(\omega,x))dx}=\ex{\int_Q V(\omega, x, \unf u_{\varepsilon}(\omega,x))dx}$. Since $\unf u\e \to u$ strongly in $L^p(\Omega\times Q)^m$ (by assumption), using the growth conditions of $V$ and the dominated convergence theorem, it follows (similarly as in part (i)) that $\lim_{\varepsilon\to 0}\ex{\int_Q V(\omega, x, \unf u_{\varepsilon}(\omega,x))dx}= \ex{\int_{Q}V(\omega,x,u(\omega,x))dx}$.
\medskip

(iii) We note that the functional $L^p(\Omega\times Q)^m\ni u \mapsto \ex{\int_{Q}V(\omega, x,u(\omega,x))dx}$ is convex and lower semi-continuous, therefore it is weakly lower semi-continuous (see \cite[Corollary 3.9]{brezis2010functional}). Combining this fact with the transformation formula from (i) and the weak convergence $\unf u_{\varepsilon}\weakto u$ (by assumption), the claim follows.
\end{proof}
Before stating the proof of Proposition \ref{prop1}, we present some auxiliary lemmas.
\begin{lemma}\label{lemA} Let $p \in (1,\infty)$ and $q=\frac{p}{p-1}$.
\begin{enumerate}[(i)]
\item
If $\varphi\in \cb{D^*\psi:\psi\in W^{1,q}(\Omega)^d}^{\bot}$, then $\varphi\in L^p_{{\inv}}(\Omega)$.
\item
If $\varphi \in \cb{\psi\in W^{1,q}(\Omega)^d: D^*\psi=0}^{\bot}$, then $\varphi\in L^p_{\pot}(\Omega)$.
\end{enumerate}
\end{lemma}
\begin{proof}
(i) First, we note that 
\begin{align*}
\varphi \in L^p_{\inv}(\Omega) \quad \Leftrightarrow \quad U_{h e_i}U_{y}\varphi=U_y\varphi  \quad \text{for all }y\in \R^d,h\in \R, i=1,...,d.
\end{align*} 
We consider $\p\in \cb{D^*\psi:\psi\in W^{1,q}(\Omega)^d}^{\bot}$ and we show that $\p\in L^p_{\inv}(\Omega)$ using the above equivalence. 
Let $\psi \in W^{1,q}(\Omega)$ and $i\in \cb{1,...,d}$. Then, by the group property we have $U_{-h e_i}\psi-\psi=\int_{0}^h U_{-t e_i}D_i^*\psi dt$ and therefore
\begin{align*}
\ex{(U_{h e_i}\varphi-\varphi)\psi}=\ex{\varphi (U_{-he_i}\psi-\psi)}=\ex{\varphi \int_{0}^h U_{-t e_i}D_i^*\psi dt}=\int_{0}^h\ex{\varphi D^*_i(U_{-t e_i}\psi)}dt.
\end{align*}
Since $U_{-t e_i}\psi \in W^{1,q}(\Omega)$ for any $t\in [0,h]$, we obtain $\ex{\varphi D^*_i(U_{-t e_i}\psi)}=0$ and thus $U_{h e_i}\varphi = \varphi$. Furthermore, for any $y\in \R^d$, we have $\ex{(U_{h e_i}U_y \varphi - U_y \varphi)\psi}=\ex{(U_{h e_i}\varphi -\varphi)U_{-y}\psi}=0$ by the same argument.
\medskip

(ii) In view of $L^p_{\pot}(\Omega)=\mathcal{N}(D^*)^{\bot}$ (see (\ref{orth})), it is sufficient to prove that $\cb{\p\in W^{1,q}(\Omega)^d: D^*\p=0}$ is dense in $\mathcal{N} (D^*)$. This follows by an approximation argument as in \cite{jikov2012homogenization}, Section 7.2.
Let $\p \in \mathcal{N}(D^*)$ and we define for $t>0$
\begin{equation*}
\p^t(\omega)=\int_{\re{d}}p_t(y)\p(\tau_y \omega)dy, \quad \text{where }p_t(y)=\frac{1}{\brac{4\pi t}^{\frac{d}{2}}}e^{-\frac{|y|^2}{4t}}.
\end{equation*}
Then the claimed density follows, since $\p^t\in W^{1,q}(\Omega)^d$, $D^*\p^t=0$ for any $t>0$ and $\p^t \rightarrow \p$ strongly in $L^q(\Omega)^d$. The last statement can be seen as follows. By the continuity property of $U_x$, for any $\varepsilon>0$ there exists $\delta>0$ such that $\ex{|\p(\tau_y \omega)-\p(\omega)|^q}\leq \varepsilon$ for any $y\in B_{\delta}(0)$.
It follows that
\begin{align*}
\ex{|\p^t-\p|^q} & =\ex{\bigg|\int_{\re{d}}p_t(y)\brac{\p(\tau_y \omega)-\p(\omega)}dy\bigg|^q}\\  & \leq \int_{\re{d}}p_t(y)\ex{|\p(\tau_y \omega)-\p(\omega)|^q}dy \\ & =   \int_{B_{\delta}}p_t(y)\ex{|\p(\tau_y \omega)-\p(\omega)|^q}dy+\int_{\re{d}\setminus B_{\delta}}p_t(y)\ex{|\p(\tau_y \omega)-\p(\omega)|^q}dy.
\end{align*}
The first term on the right-hand side of the above inequality is bounded by $\varepsilon$ as well as the second term for sufficiently small $t>0$.
\end{proof}
\begin{lemma}\label{lem6}
Let $u\e \in L^p(\Omega)\otimes \sob(Q)$ be such that $u\e \wt u$ in $\ltp$ and $\varepsilon \nabla u\e \wt 0$ in $\ltp^d$. Then $u\in L^p_{{\inv}}(\Omega)\otimes L^p(Q)$.
\end{lemma}
\begin{proof}
Consider a sequence $v\e=\varepsilon \mathcal{T}_{\varepsilon}^*(\varphi \eta)$ such that $\varphi\in W^{1,q}(\Omega)$ and $\eta\in C^{\infty}_c(Q)$. Note that $\mathcal{T}\e v\e=\varepsilon \varphi \eta$ and we have ($i=1,...,d$)
\begin{equation*}
\ex{\int_Q \partial_i u\e v\e dx}=\ex{\int_Q\mathcal{T}\e \partial_i u\e \mathcal{T}\e v\e dx}=\ex{\int_Q\mathcal{T}\e \partial_i u\e \varepsilon \varphi \eta dx}\rightarrow 0.
\end{equation*}
Moreover, it holds that $\partial_{i}v_{\varepsilon}= \mathcal{T}_{\varepsilon}^*(D_i \varphi \eta + \varepsilon \varphi \partial_i \eta)$ and therefore
\begin{align*}
\ex{\int_Q\partial_i u\e v\e dx}  &=-\ex{\int_Q u\e \partial_i v\e dx} =-\ex{\int_Q u\e \mathcal{T}_{\varepsilon}^*(D_i \varphi \eta + \varepsilon \varphi \partial_i \eta)dx}\\ &=
-\ex{\int_Q \mathcal{T}\e u\e D_i\varphi \eta+\varepsilon \mathcal{T}\e u\e \varphi \partial_i\eta dx}. 
\end{align*}
The last expression converges to $-\ex{\int_Q u D_i\varphi \eta dx}$ as $\varepsilon\to 0$.
As a result of this, $\ex{u(x)D_i\varphi}=0$ for almost every $x\in Q$ and therefore $u\in L^p_{{\inv}}(\Omega)\otimes L^p(Q)$ by Lemma \ref{lemA} (i).
\end{proof}
\begin{lemma}\label{lem7}
Let $u\e$ be a bounded sequence in $L^p(\Omega)\otimes \sob(Q)$. Then there exists $u\in L^p_{{\inv}}(\Omega)\otimes \sob(Q)$ such that
\begin{equation*}
u\e \wt u \text{ in }\ltp, \quad \pinv u\e\wt u \text{ in }\ltp,\quad \pinv \nabla u\e \wt \nabla u \text{ in }\ltp^d.
\end{equation*}
\end{lemma}
\begin{proof}
\textit{Step 1. $\pinv \circ \mathcal{T}\e=\mathcal{T}\e\circ \pinv =\pinv$.}

The second equality holds clearly. To show that  $\pinv \circ \mathcal{T}\e=\pinv$, we consider $v\in \ltp$, $\varphi\in L^q(\Omega)$ and $\eta\in L^q(Q)$. We have
\begin{align*}
\ex{\int_Q (\pinv\unf v) (\p \eta) dx}& =\ex{\int_Q (\unf v) \pinv^* (\p \eta) dx}\\ &=\ex{\int_Q v \pinv^* (\p \eta) dx}=\ex{\int_Q  (\pinv v) (\p \eta) dx},
\end{align*}
where we use the fact that $\unf^* P_{\inv}^*= P_{\inv}^*$ since the adjoint $P_{\inv}^*$ of $P_{\inv}$ satisfies $\mathcal{R}(P_{\inv}^*)\subset L^q_{\inv}(\Omega)$. The claim follows by an approximation argument since $L^q(\Omega)\overset{a}{\otimes}L^q(Q)$ is dense in $L^q(\Omega\times Q)$.
\smallskip

\textit{Step 2. Convergence of $\pinv u\e$.}

$\pinv$ is bounded and it commutes with $\nabla$ and therefore
\begin{equation*}
\limsup_{\varepsilon\to 0} \ex{\int_Q |\pinv u\e|^p+|\nabla \pinv u\e|^p dx}< \infty.
\end{equation*}
As a result of this and with help of Lemma \ref{lemma_basics} (ii) and Lemma \ref{lem6}, it follows that $\pinv u\e\overset{2s}{\weakto} v$ and $\nabla \pinv u\e\overset{2s}{\weakto} w$ (up to a subsequence), where $v\in L^p_{{\inv}}(\Omega)\otimes L^p(Q)$ and $w\in L^p_{{\inv}}(\Omega)\otimes L^p(Q)^d$.

Let $\p\in W^{1,q}(\Omega)$ and $\eta\in C^{\infty}_c(Q)$.
On the one hand, we have
\begin{equation*}
\ex{\int_Q (\partial_i \pinv u\e) \mathcal{T}_{\varepsilon}^*(\varphi \eta) dx}=\ex{\int_Q\unf (\partial_i \pinv u\e) (\p\eta) dx}\rightarrow \ex{\int_Q w_i \p\eta dx}.
\end{equation*}
On the other hand, 
\begin{equation*}
\ex{\int_Q (\partial_i \pinv u\e) \mathcal{T}_{\varepsilon}^*(\varphi \eta) dx}=-\frac{1}{\varepsilon}\ex{\int_Q (\pinv u\e) (D_i\p\eta) dx}-\ex{\int_Q(\pinv u\e) (\p \partial_i\eta) dx}.
\end{equation*}
The first term on the right-hand side vanishes since $\pinv u\e(\cdot,x)\in L^p_{{\inv}}(\Omega)$ for almost every $x\in Q$ and by (\ref{orth}). The second term converges to $-\ex{\int_Q v \p \partial_i \eta dx}$ as $\varepsilon\rightarrow 0$. Consequently, we obtain $w=\nabla v$ and therefore $v\in L^p_{{\inv}}(\Omega)\otimes \sob(Q)$.
\smallskip

\textit{Step 3. Convergence of $u\e$.}

Since $u\e$ is bounded, by  Lemma \ref{lemma_basics} (ii) and Lemma \ref{lem6} there exists $u\in L^p_{{\inv}}(\Omega)\otimes L^p(Q)$ such that $u\e \wt u$ in $\ltp$. Also, $\pinv$ is a linear and bounded operator which, together with Step 1, implies that $\pinv u\e\weakto u$. Using this, we conclude that $u=v$.
\end{proof}
\begin{proof}[Proof of Proposition \ref{prop1}]
Lemma \ref{lem7} implies that $u\e \wt u$ in $\ltp$ (up to a subsequence), where $u\in L^p_{{\inv}}(\Omega)\otimes \sob(Q)$. Moreover, it follows that there exists $v\in \ltp^d$ such that $\nabla u\e \wt v$ in $\ltp^d$ (up to another subsequence). We show that $\chi:=v-\nabla u\in L^p_{\pot}(\Omega)\otimes L^p(Q)$.

Let $\varphi \in W^{1,q}(\Omega)^d$ with $D^*\varphi=0$ and $\eta\in C^{\infty}_c(Q)$. We have
\begin{equation}\label{eq4321}
\ex{\int_Q \nabla u\e \cdot \mathcal{T}_{\varepsilon}^*(\varphi \eta) dx} = \ex{\int_Q \unf \nabla u\e \cdot \varphi \eta dx} \rightarrow \ex{\int_{Q}v \cdot \varphi \eta dx}.
\end{equation}
On the other hand,
\begin{align}\label{eq1234}
\begin{split}
\ex{\int_Q \nabla u\e \cdot \mathcal{T}_{\varepsilon}^*(\varphi \eta) dx}  &=-\ex{\int_Q u\e \sum_{i=1}^d \mathcal{T}_{\varepsilon}^*(\frac{1}{\varepsilon} D_i \varphi \eta+\varphi_i \partial_i\eta) dx}  
\\ &= \frac{1}{\varepsilon} \ex{\int_Q (\unf u\e) (D^*\varphi \eta) dx}-\ex{\int_Q (\unf u\e) \sum_{i=1}^d \varphi_i\partial_i \eta dx}.
\end{split}
\end{align}
Above, the first term on the right-hand side vanishes by assumption and the second converges to $ \ex{\int_Q\nabla u\cdot \varphi \eta}$ as $\varepsilon\rightarrow 0$. Using
(\ref{eq1234}), (\ref{eq4321}) and Lemma \ref{lemA} (ii) we  complete the proof.
\end{proof}
\begin{proof}[Proof of Lemma \ref{Nonlinear_recovery}]
For $\chi \in L^p_{\pot}(\Omega)\otimes L^p(Q)$ and $\delta>0$, by definition of the space $L^p_{\pot}(\Omega)\otimes L^p(Q)$ and by density of $\mathcal{R}(D)$ in $L^p_{\pot}(\Omega)$, we find $g_{\delta}=\sum_{i=1}^{n(\delta)}\p^{\delta}_i \eta^{\delta}_i$ with $\p_i^{\delta} \in W^{1,p}(\Omega)$ and $\eta^{\delta}_i\in C^{\infty}_c(Q)$ such that
\begin{equation*}
\|\chi - Dg_{\delta} \|_{L^p(\Omega\times Q)^d} \leq \delta.
\end{equation*}
We define $g_{\delta,\varepsilon}= \varepsilon \unf^{-1} g_{\delta}$ and note that $g_{\delta,\varepsilon} \in L^p(\Omega)\otimes W_0^{1,p}(Q)$ and $\nabla g_{\delta,\varepsilon}=\unf^{-1}D g_{\delta}+\unf^{-1}\varepsilon\nabla g_{\delta}$. As a result of this and with help of the isometry property of $\unf^{-1}$, the claim of the lemma follows.
\end{proof}
\begin{proof}[Proof of Proposition \ref{prop2}]
For $\chi \in L^p_{\pot}(\Omega)\otimes L^p(Q)$ we define $\mathcal{G}\e \chi=v\e$ as the unique weak solution in $W^{1,p}_0(Q)$ to the equation (for $P$-a.e. $\omega\in \Omega$)
\begin{equation}\label{eq99}
-\Delta v\e(\omega)=- \nabla \cdot (\unf^{-1}\chi(\omega)).
\end{equation}
Above and further in this proof, we use the notation $u(\omega):= u(\cdot,\omega)\in L^p(Q)$ for functions $u\in L^p(\Omega\times Q)$.

By Poincar{\'e}'s inequality and the Calder{\'o}n-Zygmund estimate, we obtain
\begin{equation*}
\| v\e(\omega) \|_{L^p(Q)} \leq C \|\nabla v\e(\omega) \|_{L^p(Q)^d} \leq C \| \unf^{-1}\chi(\omega) \|_{L^p(Q)^d},
\end{equation*} 
and therefore 
\begin{equation*}
\| v\e \|_{L^p(\Omega \times Q)} \leq C \|\nabla v\e \|_{L^p(\Omega \times Q)^d} \leq C \| \chi \|_{L^p(\Omega\times Q)^d}.
\end{equation*}
Using Lemma \ref{Nonlinear_recovery}, we find a sequence $g_{\delta,\varepsilon}\in L^p(\Omega)\otimes W^{1,p}_0(Q)$ such that
\begin{equation*}
\|g_{\delta,\varepsilon}(\chi)\|_{L^p(\Omega\times Q)} \leq \varepsilon C(\delta), \quad \limsup_{\varepsilon\to 0}\|\mathcal{T}_{\varepsilon}\nabla g_{\delta, \varepsilon}(\chi)-\chi\|_{L^p(\Omega\times Q)^d}\leq \delta.
\end{equation*}
 Note that $v\e(\omega)-g_{\delta,\varepsilon}(\omega)\in W^{1,p}_0(Q)$ (for $P$-a.e. $\omega\in \Omega$) and it is the unique weak  solution to
\begin{equation*}
-\Delta(v\e(\omega) - g_{\delta,\varepsilon}(\omega))=-\nabla \cdot (\unf^{-1}\chi(\omega)-\nabla g_{\delta,\varepsilon}(\omega)).
\end{equation*}
As before, we have
\begin{equation}\label{eq98}
\| v\e- g_{\delta,\varepsilon}\|_{L^p(\Omega \times Q)}\leq C \|\nabla v\e- \nabla g_{\delta,\varepsilon} \|_{L^p(\Omega \times Q)^d} \leq C \| \chi -\unf \nabla g_{\delta,\varepsilon} \|_{L^p(\Omega\times Q)^d}.
\end{equation}
Therefore, using the isometry property of $\unf$, we obtain
\begin{align*}
\|\unf \nabla v\e- \chi \|_{L^p(\Omega \times Q)^d} & \leq \|\nabla v\e- \nabla g_{\delta,\varepsilon} \|_{L^p(\Omega \times Q)^d}+\|\unf \nabla g_{\delta,\varepsilon} - \chi \|_{L^p(\Omega \times Q)^d} \\ & \leq C \| \chi -\unf \nabla g_{\delta,\varepsilon} \|_{L^p(\Omega\times Q)^d}.
\end{align*}
Consequently, first letting $\varepsilon\rightarrow 0$ and then $\delta\rightarrow 0$ we obtain that $\nabla v\e \overset{2s}{\rightarrow} \chi$ in $\ltp^d$. Furthermore, using (\ref{eq98}) we obtain that $v\e \overset{2s}{\rightarrow }0$ in $\ltp$ which completes the proof.
\end{proof}
\section{Applications to  homogenization in the mean}\label{Section_Applications}
In this section we apply the stochastic unfolding method to homogenization problems. We discuss the classical homogenization problem of convex integral functionals and derive a homogenization result for an evolutionary gradient system. We refer to \cite{neukamm2017stochastic} where a similar analysis has been conducted in a discrete-to-continuum setting for convex integral functionals and for an evolutionary rate-independent system.
\medskip

The treatment of integral functionals is a well-known topic in stochastic homogenization and previous results typically rely on the subadditive ergodic theorem (see e.g. \cite{DalMaso1986,neukamm_schaeffner}) or on the notion of quenched stochastic two-scale convergence (see \cite{HeidaNesenenko2017monotone} and Section 4). The analysis via unfolding is less involved than these methods since it merely relies on lower semi-continuity of convex functionals and weak compactness properties of ``unfolded'' sequences in $L^p(\Omega\times Q)$. On the other hand, the method we present yields weaker results than other procedures, namely convergence for solutions is obtained in a statistically averaged sense (see Theorem \ref{thm2}), whereas the analysis based on the subadditive ergodic theorem (e.g. \cite{neukamm_schaeffner}) yields convergence for every typical realization of the medium and it even allows to consider non-convex functionals. We refer to a recent study \cite{Berlyand2017} for an investigation of homogenization of non-convex integral functionals by a two-scale $\Gamma$-convergence approach.
\medskip

The second part of this section is dedicated to the analysis of an evolutionary problem, a gradient system which corresponds to an Allen-Cahn type equation. A significant number of mathematical models can be phrased in the setting of evolutionary gradient systems which are formulated variationally, with the help of an energy and a dissipation functional (see Section \ref{Section_3.2} for a specific example). We refer to \cite{ambrosio2008gradient,savare2007gradient,mielke2013nonsmooth} for the abstract theory of gradient systems. Typically, the asymptotic analysis of sequences of gradient systems (so called evolutionary $\Gamma$-convergence \cite{mielke2016evolutionary}) relies merely on $\Gamma$-convergence properties of the underlying two functionals. For various general strategies for such problems we refer to \cite{attouch1984variational,sandier2004gamma,daneri2010lecture,mielke2013nonsmooth,mielke2016evolutionary}. 
In \cite{liero2015homogenization} a gradient system driven by a non-convex (Cahn-Hilliard type) energy is considered and a periodic homogenization result is established using periodic unfolding. In this study, we consider a related random model and derive a homogenization result based on the stochastic unfolding procedure (see Section \ref{Section_3.2}). We refer to \cite{hornung1994reactive,mielke2014two} for other related periodic homogenization results, where reaction-diffusion equations with periodic coefficients are considered.

In Section \ref{Section:3:3} we argue on the level of convex functionals that \textit{quenched homogenization} and \textit{homogenization in the mean} (via stochastic unfolding) typically lead to the same limiting equation. We therefore view stochastic unfolding as a useful tool to identifiy homogenized limit equations. 
       
\subsection{Convex integral functionals}\label{Section_Convex}
Let $p\in (1,\infty)$ and $Q\subset \R^d$ be open and bounded. We consider $V:\Omega\times Q\times \re{d\times d}\rightarrow \re{}$ and the following set of  assumptions.
\begin{itemize}
\item[(A1)] $V(\cdot,\cdot, F)$ is $\mathcal{F}\otimes \mathcal{L}(Q)$-measurable for all $F\in \R^{d\times d}$.
\item[(A2)] $V(\omega, x, \cdot)$ is convex for a.e. $(\omega,x)\in \Omega\times Q$.
\item[(A3)]\label{grow_cond} There exists a $C>0$ such that
\begin{equation*}
\frac{1}{C}|F|^p-C\leq V(\omega, x, F) \leq C(|F|^p+1)
\end{equation*}
for a.e. $(\omega,x) \in \Omega\times Q$ and all $F\in \re{d\times d}$.
\item[(A4)] For a.e. $(\omega,x) \in \Omega\times Q$, $V(\omega, x,\cdot)$ is uniformly convex with modulus $(\cdot)^p$, i.e. there exists $C>0$ (independent of $\omega$ and $x$) such that for all $F, G\in \re{d\times d}$ and $t\in [0,1]$
\begin{align*}
V(\omega, x,tF+(1-t)G)\leq t V(\omega, x,F)+(1-t) V(\omega, x,G)-(1-t)tC|F-G|^p.
\end{align*} 
\end{itemize}
Below we use the shorthand notation $\nabla^s u=\frac{1}{2}\brac{\nabla u+\nabla u ^{T}}$ and $\chi^{s}=\frac{1}{2}\brac{\chi+\chi^{T}}$. We consider problems with homogeneous Dirichlet boundary conditions and energy functional 
\begin{equation}\label{energy}
 \mathcal{E}_{\varepsilon}:L^p(\Omega)\otimes \sob_0(Q)^d\rightarrow \re{},\quad
 \mathcal{E}_{\varepsilon}(u)=\ex{\int_Q V(\tau_{\frac{x}{\varepsilon}}\omega, x,\nabla^s u(\omega,x))dx}.
\end{equation}

Under the assumptions $(A1)-(A3)$, in the limit $\varepsilon\rightarrow 0$ we obtain the following functional
\begin{align}\label{energy_hom}
\begin{split}
& \mathcal{E}_0:\brac{L^p_{{\inv}}(\Omega)\otimes \sob_0(Q)^d} \times \brac{L^p_{\pot}(\Omega)\otimes L^p(Q)^d},\\
& \mathcal{E}_0(u,\chi)=\ex{\int_Q V(\omega, x, \nabla^s u(\omega,x)+ \chi^{s}(\omega,x)) dx}.
\end{split}
\end{align}
\begin{thm}[Two-scale homogenization]\label{thm1}
Let $p\in (1,\infty)$ and $Q\subset \R^d$ be open and bounded. Assume $(A1)-(A3)$.
\begin{itemize}
\item[(i)](Compactness) Let $u\e \in L^p(\Omega)\otimes \sob_0(Q)^d$ be such that
$\limsup_{\varepsilon\rightarrow 0}\mathcal{E}_{\varepsilon}(u\e)<\infty$. 
There exist $(u,\chi) \in \brac{L^p_{{\inv}}(\Omega)\otimes \sob_0(Q)^d} \times \brac{L^p_{\pot}(\Omega)\otimes L^p(Q)^d}$ and a subsequence (not relabeled) such that
\begin{equation}\label{convergence}
u\e \wt u \text{ in }\ltp^d, \quad \nabla u\e \wt \nabla u+\chi  \text{ in }\ltp^{d\times d}.
\end{equation}
\item[(ii)](Liminf inequality) If the above convergence holds for the whole sequence, then
\begin{equation*}
\liminf_{\varepsilon\rightarrow 0}\mathcal{E}_{\varepsilon}(u\e)\geq \mathcal{E}_0(u,\chi).
\end{equation*}
\item[(iii)](Limsup inequality) Let $(u,\chi)\in \brac{L^p_{{\inv}}(\Omega)\otimes \sob_0(Q)^d} \times \brac{L^p_{\pot}(\Omega)\otimes L^p(Q)^d}$. There exists a sequence $u\e \in L^p(\Omega)\otimes \sob_0(Q)^d$ such that
\begin{equation*}
u\e \st u \text{ in }\ltp^d, \quad  \nabla u\e \st \nabla u+\chi  \text{ in }\ltp^{d\times d}, \quad
\lim_{\varepsilon\rightarrow 0}\mathcal{E}_{\varepsilon}(u\e)=\mathcal{E}_0(u,\chi).
\end{equation*}
\end{itemize}
\end{thm}
\textit{(For the proof see Section \ref{S_Proof_3}.)}
\begin{corollary}\label{C:thm1}
Assume the same assumptions as in Theorem \ref{thm1}.  Let $u\e\in L^p(\Omega)\otimes W^{1,p}_0(Q)^d$ be a minimizer of $\cE\e$. Then there exists a subsequence (not relabeled),  $u\in L^p_{\inv}(\Omega)\times W^{1,p}_0(Q)^d$, and $\chi\in L^p_{\pot}(\Omega)\otimes L^p(Q)^d$ such that $u\e \wt u \text{ in }\ltp^d$, $\nabla u\e \wt \nabla u+\chi  \text{ in }\ltp^{d\times d}$, and 
  \begin{equation*}
    \lim\limits_{\eps\to 0}\min\cE\e=    \lim\limits_{\eps\to 0}\cE\e(u\e)=\cE_0(u,\chi)=\min\cE_0.
  \end{equation*}
\end{corollary}
\textit{(For the proof see Section \ref{S_Proof_3}.)}
\begin{remark}
If $V(\omega, x,\cdot)$ is strictly convex the minimizers are unique and the convergence in the above corollary holds for the entire sequence.
\end{remark}
\begin{remark}
We might consider the perturbed energy functional $\mathcal I_{\varepsilon}(\cdot)=\cE_{\varepsilon}(\cdot)+\ex{l_{\varepsilon},\cdot}_{(L^p)^*,L^p}$ with $l_{\varepsilon} \overset{2}{\to} l$ in $L^{q}(\Omega\times Q)$. As in Corollary \ref{C:thm1}, minimizers of $\mathcal I_{\varepsilon}$ converge in the above two-scale sense (up to a subsequence) to minimizers of $(u,\chi)\mapsto \mathcal I_0(u,\chi):=\cE_{0}(u,\chi)+\ex{P_{\inv}l,u}_{(L^p)^*,L^p}$. 
\end{remark}
\smallskip

If we additionally assume that $\ex{\cdot}$ is ergodic, the limit functional reduces to a single-scale energy
\begin{equation*}
 \mathcal{E}_{\hom}:\sob_0(Q)^d \rightarrow \re{}, \quad
 \mathcal{E}_{\hom}(u)=\int_Q V_{\hom}(x,\nabla u(x))dx,
\end{equation*}
where the homogenized integrand $V_{\hom}$ is given for $x\in \R^d$ and $F\in \R^{d\times d}$ by 
\begin{align}\label{equation}
V_{\hom}(x,F)=\inf_{\chi\in L^p_{\pot}(\Omega)^d}\ex{V(\omega, x,F^s+\chi^s(\omega))}.
\end{align}
\begin{thm}[Ergodic case]\label{thm2}
Assume the same assumptions as in Theorem \ref{thm1}. Moreover, we assume that $\ex{\cdot}$ is ergodic.
\begin{itemize}
\item[(i)] Let $u\e \in L^p(\Omega)\otimes \sob_0(Q)^d$ be such that $\limsup_{\varepsilon\rightarrow 0}\mathcal{E}_{\varepsilon}(u\e)<\infty$. There exist $u \in  \sob_0(Q)^d $ and a subsequence (not relabeled) such that
\begin{gather}
u\e \wt u \text{ in }\ltp^d , \quad \ex{u\e}\rightarrow u \text{ strongly in } L^p(Q)^d, \nonumber \\ \ex{\nabla u\e} \weakto \nabla u \text{ weakly in } L^p(Q)^{d\times d}.
\end{gather}
Moreover,
\begin{align*}
\liminf_{\varepsilon\rightarrow 0}\mathcal{E}_{\varepsilon}(u\e)\geq \mathcal{E}_{\hom}(u).
\end{align*}
\item[(ii)] Let $u\in\sob_0(Q)^d$. There exists a sequence $u\e \in L^p(\Omega)\otimes \sob_0(Q)^d$ such that
\begin{align*}
u\e \st u  \text{ in } \ltp^d, \quad \ex{\nabla u\e} \to \nabla u \text{ strongly in }L^p(Q)^{d\times d}, 
\quad \lim_{\varepsilon\rightarrow 0}\mathcal{E}_{\varepsilon}(u\e)= \mathcal{E}_{\hom}(u).
\end{align*}
\end{itemize}
\end{thm}
\textit{(For the proof see Section \ref{S_Proof_3}.)}

We consider problems with an additional strong convexity assumption and consequently obtain that the whole sequence of unique minimizers of $\cE_{\varepsilon}$ converges strongly in the usual strong topology of $L^p(\Omega\times Q)$ to the unique minimizer of $\cE_{\hom}$:
\begin{prop}\label{prop3} Assume the same assumptions as in Theorem \ref{thm2}. Assume additionally $(A4)$. $\cE_{\varepsilon}$ and $\cE_{\hom}$ admit unique minimizers $u\e \in L^p(\Omega)\otimes W^{1,p}_0(Q)^d$ and $u\in W^{1,p}_0(Q)$, respectively. We have
\begin{equation*}
u\e \to u \text{ in }\ltp^d, \quad \ex{\nabla u_{\varepsilon}}{\weakto} \nabla u \text{ weakly in }L^p(Q)^{d\times d}.
\end{equation*}
\end{prop}
\textit{(For the proof see Section \ref{S_Proof_3}.)}
\subsection{Allen-Cahn type gradient flows}\label{Section_3.2}
In this section we provide a homogenization result for an evolutionary gradient system. Let $Q\subset \R^d$ be open and bounded. The system is defined on a state space $\sB := L^2(\Omega \times Q)$ and with the help of two functionals - a dissipation potential $\mathcal{R}_{\varepsilon}$ and an energy functional $\cE_{\varepsilon}$.
The dissipation potential $\mathcal{R}_{\varepsilon}: \sB\to [0,\infty)$ is given by 
\begin{equation*}
\mathcal{R}_{\varepsilon}(v)=\frac{1}{2}\ex{\int_{Q}r(\tau_{\frac{x}{\varepsilon}}\omega)|v(\omega,x)|^2 dx},
\end{equation*}
and the energy functional $\cE_{\varepsilon}:\sB\to \R\cup \cb{\infty}$ is defined as follows: For $u \in L^2(\Omega)\otimes H^1(Q)\cap L^p(\Omega \times Q)=:dom(\mathcal{E}\e)$ (where $p>2$ is fixed throughout this section),
\begin{equation*}
\cE_{\varepsilon}(u)=\ex{\int_{Q}A(\tau_{\frac{x}{\varepsilon}}\omega)\nabla u(\omega,x)\cdot \nabla u(\omega,x)+f(\tau_{\frac{x}{\varepsilon}}\omega,u(\omega,x))dx},
\end{equation*}
and $\cE_{\varepsilon} = \infty$ otherwise. Our assumptions on $r:\Omega \to \R_{+}$, $A: \Omega \to \mathbb{R}^{d\times d}_{sym}$ and $f:\Omega\times \R \to \R$ are given as follows:
\begin{itemize}
\item[(B1)] $r\in L^{\infty}(\Omega)$, $A\in L^{\infty}(\Omega)^{d\times d}$ and there exists $C>0$ such that for $P$-a.e. $\omega\in \Omega$ it holds that $\frac{1}{C}\leq r(\omega)\leq C $ and $A(\omega)F\cdot F \geq \frac{1}{C} |F|^2$ for all $F\in \R^d$.
\item[(B2)] $f(\cdot,y)$ is measurable for all $y\in \R$ and $f(\omega,\cdot)$ is continuous for $P$-a.e. $\omega\in \Omega$. There exists $\lambda \in \R$ such that for $P$-a.e. $\omega \in \Omega$ 
\begin{align*}
& f(\omega, \cdot) \text{ is }\lambda\text{-convex, i.e. } \quad y \mapsto f(\omega, y)-\frac{\lambda}{2} |y|^2 \text{ is convex},\\
&\frac{1}{C}|y|^p - C \leq f(\omega,y)\leq C(|y|^p+1) \quad \text{for all }y\in \mathbb{R}.
\end{align*}
\end{itemize}
We remark that the above assumptions imply that $u\mapsto \cE_{\varepsilon}(u)-\Lambda \mathcal{R}_{\varepsilon}(u)$ is convex, where $\Lambda:=\frac{\lambda}{C}$. 
Let $T>0$ and we consider the following differential inclusion
\begin{equation}\label{diff_incl}
0\in D\mathcal{R}_{\varepsilon}(\dot{u}(t))+\partial_{F}\mathcal{E}_{\varepsilon}(u(t)) \quad \text{for a.e. } t\in (0,T), \quad u(0)=u_{0,\varepsilon},
\end{equation}
where $\partial_F \mathcal{E}\e: \sB \to 2^{\sB^*}$ is the Frech{\'e}t subdifferential of $\mathcal{E}\e$ given by 
\begin{equation*}
\partial_F \mathcal{E}\e(u)=\cb{\xi \in \sB^{*}: \liminf_{w\to u}\frac{\mathcal{E}\e(w)-\mathcal{E}\e(u)-\ex{\xi,w-u}_{\sB,\sB^*}}{\|w-u\|_{\sB}}\geq 0}.
\end{equation*}
Since $\cE\e(\cdot)-\Lambda\mathcal{R}\e(\cdot)$ is convex, using standard convex analysis arguments (see Proposition 1.2 and Corollary 1.12.2 in \cite{kruger2003frechet}), it follows that  
\begin{equation*}
\partial_F \mathcal{E}\e(u)= \cb{\xi \in \sB^*: \cE\e(u)\leq \cE\e(w)+\ex{\xi,u-w}_{\sB^*,\sB}- \Lambda \mathcal{R}\e(u-w) \text{ for all }w\in \sB}.
\end{equation*}
We refer to \cite{mielke2016evolutionary} for various other formulations of the differential inclusion (\ref{diff_incl}). If we assume $(B1)-(B2)$ and $u_{0,\varepsilon}\in dom(\cE\e)$, then (\ref{diff_incl}) admits a unique solution $u\e \in H^1(0,T;\sB)$ (see e.g., \cite[Theorem 3.2]{clement2009introduction}).
\medskip

As $\varepsilon\to 0$, we derive a limit gradient system which is described in the following. The state space for the effective model is $\sB_0:=L^2_{\inv}(\Omega)\otimes L^2(Q)$. The effective dissipation potential $\mathcal{R}_{\mathsf{hom}}: \sB_0\to [0,\infty)$ is given by
\begin{equation*}
\mathcal{R}_{\mathsf{hom}}(v)=\ex{\int_{Q}r(\omega) |v(\omega,x)|^2 dx}.
\end{equation*}
The energy functional $\cE_{\mathsf{hom}}: \sB_{0} \to \R\cup \cb{\infty}$ is defined as 
\begin{align}\label{equation_442}
\begin{split}
 & \cE_{\mathsf{hom}}(u)  = \\ &\inf_{\chi \in L^2_{pot}(\Omega)\otimes L^2(Q)}\ex{ \int_{Q}  A(\omega)\brac{\nabla u(\omega,x)+\chi(\omega,x)}\cdot\brac{\nabla u(\omega,x)+\chi(\omega,x)}+   f(\omega,u(\omega,x))dx} 
\end{split}
\end{align}
for $u \in L^2_{\inv}(\Omega)\otimes H^1(Q)\cap L^p_{\mathsf{inv}}(\Omega)\otimes L^p(Q)=: dom(\cE_{\mathsf{hom}})$ and $\cE_{\mathsf{hom}}=\infty$ otherwise. We remark that $u\mapsto \cE_{\mathsf{hom}}(u)-\Lambda \mathcal{R}_{\mathsf{hom}}(u)$ is convex. The limit differential inclusion is 
\begin{equation}\label{diff:incl:2}
0\in D\mathcal{R}_{\mathsf{hom}}(\dot{u}(t))+\partial_{F}\mathcal{E}_{\mathsf{hom}}(u(t)) \quad \text{for a.e. } t\in (0,T), \quad u(0)=u_{0},
\end{equation}
where $\partial_{F}\cE_{\mathsf{hom}}: \sB_0 \to 2^{\sB_0^*}$ is the Frech{\'e}t subdifferential of $\cE_{\mathsf{hom}}$ defined analogously as $\partial_{F}\cE\e$.
If $(B1)-(B2)$ hold and for initial data $u_0\in dom(\cE_{\mathsf{hom}})$, (\ref{diff:incl:2}) has a unique solution $u\in H^1(0,T;\sB_0)$ (see e.g. \cite[Theorem 3.2]{clement2009introduction}).

The following homogenization result is based on a strategy related to a general method for evolutionary $\Gamma$-convergence of abstract gradient systems presented in \cite[Theorem 3.2]{mielke2014deriving}. In our particular case, the latter strategy does not apply due to the lack of a compactness property used to treat the non-convexity of the energy functional. In our model a priori bounds do not imply compactness: namely $L^2(\Omega)\otimes H^1(Q)$ is not compactly embedded into $\sB= L^2(\Omega\times Q)$. In contrast, in deterministic homogenization of similar problems (e.g. \cite{liero2015homogenization}) the compact Sobolev embedding $H^1(Q){\subset} L^p(Q)$ with $p<2^*$ is critically used. In the stochastic case, we only have $L^2(\Omega)\otimes H^1(Q)\subset L^2(\Omega)\otimes L^p(Q)$ continuously. We remedy this issue by reducing problem (\ref{diff_incl}) to an equivalent evolutionary variational inequality with a modified (convex) energy functional which allows us to pass to the limit $\varepsilon\to 0$ using merely weak convergence. As an additional merit of this procedure, in our results the growth assumptions on the integrand $f$ are independent of the Sobolev exponents. 
\begin{thm}[Evolutionary $\Gamma$-convergence]\label{s3_thm_5} Let $p>2$ and $Q\subset \R^d$ be open and bounded. Assume $(B1)-(B2)$, and consider $u_0\in dom(\cE_{\mathsf{hom}})$, $u_{0,\varepsilon} \in dom(\cE\e)$ such that
\begin{equation*}
u_{0,\varepsilon} \to u_0 \text{ strongly in }\sB, \quad \cE_{\varepsilon}(u_{0,\varepsilon}) \to \cE_{\mathsf{hom}}(u_0) \quad \text{(well-prepared initial data).}
\end{equation*}
Then $u\e\in H^1(0,T;\sB)$, the unique solution to (\ref{diff_incl}), satisfies: For all $t\in [0,T]$
\begin{align*}
& u\e(t) \to u(t) \text{ in }\sB, \quad P_{\inv}\nabla u_{\varepsilon}(t)\weakto \nabla u (t) \text{ weakly in }\sB^d,
\end{align*}
where $u\in H^1(0,T;\sB_0)$ is the unique solution to (\ref{diff:incl:2}). Moreover, it holds $\dot{u}\e \to \dot{u}$ strongly in $L^2(0,T; \sB)$ and for any $t\in [0,T]$
\begin{equation*}
\cE_{\varepsilon}(u\e(t))\to \cE_{\mathsf{hom}}(u(t)).
\end{equation*}
\end{thm}
\textit{(For the proof see Section \ref{S_Proof_3}.)}
\begin{remark}
Note that the proof of the above theorem (in particular the convergence of the energies) allows us to additionally characterize the two-scale limit of $\nabla u\e$. Specifically, for all $t\in [0,T]$ it holds
\begin{equation*}
\nabla u\e \overset{2s}{\rightharpoonup} \nabla u(t) +\chi(t) \; \text{in }\sB,
\end{equation*}
where $\chi(t)\in L^2_{\mathsf{pot}}(\Omega)\otimes L^2(Q)$ is the solution to the minimization problem given on the right-hand side of (\ref{equation_442}) (with $u=u(t)$). 
\end{remark}
\begin{remark}[Ergodic case]
If we additionally assume that $\ex{\cdot}$ is ergodic, the limit system is driven by deterministic functionals. In particular, the limit is described by a state space $\tilde{\sB}_0=L^2(Q)$, dissipation potential
\begin{equation*}
\tilde{\mathcal{R}}_{\mathsf{hom}}(u)=\int_Q \ex{r}|u(x)|^2 dx, 
\end{equation*}
and energy functional (for $u\in H^1(Q)\cap L^p(Q)$ and otherwise $\infty$)
\begin{equation*}
\tilde{\cE}_{\mathsf{hom}}(u)=\int_{Q}A_{\mathsf{hom}}\nabla u(x)\cdot \nabla u(x) + f_{\mathsf{hom}}(u(x)) dx,
\end{equation*}
where $A_{\mathsf{hom}}$ and $f_{\mathsf{hom}}$ are defined as: Let $A_{\mathsf{hom}}F\cdot F=\inf_{\chi \in L^2_{pot}(\Omega)}\ex{A(\omega)(F+\chi(\omega))\cdot (F+\chi(\omega))}$ for $F\in \R^d$, and let $f_{\mathsf{hom}}(y)=\ex{f(\omega,y)}$ for $y\in \R$. This suggests that in the ergodic case we might lift the above averaged result to a quenched statement (convergence for $P$-a.e. $\omega\in \Omega$) similarly as in Section \ref{Section:4:3} for homogenization of convex integrals. 
\end{remark}
\subsection{Equality of mean and quenched limits}\label{Section:3:3}
In this section we show that for sequences of random functionals both mean and quenched homogenization (if both are possible) yield the same effective functional. 

Let $p\in (1,\infty)$ and $Q\subset \R^d$ be open. Consider $\cb{\cE^{\omega}\e: L^p(Q)\to \R \cup \cb{\infty}}_{\omega\in\Omega}$, a family of random functionals that $\Gamma$-converges to a deterministic functional $\cE_{\mathsf{hom}}: L^p(Q)\to \R \cup \cb{\infty}$ for $P$-a.e. $\omega\in \Omega$ (we refer to this notion as quenched homogenization). Under certain measurability assumptions (detailed below), we may consider the averaged functional $\cE\e: L^p(\Omega\times Q)\to \R \cup \cb{\infty}$, $\cE\e(u)=\ex{\cE^{\omega}\e(u(\omega))}$. We assume that $\cE\e$ $\Gamma$-converges in the mean ($2$-scale sense) to a deterministic limit $\widetilde{\cE}_{\mathsf{hom}}: L^p(Q)\to \R\cup \cb{\infty}$. A specific example of such situation are integral functionals of the form $\cE\e^{\omega}(u)=\int_{Q}V(\tau_{\frac{x}{\varepsilon}}\omega,\nabla u(x))dx$, where $V$ satisfies the assumptions from Section \ref{Section_Convex}.  A quenched homogenization result for integral functionals of this form is obtained in \cite{DalMaso1986} (based on the subadditive ergodic theorem) and a mean homogenization result for the corresponding averaged functionals is given in Section \ref{Section_Convex} (using the unfolding procedure). For a generic situation, below we show that the mean and quenched $\Gamma$-limits match, i.e. $\cE_{\mathsf{hom}}=\widetilde{\cE}_{\mathsf{hom}}$. 

To make the above discussion precise, we require the following assumptions: There exist $\Omega'\subset \Omega$ with $P(\Omega')= 1$, $C>0$, and $\psi \in L^1(\Omega)$ such that: 
\begin{enumerate}[(C1)]
\item {The mapping $\Omega\times L^p(Q)\ni (\omega,u)\mapsto \cE^{\omega}\e(u)$ is $\mathcal{F}\otimes \cB(L^p(Q))$-measurable. For all $\omega \in \Omega'$, $\cE^{\omega}_{\varepsilon}\geq -C$, $\inf_{u}\cE^{\omega}\e(u) \leq \psi(\omega)$ and $\cE^{\omega}_{\varepsilon}$ is convex, proper, and l.s.c.}

(This implies that $\cE\e$ is a (well-defined) convex, proper and l.s.c. functional.)
\item {It holds that $dom(\cE\e^{\omega})=X \subset W^{1,p}(Q)$ ($X$ is convex, closed and compactly embedded in $L^p(Q)$) and for all $\omega\in \Omega'$, $\cE_{\varepsilon}^{\omega}(u)\geq \frac{1}{C}\|u\|^p_{W^{1,p}(Q)} - C$ for all $u\in L^p(Q)$.}

(This implies that $\cE\e(u)\geq \frac{1}{C}\|u\|_{L^p(\Omega)\otimes W^{1,p}(Q)}-C$ for all $u \in L^p(\Omega\times Q)$. Moreover, $\cE\e^{\omega}$ (resp. $\cE\e$) is equi-mildly coercive in $L^p(Q)$ (resp. w.r.t. weak two-scale convergence).)
\item {There exist $\cE_{\mathsf{hom}}: L^p(Q)\to \R \cup \cb{\infty}, \; \widetilde{\cE}_{\mathsf{hom}}: L^p(Q)\to \R \cup \cb{\infty}$ such that for all $\omega\in \Omega'$, $\cE\e^{\omega}\overset{\Gamma}{\to}\cE_{\mathsf{hom}}$ in $L^p(Q)$, and $\cE\e \overset{\Gamma}{\rightharpoonup} \widetilde{\cE}_{\mathsf{hom}}$ in the following sense:}
\begin{enumerate}[(i)]
\item {If $u\e \overset{2}{\rightharpoonup} u$ where $u\in L^p(Q)$, then $\liminf_{\varepsilon\to 0}{\cE}\e(u\e) \geq \widetilde{\cE}_{\mathsf{hom}}(u)$.
\item For $u\in L^p(Q)$, there exists $u\e \in L^p(\Omega\times Q)$, such that $u\e \overset{2}{\rightharpoonup} u$, ${\cE}\e(u\e) \to \widetilde{\cE}_{\mathsf{hom}}(u).$}
\end{enumerate}
\end{enumerate} 
\begin{prop}\label{proposition:818}
Let $p\in (1,\infty)$, $Q\subset \R^d$ be open and $\ex{\cdot}$ be ergodic. If we assume $(C1)-(C3)$, then 
\begin{equation*}
\cE_{\mathsf{hom}}=\widetilde{\cE}_{\mathsf{hom}}.
\end{equation*}
(For the proof see Section \ref{S_Proof_3}.)
\end{prop}
\subsection{Proofs}\label{S_Proof_3}
\begin{proof}[Proof of Theorem \ref{thm1}]
(i)
The Poincar{\'e}-Korn inequality and the growth conditions of $V$ imply that $u\e$ is bounded in $L^p(\Omega)\otimes W^{1,p}(Q)^d$. By Proposition \ref{prop1} there exist $u \in L^p_{{\inv}}(\Omega)\otimes W^{1,p}(Q)^d$ and $\chi \in L^p_{\pot}(\Omega)\otimes L^p(Q)^d$ with the claimed convergence (up to a subsequence). From $\unf u\e \in L^p(\Omega)\otimes W^{1,p}_0(Q)^d$ for every $\varepsilon>0$, we conclude that $u \in L^p_{{\inv}}(\Omega)\otimes W^{1,p}_0(Q)^d$ (cf. Remark \ref{R_Two_1}).
\medskip

(ii) The claim follows from Proposition \ref{P_Cont_1} (iii).
\medskip

(iii) The existence of a strongly two-scale convergent sequence $u\e \in L^p(\Omega)\otimes W^{1,p}_0(Q)^d$ follows from Remark \ref{rem14}. Furthermore, the convergence of the energy $\cE_{\varepsilon}(u_{\varepsilon})\to \cE_{0}(u,\chi)$ follows from Proposition \ref{P_Cont_1} (ii).
\end{proof}
\medskip
\begin{proof}[Proof of Corollary~\ref{C:thm1}]
 The statement follows by a standard argument from $\Gamma$-convergence: Since $u\e$ is a minimizer we conclude that $\limsup_{\eps\to 0}\cE_\eps(u\e)\leq\limsup_{\eps\to 0}\cE_\eps(0)<\infty$. Hence, by Theorem~\ref{thm1} there exists $u\in L^p_{\inv}(\Omega)\otimes W^{1,p}_0(Q)^d$ and $\chi\in L^p_{\pot}(\Omega)\otimes L^p(Q)^d$ such that $u\e \wt u \text{ in }\ltp^d$, $\nabla u\e \wt \nabla u+\chi  \text{ in }\ltp^{d\times d}$, and 
  \begin{equation*}
    \liminf\limits_{\eps\to 0}\cE\e(u\e)\geq \cE_0(u,\chi).
  \end{equation*}
  Let $(u_0,\chi_0)$ denote the minimizer of $\cE_0$. Then by Theorem~\ref{thm1} (iii) there exists a recovery sequence $v\e$ s.t.~$\cE_\eps(v\e)\to \cE_0(u_0,\chi_0)$, and thus
  \begin{equation*}
    \min\cE_0=\lim\limits_{\eps\to 0}\cE_\eps(v\e)\geq \liminf\limits_{\eps\to 0}\cE_\eps(u\e)=\liminf\limits_{\eps\to 0}\min\cE_\eps\geq \cE_0(u,\chi)\geq \min\cE_0,
  \end{equation*}
  and thus $(u,\chi)$ is a minimizer of $\cE_0$ and $\cE_\eps(u\e)=\min\cE_\eps\to \min\cE_0=\cE_0(u,\chi)$.
\end{proof}
Before presenting the proof of Theorem \ref{thm2}, we provide two auxiliary results.
\begin{lemma}[Stochastic Korn inequality]\label{lem8} Let $p\in (1,\infty)$. There exists $C>0$ such that
\begin{equation*}
\ex{|\chi|^p}\leq C \ex{|\chi^s|^p} \quad \text{for every }\chi\in L^p_{\pot}(\Omega)^d.
\end{equation*}
\end{lemma}
The proof of the above inequality is similar as the argument for the case $p=2$ in \cite{heida_schweizer2017stochastic}. For the reader's convenience, we show it in the Appendix \ref{appendix:1}.

For the proof of Theorem~\ref{thm2} we apply Castaing's measurable selection lemma in the following form:
\begin{lemma}[See Theorem III.6 and Proposition III.11 in \cite{castaing2006convex}]\label{castaing}
Let $X$ be a complete separable metric space, $(\mathcal{S},\sigma)$ a measurable space and $\Gamma:\mathcal{S}\rightarrow 2^X$ a multifunction. Further, assume that for all $x\in \mathcal{S}$, $\Gamma(x)$ is nonempty and closed in $X$, and for any closed $G\subset X$ we have
\begin{equation*}
\Gamma^{-1}(G):=\cb{x\in S: \Gamma(x)\cap G\ne \varnothing}\in \sigma.
\end{equation*}
Then $\Gamma$ admits a measurable selection, i.e. there exists $\tilde{\Gamma}:\mathcal{S}\rightarrow X$ measurable with $\tilde{\Gamma}(x)\in \Gamma(x)$.
\end{lemma}

\begin{proof}[Proof of Theorem \ref{thm2}]
(i) According to Theorem \ref{thm1} (i) there exist $u\in \sob_0(Q)$ and $\chi\in L^p_{\pot}(\Omega)\otimes L^p(Q)^d$ such that (using Proposition \ref{prop1}) $u\e$ satisfies the claimed convergences. Furthermore, we have 
\begin{align*}
\liminf_{\varepsilon\rightarrow 0}\mathcal{E}_{\varepsilon}(u\e)\geq \mathcal{E}_{0}(u,\chi)\geq \mathcal{E}_{\hom}(u).
\end{align*}
(ii) We show that there exists $\chi \in L^p_{\mathsf{pot}}(\Omega)\otimes L^p(Q)^d$ such that $\cE_{0}(u,\chi)=\cE_{\hom}(u)$, which implies the claim applying Theorem \ref{thm1} (iii). It is sufficient to show that for fixed $F \in \R^{d\times d}$ and a fixed $Q'\subset Q$ (measurable), we can find $\chi \in L^p_{\mathsf{pot}}(\Omega)\otimes L^p(Q')^d$ such that
\begin{equation}\label{claim1290}
\int_{Q'}\ex{V(\omega,x,F^s+ \chi^s(x,\omega))}dx = \int_{Q'}V_{\mathsf{hom}}(x,F)dx.
\end{equation}
Indeed, if the above holds, we approximate $\nabla u$ by piecewise-constant functions $F_k=\sum_{i}\mathbf{1}_{Q_{k,i}}F_{k,i} $ (in the strong $L^p(Q)$ topology), where $F_{k,i}\in \R^{d \times d}$, and we find $\chi_{k}=\sum_{i}\mathbf{1}_{Q_{k,i}}\chi_{k,i} \in L^p_{\mathsf{pot}}(\Omega)\otimes L^p(Q)^d$ such that
\begin{equation}\label{claim858}
\int_{Q}\ex{V(\omega,x,F_k^s(x)+\chi_k^s(x,\omega))}dx = \int_{Q}V_{hom}(x,F_k(x))dx.
\end{equation}
Using the growth conditions of $V$ and Lemma \ref{lem8}, it follows $\limsup_{k\to \infty} \|\chi_k\|_{L^p_{\mathsf{pot}}(\Omega)\otimes L^p(Q)^d}<\infty$ and therefore we may extract a (not relabeled) subsequence and $\chi \in  L^p_{\mathsf{pot}}(\Omega) \otimes L^p(Q)^d$ such that $\chi_k \rightharpoonup \chi$ weakly in $L^p(\Omega\times Q)^{d\times d}$. Note that the functional on the left-hand side of (\ref{claim858}) is weakly l.s.c. and the functional on the right-hand side is continuous (by continuity of $V_{\mathsf{hom}}(x,\cdot)$ and growth conditions of $V$). As a result of this, we may pass to the limit $k\to \infty$ in (\ref{claim858}), in order to obtain $\cE_{0}(u,\chi) \leq \cE_{\hom}(u)$. Also, the other inequality $\cE_{0}(u,\chi) \geq \cE_{\hom}(u)$ follows by the definition of $V_{\mathsf{hom}}$ and therefore we conclude that $\cE_{0}(u,\chi) = \cE_{\hom}(u)$.

In the following we show (\ref{claim1290}). Fix $F\in \R^{d\times d}$ and $Q'\subset Q$ and let 
\begin{align*}
& f: Q'\times L^p_{\mathsf{pot}}(\Omega)^d \to \R, \quad f(x,\chi)= \ex{V(\omega,x,F^s+\chi^s(\omega))},\\
& \phi: Q' \to \R,\quad \phi(x)=\inf_{\chi\in L^p_{\mathsf{pot}}(\Omega)^d}f(x,\chi).
\end{align*}
We define a multifunction $\Gamma: Q' \to 2^{L^p_{\mathsf{pot}}(\Omega)^d}$ as (the set of all correctors corresponding to the point $x\in Q'$)
\begin{equation*}
\Gamma(x)= \cb{\chi \in L^p_{\mathsf{pot}}(\Omega)^d: f(x,\chi)\leq \phi(x)}.
\end{equation*}
For each $x\in Q'$, $\Gamma(x)$ is non-empty and closed (using the direct method of calculus of variations). 

In the following, we show that for a closed set $G \subset L^p_{\mathsf{pot}}(\Omega)^d$, $\Gamma^{-1}(G)$ is measurable and therefore Lemma \ref{castaing} implies that there exists $\chi \in  L^p_{\mathsf{pot}}(\Omega) \otimes L^p(Q')^d$ which satisfies (\ref{claim1290}).

Note that $f$ defines a Charath{\'e}odory integrand in the sense that for fixed $x\in Q'$, $f(x,\cdot)$ is continuous, and for fixed $\chi \in L^p_{\mathsf{pot}}(\Omega)$, $f(\cdot, \chi)$ is measurable. Moreover, since $L^p_{\mathsf{pot}}(\Omega)$ is separable, we find a countable (dense) set $\cb{\chi_k}\subset L^p_{\mathsf{pot}}(\Omega)^d$ such that $\phi(x)=\inf_{\chi\in L^p_{\mathsf{pot}}(\Omega)^d}f(x,\chi)=\inf_{k}f(x,\chi_k)$ (using that the infimum in the definition of $\phi$ is attained and that $f(x,\cdot)$ is continuous). As a result of this, we conclude that $\phi$ is measurable and moreover we have that the function $\tilde{f}:=f-\phi: Q'\times L^p_{\mathsf{pot}}(\Omega)^d\to \R$ is as well a Charath{\'e}odory integrand. Consequently, \cite[Proposition 1]{rockafellar1971} implies that $x \mapsto epi \tilde{f}_x$ (where $epi \tilde{f}_x$ denotes the epigraph of the function $\tilde{f}(x,\cdot)$) is measurable in the following sense (see \cite[Theorem 1]{rockafellar1971}): For any closed set $\tilde{G} \subset \R \times L^p_{\mathsf{pot}}(\Omega)^d$, it holds that  
\begin{equation*}
\cb{x \in Q': epi \tilde{f}_x \cap \tilde{G} \neq \varnothing}= \cb{x\in Q': \exists (\alpha,\chi) \in \tilde{G}, \quad \tilde{f}(x,\chi)\leq \alpha}
\end{equation*}
is measurable. We choose $\tilde{G}=\cb{0}\times G$ and the above implies that $\Gamma^{-1}(G)$ is measurable. This concludes the proof.
\end{proof}
\begin{proof}[Proof of Proposition \ref{prop3}]
Uniqueness of minimizers follows by the uniform convexity assumption on the integrand $V$. As in the proof of Theorem \ref{thm2} (ii), we select $\chi \in L^p_{\pot}(\Omega)\otimes L^p(Q)^d$ such that $\int_{Q}V_{\hom}(x,\nabla u(x))dx=\int_{Q}\ex{V(\omega, x, \nabla^s u(x)+\chi^s(\omega,x))}dx$. Theorem \ref{thm1} (iii) implies that there exists a sequence $v\e\in L^p(\Omega)\otimes W^{1,p}_0(Q)^d$ such that $v\e \st u$ in $\ltp^d$ and $\mathcal{E}\e(v\e)\to \mathcal{E}_0(u,\chi)=\mathcal{E}_{\hom}(u)$. By the triangle inequality we have
$\|u\e - u\|_{L^p(\Omega\times Q)}\leq \|u\e - v\e\|_{L^p(\Omega\times Q)}+\|v\e - u\|_{L^p(\Omega\times Q)}$. By the isometry property of $\unf$ and strong two-scale convergence of $v\e$ we have $\|v\e - u\|_{L^p(\Omega\times Q)}=\|\unf(v\e - u)\|_{L^p(\Omega\times Q)}=\|\unf v\e - u\|_{L^p(\Omega\times Q)}\to 0$. Furthermore, the Poincar{\'e}-Korn inequality $\|u\e-v\e\|^p_{L^p(\Omega\times Q)}\leq C \|\nabla^su\e-\nabla^sv\e\|^p_{L^p(\Omega\times Q)}$ (for a generic constant $C$ that is independent of $\eps$ but might change from line to line), the uniform convexity of $V$ in form of $\frac{C}{4}\|\nabla^su\e-\nabla^sv\e\|_{L^p(\Omega\times Q)}^p\leq \frac12\cE\e(v\e)+\frac12\cE\e(u\e)-\cE\e(\frac12(u\e+v\e))$,
and the minimality of $u\e$, yield the estimate
\begin{equation*}
  \|u\e-v\e\|^p_{L^p(\Omega\times Q)}\leq C\brac{\cE\e(v\e)-\cE\e(\tfrac12(u\e+v\e))}.
\end{equation*}
Since $\cE\e(v\e)\to\cE_{\hom}(u)$ and $\liminf\limits_{\eps\to 0}\cE\e(\tfrac12(u\e+v\e))\geq \cE(u,\chi)=\cE_{\hom}(u)$, we conclude that the right-hand side converges to $0$.
Thus, $u\e\to u$ in $L^p(\Omega\times Q)$, and the convergence of the gradient follows using Proposition \ref{prop1}.
\end{proof}
\begin{proof}[Proof of Theorem \ref{s3_thm_5}]

\textit{Step 1. A priori estimates and compactness.}

In the following, using a standard argument, we derive an a priori estimate for the solution $u\e$. We note that (\ref{diff_incl}) implies
\begin{equation*}
\mathcal{R}\e(\dot{u}\e(t))\leq \ex{-\xi\e(t),\dot{u}\e(t)} ,
\end{equation*}
where $\xi\e(t)\in \partial_F \cE\e(u\e(t))$. Integrating the above on the interval $(0,t)$ (with arbitrary $t\in (0,T]$) and using the chain rule for the ($\Lambda$-convex) energy functional $\cE\e$ (see e.g., \cite{Rossi2006} for the chain rule), we obtain 
\begin{equation}\label{help:2}
\cE\e(u\e(t))+ \int_0^t \mathcal{R}\e(\dot{u}\e(s))ds \leq \cE\e(u\e(0)).
\end{equation}
Using the assumptions on the initial data $u\e(0)$, the above  implies that there exists $C>0$ (independent of $\varepsilon$) such that
\begin{equation}\label{uniform:estimate}
\sup_{t\in [0,T]} \brac{\|u\e(t)\|_{L^p(\Omega\times Q)}+\|u\e(t)\|_{L^2(\Omega)\otimes H^1(Q)}} + \|u\e\|_{H^1(0,T;\sB)}\leq C.
\end{equation}
As a result of this, we find a (not relabeled) subsequence and $\widetilde{u}\in H^1(0,T; \sB)\cap L^p(0,T; L^p(\Omega\times Q))$ such that 
\begin{equation*}
u\e \rightharpoonup \widetilde{u} \quad \text{weakly in }H^1(0,T; \sB) \text{ and weakly in }L^p(0,T; L^p(\Omega\times Q)).
\end{equation*}
Moreover, using the Arzel{\`a}-Ascoli theorem \cite[Proposition 3.3.1]{ambrosio2008gradient}, we might extract another subsequence such that for all $t\in [0,T]$
\begin{equation*}
u_{\varepsilon}(t)\rightharpoonup \widetilde{u}(t) \quad \text{weakly in }\sB,
\end{equation*}
and using (\ref{uniform:estimate}) and Proposition \ref{prop1} we conclude that $u\e(t)\overset{2s}{\rightharpoonup} u(t)$ in $\sB$ where $u(t):=P_{\inv}\widetilde{u}(t)$. We consider the linear extension of $\unf$ to an (not relabeled) operator $\unf: L^2(0,T; \sB)\mapsto L^2(0,T; \sB)$ and note that $\unf$ and $\dot{(\cdot)}$ commute. This results in the convergence $\unf \dot{u}\e \rightharpoonup \dot{u}$ weakly in $L^2(0,T; \sB)$. 
\medskip

\textit{Step 2. Reduction to a convex problem.}

The differential inclusion in (\ref{diff_incl}) is equivalent to 
\begin{equation*}
\cE\e(u\e(t))\leq \cE\e(w)-\ex{D\mathcal{R}\e(\dot{u}\e(t)),u\e(t)-w}_{\sB^*,\sB}- \Lambda \mathcal{R}\e(u\e(t)-w) \quad \text{for all }w \in \sB.
\end{equation*}
We set $w= e^{-\Lambda t} \tilde{w}$ with an arbitrary $\tilde{w}\in \sB$ and multiply the above inequality by $e^{2\Lambda t}$ to obtain
\begin{eqnarray}\label{eq:898}
&& e^{2\Lambda t}\cE\e(e^{-\Lambda t}e^{\Lambda t}u\e(t)) \\ & \leq & e^{2\Lambda t}\cE\e(e^{-\Lambda t}\tilde{w}) - e^{2\Lambda t} \ex{D\mathcal{R}\e(\dot{u}\e(t)),u\e(t) - e^{-\Lambda t}\tilde{w}}_{\sB^*,\sB}-\Lambda e^{2\Lambda t} \mathcal{R}\e(u\e(t)-e^{-\Lambda t}\widetilde{w}). \nonumber
\end{eqnarray}
Using the quadratic structure of $\mathcal{R}\e$ (and its homogeneity of degree 2), we compute 
\begin{equation*}
-\Lambda e^{2\Lambda t}\mathcal{R}\e (u\e(t)-e^{-\Lambda t}\tilde{w})=\Lambda \mathcal{R}\e(e^{\Lambda t}u\e(t))-\Lambda \mathcal{R}\e(\tilde{w})- \ex{D\mathcal{R}\e(\Lambda e^{\Lambda t}u\e(t)),e^{\Lambda t}u\e(t)-\tilde{w}}_{\sB^*,\sB}.
\end{equation*}
Using this equality, (\ref{eq:898}) implies that for any $\tilde{w}\in \sB$,
\begin{equation*}
\widetilde{\cE}\e(t,e^{\Lambda t}u\e(t)) \leq \widetilde{\cE}\e(t,\tilde{w})-\ex{D\mathcal{R}\e(e^{\Lambda t}\dot{u}\e(t)+\Lambda e^{\Lambda t}u\e(t)),e^{\Lambda t}u\e(t)-\tilde w}_{\sB^*,\sB},
\end{equation*} 
where $\widetilde{\cE}\e:[0,T]\times \sB \to \R \cup \cb{\infty}$, $\widetilde{\cE}\e(t,v) = e^{2\Lambda t}\cE\e(e^{-\Lambda t}v)-\Lambda \mathcal{R}\e(v)$. We remark that for each $t\in [0,T]$, $\widetilde{\cE}\e(t,\cdot)$ is a convex, l.s.c, proper functional.
We introduce a new variable $v\e: [0,T]\to \sB$ given by $v\e(t):=e^{\Lambda t}u\e(t)$ and note that $\dot{v}\e(t)=e^{\Lambda t} \dot{u}\e(t)+\Lambda e^{\Lambda t}u\e(t)$. As a result of this,
the above inequality implies that for all $w \in L^2(0,T; \sB)$
\begin{equation}\label{main:B}
\widetilde{\cE}\e(t,v\e(t))+\ex{-D\mathcal{R}\e(\dot{v}\e(t)),w(t)}_{\sB^*,\sB}-\widetilde{\cE}\e(t,w(t))\leq \ex{-D\mathcal{R}\e(\dot{v}\e(t)),v\e(t)}_{\sB^*,\sB}.
\end{equation}
Integrating the above inequality on the interval $[0,T]$ and using the chain rule, we obtain (setting $w=w\e$ a sequence that is arbitrary, but that we specify below)
\begin{equation}\label{main:inequality:189}
\int_0^T \widetilde{\cE}\e(t,v\e(t))+\ex{-D\mathcal{R}\e(\dot{v}\e(t)),w\e(t)}_{\sB^*,\sB}-\widetilde{\cE}\e(t,w\e(t))dt\leq -\mathcal{R}\e(v\e(T))+\mathcal{R}\e(v\e(0)).
\end{equation}
\textit{Step 3. Passage to the limit $\varepsilon\to 0$.}

Since $u\e(t)\overset{2s}{\rightharpoonup}u(t)$ in $\sB$, it follows that $v\e(t)\overset{2s}{\rightharpoonup} v(t):= e^{\Lambda t} u(t)$ in $\sB$ for all $t\in [0,T]$. Note that $v\e(0)=u\e(0)\to u(0)=v(0)$ strongly in $\sB$ and therefore the last term on the right-hand side of (\ref{main:inequality:189}) converges to $\mathcal{R}_{\mathsf{hom}}(v(0))$. The first term on the right-hand side satisfies
\begin{equation}\label{eq:920}
\limsup_{\varepsilon\to 0}\brac{-\mathcal{R}\e(v\e(T))} = -\liminf_{\varepsilon\to 0}\ex{\int_Q r(\omega)|\unf v\e(T)|^2 dx }\leq -\mathcal{R}_{\mathsf{hom}}(v(T)).
\end{equation}
For the first term on the left-hand side, using Fatou's lemma, we have
\begin{equation*}
\liminf_{\varepsilon\to 0}\int_{0}^{T}\widetilde{\cE}\e(t,v\e(t))dt \geq \int_{0}^{T}\liminf_{\varepsilon\to 0} \widetilde{\cE}\e(t,v\e(t))dt.
\end{equation*}
Moreover, for a fixed $t\in (0,T)$, we find a subsequence $\varepsilon'$ such that $\liminf_{\varepsilon\to 0}\widetilde{\cE}\e(t,v\e(t))=\lim_{\varepsilon'\to 0}\widetilde{\cE}_{\varepsilon'}(t,v_{\varepsilon'}(t))$. With help of the uniform estimate (\ref{uniform:estimate}), we obtain $v_{\varepsilon'}(t)\overset{2s}{\rightharpoonup} v(t)$ in $L^p(\Omega\times Q)$ and (up to a subsequence) $\nabla v_{\varepsilon'}(t)\overset{2s}{\rightharpoonup}\nabla v(t)+\chi$ in $L^2(\Omega\times Q)$ for some $\chi \in L^2_{\mathsf{pot}}(\Omega)\otimes L^2(Q)$. Therefore, we have
\begin{align}\label{eq:928}
\begin{split}
\liminf_{\varepsilon\to 0}\widetilde{\cE}\e(t,v\e(t))  \geq & \liminf_{\varepsilon'\to 0} \ex{\int_Q A(\omega)\mathcal{T}_{\varepsilon'}\nabla v_{\varepsilon'}(t)(\omega,x)\cdot \mathcal{T}_{\varepsilon'} \nabla v_{\varepsilon'}(t)(\omega,x)dx} \\ & +\liminf_{\varepsilon'\to 0} \ex{\int_Q e^{2\Lambda t}f(\omega, e^{-\Lambda t} \mathcal{T}_{\varepsilon'} v_{\varepsilon'}(t)(\omega,x))-\Lambda \frac{r(\omega)}{2}|\mathcal{T}_{\varepsilon'} v_{\varepsilon'}(t)(\omega,x)|^2 dx}\\ \geq & \widetilde{\cE}_{\mathsf{hom}}(t,v(t)),
\end{split}
\end{align}
where $\widetilde{\cE}_{\mathsf{hom}}(t,v):=e^{2\Lambda t}\cE_{\mathsf{hom}}(e^{-\Lambda t}v)-\Lambda \mathcal{R}_{\mathsf{hom}}(v)$ and the last inequality follows by convexity (and l.s.c.) of the underlying functionals. 

In order to complete the limit passage in (\ref{main:inequality:189}), it is left to treat the second and third terms on the left-hand side. In the following, we show that there exists a sequence $w\e\in L^2(0,T; \sB)$ such that
\begin{equation}\label{second:term:convergence}
\lim_{\varepsilon\to 0} \int_0^T \ex{-D\mathcal{R}\e(\dot{v}\e(t)),w\e(t)}_{\sB^*,\sB}-\widetilde{\cE}\e(t,w\e(t))dt= \int_0^T \widetilde{\cE}_{\mathsf{hom}}^*(t,-D\mathcal{R}_{\mathsf{hom}}(\dot{v}(t)))dt,
\end{equation}
where $\widetilde{\cE}_{\mathsf{hom}}^*(t,\cdot):\sB_0^* \to \R \cup \cb{\infty}$ denotes the conjugate of the functional $\widetilde{\cE}_{\mathsf{hom}}(t,\cdot)$, and it is defined as 
\begin{equation*}
\widetilde{\cE}_{\mathsf{hom}}^*(t,\xi)=\sup_{w\in \sB_0}\brac{\ex{\xi,w}_{\sB_0^*,\sB_0}-\widetilde{\cE}_{\mathsf{hom}}(t,w)}.
\end{equation*}
We remark that $\widetilde{\cE}_{\mathsf{hom}}: [0,T]\times \sB_0 \to \R \cup \cb{\infty}$ is a normal integrand in the sense that it is $\mathcal{L}(0,T)\otimes \mathcal{B}(\sB_0)$-measurable and for each $t\in [0,T]$, $\widetilde{\cE}_{\mathsf{hom}}(t,\cdot)$ is l.s.c. on $\sB_0$ and it is proper. Moreover, for each $t\in [0,T]$, $\widetilde{\cE}_{\mathsf{hom}}(t,\cdot)$ is convex. \cite[Theorem 2]{rockafellar1971} and the fact that the functional $w\in L^2(0,T; \sB_0)\mapsto \int_0^T \widetilde{\cE}_{\mathsf{hom}}(t,w(t))dt$ attains its minimum imply that there exists $w\in L^2(0,T; \sB_0)$ such that
\begin{equation*}
\int_0^T \widetilde{\cE}_{\mathsf{hom}}^*(t,-D\mathcal{R}_{\mathsf{hom}}(\dot{v}(t)))dt = \int_0^T \ex{-D\mathcal{R}_{\mathsf{hom}}(\dot{v}(t)),w(t)}_{\sB_0^*,\sB_0}-\widetilde{\cE}_{\mathsf{hom}}(t,w(t)) dt.
\end{equation*}
Moreover, it holds that $\nabla w \in L^2(0,T; \sB_0^d)$ and therefore similarly as in the proof of Theorem \ref{thm2} (ii) we may find $\chi \in L^2(0,T; L^2_{\mathsf{pot}}(\Omega)\otimes L^2(Q))$ such that
\begin{equation*}
\int_0^T \widetilde{\cE}_{\mathsf{hom}}(t,w(t))dt = \int_0^T e^{2\Lambda t}\cE_0(e^{-\Lambda t}w(t),e^{-\Lambda t}\chi(t))-\Lambda \mathcal{R}_{\mathsf{hom}}(w(t))dt,
\end{equation*}
where $\cE_0(v,\chi)= \ex{\int_Q A(\omega)(\nabla v(\omega,x)+\chi(\omega,x))\cdot (\nabla v(\omega,x) + \chi(\omega,x))+ f(\omega, v(\omega,x))dx}$. In the following, we construct a strong recovery sequence for $w$ and $\chi$ similarly as in Lemma \ref{Nonlinear_recovery} with the only difference that the functions to be recovered are time-dependent. Since $\chi \in  L^2(0,T; L^2_{\mathsf{pot}}(\Omega)\otimes L^2(Q))$, we find a sequence $g_{\delta}=\sum_{i=1}^{n_{\delta}}\xi_i^{\delta}\eta_i^{\delta}\varphi_i^{\delta}$ with $\xi_i^{\delta}\in L^p(0,T)$, $\eta_i^{\delta}\in C^{\infty}_c(Q)$ and $\varphi_i^{\delta}\in H^1(\Omega)\cap L^p(\Omega)$, such that
\begin{equation*}
\|Dg_{\delta}-\chi\|_{L^2(0,T;\sB^d)}\to 0 \quad \text{as }\delta \to 0.
\end{equation*}
Above, by a truncation and mollification argument (see e.g. \cite[Lemma 2.2]{bourgeat1994stochastic}) we may choose $\varphi_i^{\delta}\in H^1(\Omega)\cap L^p(\Omega)$ and not only in $H^1(\Omega)$ (as the definition of $L^2_{\mathsf{pot}}(\Omega)$ suggests).
We define $w_{\delta,\varepsilon}= w+ \varepsilon \mathcal{T}_{-\varepsilon}g_{\delta}$ and similarly as in the proof of Lemma \ref{Nonlinear_recovery}, we compute
\begin{eqnarray*}
&& \|\unf w_{\delta,\varepsilon} - w\|_{L^p(0,T; L^p(\Omega \times Q))}+\|\unf \nabla w_{\delta,\varepsilon} - \nabla w -\chi \|_{L^2(0,T; \sB^d)}\\ & \leq & \varepsilon \|g_{\delta}\|_{L^p(0,T; L^p(\Omega \times Q))}+ \|Dg_{\delta} -\chi \|_{L^2(0,T; \sB^d)}+\varepsilon \|\nabla g_{\delta} \|_{L^2(0,T; \sB^d)}.
\end{eqnarray*}
Letting first $\varepsilon \to 0$ and then $\delta \to 0$, the right-hand side above vanishes, therefore we can extract a diagonal sequence $\delta(\varepsilon)\to 0$ (as $\varepsilon\to 0$) such that $w_{\varepsilon}:=w_{\delta(\varepsilon),\varepsilon}$ satisfies
\begin{equation*}
\unf w\e \to w \text{ strongly in }L^p(0,T; L^p(\Omega\times Q)), \quad \unf \nabla w\e \to \nabla w + \chi \text{ strongly in }L^2(0,T; \sB^d).
\end{equation*} 
For the sequence $w\e$, we have 
\begin{eqnarray}
& & \int_0^T \ex{-D\mathcal{R}\e(\dot{v}\e(t)),w\e(t)}_{\sB^*,\sB}-\widetilde{\cE}\e(t,w\e(t))dt\\ & = & \int_{0}^T\ex{\int_Q -r(\omega)\unf \dot{v}\e(t,\omega,x) \unf w\e(t,\omega,x)- A(\omega)\unf \nabla w\e(t,\omega,x)\cdot \unf \nabla w\e(t,\omega,x)dx}dt \nonumber \\
&& - \int_{0}^T e^{2\Lambda t} \ex{\int_{Q}f(\omega,e^{-\Lambda t}\unf w\e(t,\omega,x))-\frac{\Lambda}{2} r(\omega)|\unf w\e(t,\omega,x)|^2 dx}dt. \nonumber 
\end{eqnarray}
Using the convergence properties of $u\e$, we obtain that $\unf \dot{v}\e \rightharpoonup \dot{v}$ weakly in $L^2(0,T; \sB)$ and therefore the first term on the right-hand side above converges to $\int_{0}^T \ex{-D\mathcal{R}_{\mathsf{hom}}(\dot{v}(t)),w(t)}_{\sB_0^*,\sB_0}dt$. Using the strong convergence of $\unf \nabla w\e$, it follows that the second term on the right-hand side converges to $-\int_0^{T}\ex{\int_Q A(\omega)(\nabla w(t,\omega,x)+\chi(t,\omega,x))\cdot (\nabla w(t,\omega,x)+\chi(t,\omega,x))dx}dt$. Using the growth assumptions of $f$ and its continuity in its second variable and with help of the strong convergence $\unf w\e \to w$ in $L^p(0,T; L^p(\Omega\times Q))$, we conclude that the sum of the last two terms converges to 
\begin{equation*}
-\int_{0}^T e^{2\Lambda t}\ex{\int_{Q}f(\omega,e^{-\Lambda t}w(t,\omega,x))dx}+\Lambda \mathcal{R}_{\mathsf{hom}}(w(t))dt.
\end{equation*}
Collecting the above statements, we have that $w\e$ satisfies (\ref{second:term:convergence}). 

Finally, considering all the above estimates for the terms in (\ref{main:inequality:189}), we are able to pass to the limit $\varepsilon\to 0$ in (\ref{main:inequality:189}) to obtain
\begin{eqnarray*}
& & \int_{0}^T \widetilde{\cE}_{\mathsf{hom}}(t,v(t))+\widetilde{\cE}^*_{\mathsf{hom}}(t,-D\mathcal{R}_{\mathsf{hom}}(\dot{v}(t))) dt\\ & \leq & -\mathcal{R}_{\mathsf{hom}}(v(T))+\mathcal{R}_{\mathsf{hom}}(v(0))=\int_{0}^T \ex{-D\mathcal{R}_{\mathsf{hom}}(\dot{v}(t)),v(t)}_{\sB_0^*,\sB_0}dt.
\end{eqnarray*}
We have $\widetilde{\cE}_{\mathsf{hom}}(t,v(t))+\widetilde{\cE}^*_{\mathsf{hom}}(t,-D\mathcal{R}_{\mathsf{hom}}(\dot{v}(t))) \geq \ex{-D\mathcal{R}_{\mathsf{hom}}(\dot{v}(t)),v(t)}_{\sB_0^*,\sB_0}$ (for a.e. $t$) by the definition of $\widetilde{\cE}_{\mathsf{hom}}^*$. As a result of this and of the above inequality, it follows that for a.e. $t$ it holds
\begin{equation}\label{help:1}
\widetilde{\cE}_{\mathsf{hom}}(t,v(t))+\widetilde{\cE}^*_{\mathsf{hom}}(t,-D\mathcal{R}_{\mathsf{hom}}(\dot{v}(t))) = \ex{-D\mathcal{R}_{\mathsf{hom}}(\dot{v}(t)),v(t)}_{\sB_0^*,\sB_0}.
\end{equation}
Consequently, we obtain (using standard convex analysis arguments)
\begin{equation*}
\widetilde{\cE}_{\mathsf{hom}}(t,v(t)) \leq \widetilde{\cE}_{\mathsf{hom}}(t,w)+\ex{-D\mathcal{R}_{\mathsf{hom}}(\dot{v}(t)),v(t)-w}_{\sB_0^*,\sB_0} \quad \text{for all }w\in \sB_0.
\end{equation*}
Using the above inequality and similar reasoning as in Step 2, we obtain that $u$ satisfies (\ref{diff:incl:2}).

Moreover, note that using (\ref{eq:920}) and the strong convergence of the initial data, we obtain
\begin{equation*}
\limsup_{\varepsilon\to 0} \brac{-\mathcal{R}\e(v\e(T))+\mathcal{R}\e(v\e(0))}\leq -\mathcal{R}_{\mathsf{hom}}(v(T))+\mathcal{R}_{\mathsf{hom}}(v(0)).
\end{equation*}
Also, exploiting the inequality (\ref{main:inequality:189}) and the liminf inequalities (\ref{eq:928}) and (\ref{second:term:convergence}), we obtain
\begin{align*}
\liminf_{\varepsilon\to 0}\brac{-\mathcal{R}\e(v\e(T))+\mathcal{R}\e(v\e(0))} \geq \int_{0}^T \widetilde{\cE}_{\mathsf{hom}}(t,v(t))+\widetilde{\cE}^*_{\mathsf{hom}}(t,-D\mathcal{R}_{\mathsf{hom}}(\dot{v}(t))) dt.
\end{align*} 
Finally, the above and (\ref{help:1}) imply that $\liminf_{\varepsilon\to 0}\brac{-\mathcal{R}\e(v\e(T))+\mathcal{R}\e(v\e(0))} \geq -\mathcal{R}_{\mathsf{hom}}(v(T))+\mathcal{R}_{\mathsf{hom}}(v(0))$ and therefore
\begin{equation*}
\lim_{\varepsilon \to 0}\brac{-\mathcal{R}\e(v\e(T))+\mathcal{R}\e(v\e(0))}= -\mathcal{R}_{\mathsf{hom}}(v(T))+\mathcal{R}_{\mathsf{hom}}(v(0)).
\end{equation*}
Furthermore, since $v\e(0)=u\e(0)\to u(0)=v(0)$ strongly in $\sB$, it follows that $\mathcal{R}\e(v\e(0))\to \mathcal{R}_{\mathsf{hom}}(v(0))$ and thus $\mathcal{R}\e(v\e(T))\to \mathcal{R}_{\mathsf{hom}}(v(T))$. This yields that $v\e(T)\overset{2}{\to} v(T)$ strongly in $\sB$ and using that $P_{\mathsf{inv}}v(T)=v(T)$ it follows that $v\e(T) \to v(T)$ strongly in $\sB$. Consequently, we obtain $u\e(T)\to u(T)$ strongly in $\sB$. The above procedure can be repeated with $T$ replaced by an arbitrary $t\in (0,T)$, hence we obtain that $u\e(t)\to u(t)$ strongly in $\sB$ for each $t\in [0,T]$.
\medskip

\textit{Step 4. Convergence of $\dot{u}\e$ and $\cE\e(u\e(t))$.}

We test (\ref{diff_incl}) with $\dot{u}\e$ and with the help of the chain rule for $\cE\e$ we obtain
\begin{equation*}
\ex{D\mathcal{R}\e(\dot{u}\e(t)),\dot{u}\e(t)}_{\sB^{*},\sB}=- \frac{d}{dt}\cE\e(u\e(t)).
\end{equation*}
For an arbitrary $t\in (0,T]$, we integrate the above equality on the interval $(0,t)$ to obtain
\begin{equation*}
\int_{0}^t \ex{D\mathcal{R}\e(\dot{u}\e(s)),\dot{u}\e(s)}_{\sB^{*},\sB}ds = \cE\e(u\e(0))-\cE\e(u\e(t)). 
\end{equation*}
Since $u\e(t)\to u(t)$ strongly in $\sB$, we obtain that $\liminf_{\varepsilon\to 0}\cE\e(u\e(t))\geq \cE_{\mathsf{hom}}(u(t))$  using the usual two-scale convergence arguments for the first (quadratic) part of the energy and strong convergence of $\unf u\e$ for the second (non-convex) part. As a consequence, using the convergence $\cE\e(u\e(0))\to \cE_{\mathsf{hom}}(u(0))$, we obtain 
\begin{equation*}
\limsup_{\varepsilon\to 0}\int_{0}^t \ex{D\mathcal{R}\e(\dot{u}\e(s)),\dot{u}\e(s)}_{\sB^{*},\sB}ds \leq \cE_{\mathsf{hom}}(u(0))-\cE_{\mathsf{hom}}(u(t)) = \int_{0}^t \ex{D\mathcal{R}_{\mathsf{hom}}(\dot{u}(s)),\dot{u}(s)}_{\sB^{*}_0,\sB_0}, 
\end{equation*}
where in the last equality we use that $u$ satisfies (\ref{diff:incl:2}) and the chain rule for $\cE_{\mathsf{hom}}$. Note that $\int_{0}^t\ex{D\mathcal{R}\e(\dot{u}\e(s)),\dot{u}\e(s)}_{\sB^{*},\sB} = \int_0^t\ex{\int_{Q} r |\unf \dot{u}\e(s)|^2 }ds$ and since $\unf \dot{u}\e \rightharpoonup \dot{u}$ weakly in $L^2(0,T; \sB)$, it follows that
\begin{equation*}
\liminf_{\varepsilon\to 0} \int_{0}^t\ex{D\mathcal{R}\e(\dot{u}\e(s)),\dot{u}\e(s)}_{\sB^{*},\sB} \geq \int_{0}^t \ex{D\mathcal{R}_{\mathsf{hom}}(\dot{u}(s)),\dot{u}(s)}_{\sB^{*}_0,\sB_0}.
\end{equation*}
Combining the last two inequalities (and the weak convergence $\unf \dot{u}\e \rightharpoonup \dot{u}$), we conclude that for any $t\in(0,T]$
\begin{equation*}
\dot{u}\e \to \dot{u} \text{ strongly in }L^2(0,t; \sB), \quad \cE\e(u\e(t))\to \cE_{\mathsf{hom}}(u(t)).
\end{equation*} 
Note that all of the above results hold for a subsequence of $\cb{\varepsilon}$, however using the uniqueness property of the solution of the limit problem we conclude (using a standard contradiction argument) that all the convergence statements hold for the entire sequence $\cb{\varepsilon}$. This concludes the proof.
\end{proof}
\begin{proof}[Proof of Proposition \ref{proposition:818}]
Note that using the $\Gamma$-convergence in $(C3)$ and the properties of $\cE\e^{\omega}$ and $\cE\e$, it follows that $\cE_{\mathsf{hom}}$ and $\widetilde{\cE}_{\mathsf{hom}}$ are convex, proper and l.s.c. functionals. To prove the claim of the proposition it is sufficient to show that $\widetilde{\cE}_{\mathsf{hom}}^*=\cE_{\mathsf{hom}}^*$ since by Proposition 2 in \cite{rockafellar1971} $\widetilde{\cE}_{\mathsf{hom}}^{**} = \widetilde{\cE}_{\mathsf{hom}}$ and ${\cE}_{\mathsf{hom}}^{**}={\cE}_{\mathsf{hom}}$. Note that $\widetilde{\cE}_{\mathsf{hom}}^*: L^q(Q)\to \R \cup \cb{\infty}$ is the Legendre-Fenchel conjugate of the functional $\widetilde{\cE}_{\mathsf{hom}}$ defined by $\widetilde{\cE}_{\mathsf{hom}}^*(f)=\sup_{u} \brac{\int_{Q}fu-\widetilde{\cE}_{\mathsf{hom}}(u)}$ (analogously we define $\brac{\cE\e^{\omega}}^*, \cE\e^{*}$ and $\cE_{\mathsf{hom}}^*$).

According to Theorem 2 in \cite{rockafellar1971}, it holds that for any $f\in L^q(Q)$,
\begin{equation}\label{equation:117}
\ex{\brac{\cE\e^{\omega}}^{*}(f)}= {\cE}\e^{*}(f).
\end{equation}
$\cE_{\mathsf{hom}}^*=\widetilde{\cE}_{\mathsf{hom}}^*$ follows by passing to the limit $\varepsilon \to 0$ in the above equality and using the following: \textit{(a)} $\ex{\brac{\cE\e^{\omega}}^*(f)} \to \cE_{\mathsf{hom}}^*(f)$, \textit{(b)} $\cE^*\e(f)\to \widetilde{\cE}_{\mathsf{hom}}^*(f)$ as $\varepsilon\to 0$. 

In the following we show only \textit{(a)}, and \textit{(b)} follows similarly (cf. proof of Corollary \ref{C:thm1}). For an arbitrary $u\in L^p(Q)$, using the growth assumption in $(C2)$, it follows that
\begin{equation*}
\cE\e^{\omega}(u)-\int_{Q}fu \leq \brac{\frac{C}{p}+1}\cE\e^{\omega}(u)+\frac{C^2}{p}+\frac{1}{q}{\|f\|^q_{L^q(Q)}}.
\end{equation*}
As a result of this and the assumption $\inf_{u}\cE\e^{\omega}(u)\leq \psi(\omega)$, it follows that 
\begin{equation*}
\inf_{u}\brac{\cE\e^{\omega}(u)-\int_{Q}fu} \leq \brac{\frac{C}{p}+1}\psi(\omega)+\frac{C^2}{p}+\frac{1}{q}{\|f\|^q_{L^q(Q)}}.
\end{equation*}
Note that for $P$-a.e. $\omega\in \Omega$, $\psi(\omega)<\infty$, thus the above inequality and the $\Gamma$-convergence $\cE\e^{\omega}\overset{\Gamma}{\to}\cE_{\mathsf{hom}}$ imply that for $P$-a.e. $\omega\in \Omega$ (using a standard $\Gamma$-convergence argument)
\begin{equation*}
-\brac{\cE\e^{\omega}}^*(f)=\inf_{u}\brac{\cE\e^{\omega}(u)-\int_{Q}f u} \to \inf_{u}\brac{\cE_{\mathsf{hom}}(u)-\int_{Q}f u}= -\cE_{\mathsf{hom}}^*(f).
\end{equation*}
Consequently, the dominated convergence theorem implies \textit{(a)}.
\end{proof}
\section{Quenched stochastic two-scale convergence and relation to stochastic unfolding}\label{Section_4}

In this section, we recall the concept of quenched stochastic two-scale convergence (cf.~\cite{Zhikov2006,heida2011extension})  and study its relation to stochastic unfolding. The notion of quenched stochastic two-scale convergence is based on the individual ergodic theorem, see Theorem~\ref{thm:ergodic-thm}. We thus assume throughout this section that 
\begin{equation*}
  (\Omega,\mathcal F,P,\tau)\text{ satisfies Assumption~\ref{Assumption_2_1} and $P$ is ergodic.}
\end{equation*}
Moreover, throughout this section we fix exponents $p\in(1,\infty)$, $q:=\frac{p}{p-1}$, and an open and bounded domain $Q\subset\R^d$. We denote by $(\sB^p, \|\cdot\|_{\sB^p})$ the Banach space $L^p(\Omega\times Q)$ and the associated norm, and we write $(\sB^p)^*$ for the dual space. For the definition of quenched two-scale convergence we need to specify a suitable space of test-functions in $\sB^q$ that is countably generated. To that end we fix sets $\sD_\Omega$ and $\sD_Q$ such that
\begin{itemize}
\item $\sD_\Omega$ is a countable set of bounded, measurable functions on $(\Omega,\mathcal F)$ that contains the identity $\mathbf 1_{\Omega}\equiv 1$ and is dense in $L^1(\Omega)$ (and thus in $L^r(\Omega)$ for any $1\leq r<\infty$).
\item $\sD_Q\subset C(\overline Q)$ is a countable set that contains the identity $\mathbf 1_Q\equiv 1$ and is dense in $L^1(Q)$ (and thus in $L^r(Q)$ for any $1\leq r<\infty$).
\end{itemize}
We denote by $\sA:=\{\varphi(\omega,x)=\varphi_\Omega(\omega)\varphi_Q(x)\,:\,\varphi_\Omega\in\sD_\Omega,\varphi_Q\in\sD_Q\}$ the set of simple tensor products (a countable set), and by $\sD_0$ the $\mathbb Q$-linear span of $\sA$, i.e.~
\begin{equation*}
  \sD_0:=\big\{\,\sum_{j=1}^m\lambda_j\varphi_j\,:\,m\in\N,\,\lambda_1,\ldots,\lambda_m\in\mathbb Q,\,\varphi_1,\ldots,\varphi_m\in\sA\,\big\}.
\end{equation*}
We finally set $\sD:=\mbox{span}\sA=\mbox{span}\sD_0$ and denote by $\overline{\sD}:=\mbox{span}(\sD_Q)$ (the span of $\sD_Q$ seen as a subspace of $\sD$), and note that $\sD$ and $\sD_0$ are dense subsets of $\sB^q$, while the closure of $\overline{\sD}$ in $\sB^q$ is isometrically isomorphic to $L^q(Q)$. Let us anticipate that $\sD$ serves as our space of test-functions for stochastic two-scale convergence. As opposed to two-scale convergence in the mean, ``quenched'' stochastic two-scale convergence is defined relative to a fixed ``admissible'' realization $\omega_0\in\Omega$. Throughout this section we denote by $\Omega_0$ the set of admissible realizations; it is a set of full measure determined by the following lemma:
\begin{lemma}\label{L:admis}
  There exists a measurable set $\Omega_0\subset\Omega$ with $P(\Omega_0)=1$ s.t. for all $\varphi,\varphi'\in\sA $, all $\omega_0\in\Omega_0$, and $r\in\{p,q\}$ we have with $(\unf^*\varphi)(\omega,x):=\varphi(\tau_{\frac{x}{\eps}}\omega,x)$,
  \begin{align*}
    \limsup\limits_{\eps\to 0}\|(\unf^*\varphi)(\omega_0,\cdot)\|_{L^r(Q)}&\leq \|\varphi\|_{\sB^r}\\ \text{and}\qquad \lim\limits_{\eps\to 0}\int_Q\unf^*(\varphi\varphi')(\omega_0,x)dx&=\ex{\int_Q(\varphi\varphi')(\omega_0,x)\,dx}.
  \end{align*}
\end{lemma}
\begin{proof}
  This is a simple consequence of Theorem~\ref{thm:ergodic-thm} and the fact that $\sA$ is countable.
\end{proof}
For the rest of the section $\Omega_0$ is fixed according to Lemma~\ref{L:admis}.
\medskip

\textbf{Structure of Section \ref{Section_4}.} In Section \ref{Section:4:1} we quickly recall the definition of quenched two-scale convergence and its main properties. Section \ref{Section:4:2} is dedicated to the comparison of the notions of quenched and mean stochastic two-scale convergence using Young measures. In the last Section \ref{Section:4:3} we demonstrate that mean homogenization results (e.g., homogenization of convex integral functionals) might be extended to quenched results appealing to some aspects of the theory of quenched two-scale convergence. 

\subsection{Definition and basic properties}\label{Section:4:1}
The idea of quenched stochastic two-scale convergence is similar to periodic two-scale convergence: We associate with a bounded sequence $(u\e)\subset L^p(Q)$ and $\omega_0\in\Omega_0$, a sequence of linear functionals $(U\e)$ defined on $\sD$. We can pass (up to a subsequence) to a pointwise limit $U$, which is again a linear functional on $\sD$ and which (thanks to Lemma~\ref{L:admis}) can be uniquely extended to a bounded linear functional on $\sB^q$. We then define the \textit{weak quenched $\omega_0$-two-scale limit} of $(u\e)$ as the Riesz-representation $u\in \sB^p$ of $U\in(\sB^q)^*$.
\begin{defn}[quenched two-scale limit, cf.~\cite{Zhikov2006,Heida2017b}]
\label{def:two-scale-conv}
Let $(u\e)$ be a sequence in $L^{p}(Q)$, and let $\omega_0\in\Omega_0$ be fixed.  We say that $u\e$ converges (weakly, quenched) $\omega_0$-two-scale to $u\in \sB^{p}$, and write
$u\e\tsq{\omega_0}u$, if the sequence $u\e$ is bounded in $L^p(Q)$, and for all $\varphi\in \sD$ we have
\begin{equation}
  \lim_{\eps\to0}\int_{Q}u\e(x)(\unf^*\varphi)(\omega_0,x)\,dx=\int_\Omega\int_{Q}u(x,\omega)\varphi(\omega,x)\,dx\,dP(\omega).\label{eq:def-quenched-two-scale}
\end{equation}
\end{defn}
\begin{lemma}[Compactness]\label{lem:two-scale-limit}
Let $(u\e)$ be a bounded sequence in $L^p(Q)$ and $\omega_0\in \Omega_0$. Then there exists a subsequence (still denoted by $\eps$) and $u\in \sB^p$ such that $u_{\eps}\tsq{\omega_0}u$ and
    \begin{equation}
      \| u\|_{\sB^{p}}\leq \liminf_{\eps\to0}\|u_{\eps}\|_{L^{p}(Q)},\,\label{eq:two-scale-limit-estimate}
    \end{equation}
    and $u_{\eps}\weakto \ex{u}$ weakly in $L^p(Q)$.
\end{lemma}
{\it (For the proof see Section~\ref{S:4:p1}).}\smallskip

For our purpose it is convenient to have a metric characterization of two-scale convergence.
\begin{lemma}[Metric characterization]\label{L:metric-char}
  \begin{enumerate}[(i)]
  \item 
  Let $\{\varphi_j\}_{j\in\N}$ denote an enumeration of  $\sA_1:=\{\frac{\varphi}{\|\varphi\|_{\sB^q}}\,:\,\varphi\in \sD_0\}$. The vector space $\mbox{\rm Lin}(\sD):=\{U:\sD\to\R\text{ linear}\,\}$ endowed with the metric $$d(U,V;\mbox{\rm Lin}(\sD)):=\sum_{j\in\N}2^{-j}\frac{|U(\varphi_j)-V(\varphi_j)|}{|U(\varphi_j)-V(\varphi_j)|+1}$$  is complete and separable.
\item   Let $\omega_0\in\Omega_0$. Consider the maps
  \begin{align*}
    &J_\eps^{\omega_0}: L^p(Q)\to \mbox{\rm Lin}(\sD),\qquad (J_\eps^{\omega_0} u)(\varphi):=\int_Qu(x)(\unf^*\varphi)(\omega_0,x)\,dx,\\
    &J_0:\sB^p\to \mbox{\rm Lin}(\sD),\qquad (J_0u)(\varphi):=\ex{\int_Qu\varphi}.
  \end{align*}
  Then for any bounded sequence $u_\eps$ in $L^p(Q)$ and any $u\in \sB^p$ we have $u_\eps\tsq{\omega_0}u$ if and only if $J_\eps^{\omega_0} u_\eps \to J_0u$ in $\mbox{\rm Lin}(\sD)$.
\end{enumerate}
\end{lemma}
{\it (For the proof see Section~\ref{S:4:p1}).}\smallskip
\begin{remark}
  Convergence in the metric space $(\mbox{\rm Lin}(\sD),d(\cdot,\cdot,\mbox{\rm Lin}(\sD))$ is equivalent to pointwise convergence. $(\sB^q)^*$ is naturally embedded into the metric space by means of the restriction $J:(\sB^q)^*\to\mbox{\rm Lin}(\sD)$, $JU=U\vert_{\sD}$. In particular, we deduce that for a bounded sequences $(U_k)$ in $(\sB^q)^*$ we have $U_k\stackrel{*}{\weakto} U$ if and only if $JU_k\to JU$ in the metric space. Likewise, $\sB^p$  (resp. $L^p(Q)$) can be embedded into the metric space $\mbox{\rm Lin}(\sD)$ via $J_0$ (resp. $J_\eps^{\omega_0}$ with $\eps>0$ and $\omega_0\in\Omega_0$ arbitrary but fixed), and for a bounded sequence $(u_k)$ in $\sB^p$ (resp. $L^p(Q)$) weak convergence in $\sB^p$ (resp. $L^p(Q)$) is equivalent to convergence of $(J_0u_k)$ (resp. $(J^{\omega_0}_\eps u_k)$) in the metric space.
\end{remark}
\begin{lemma}[Strong convergence implies quenched two-scale convergence]\label{L:strong}
Let $(u\e)$ be a strongly convergent sequence in $L^p(Q)$ with limit $u\in L^p(Q)$. Then for all $\omega_0\in\Omega_0$ we have $u\e\tsq{\omega_0}u$.  
\end{lemma}
{\it (For the proof see Section~\ref{S:4:p1}).}\smallskip

\begin{defn}[set of quenched two-scale cluster points]
  For a bounded sequence $(u\e)$ in $L^p(Q)$ and $\omega_0\in\Omega_0$ we denote by $\CL(\omega_0,(u\e))$ the set of all $\omega_0$-two-scale cluster points, i.e. the set of $u\in\sB^p$ with  $J_0u\in \bigcap_{k=1}^\infty\overline{\big\{J^{\omega_0}_{\eps} u_\eps\,:\,\eps<\frac{1}{k}\big\}}$  where the closure is taken in the metric space $\big(\mbox{\rm Lin}(\sD),d(\cdot,\cdot;\mbox{\rm Lin}(\sD)))$.
\end{defn}
We conclude this section with two elementary results on quenched stochastic two-scale convergence:
\begin{lemma}[Approximation of two-scale limits]
\label{lem:Every-v-is-a-fB-limit}Let $u\in\sB^p$.
Then for all $\omega_0\in\Omega_0$, there exists a sequence
$u\e\in L^{p}(Q)$ such that $u\e\stackrel{2s}{\weakto}_{\omega_0} u$ as $\eps\to0$.
\end{lemma}
{\it (For the proof see Section~\ref{S:4:p1}).}\smallskip

Similar to the slightly different setting in \cite{Heida2017b}
one can prove the following result:
\begin{lemma}[Two-scale limits of gradients]
\label{lem:sto-conver-grad}
Let $(u_{\eps})$ be a sequence in $W^{1,p}(Q)$ and $\omega_0\in\Omega_0$. Then there exist a subsequence (still denoted by $\eps$) and functions $u\in W^{1,p}(Q)$ and $\chi\in L^p_{\textsf{pot}}(\Omega)\otimes L^{p}(Q)$ such that $u\e\weakto u$ weakly in $W^{1,p}(Q)$ and
\[
u_{\eps}\tsq{\omega_0}u\quad\mbox{and }\quad\nabla u_{\eps}\tsq{\omega_0}\nabla u+\chi\qquad\mbox{as }\eps\to0\,.
\]
\end{lemma}
\subsubsection{Proofs}\label{S:4:p1}
\begin{proof}[Proof of Lemma~\ref{lem:two-scale-limit}]
  Set $C_0:=\limsup\limits_{\eps\to 0}\|u\e\|_{L^p(Q)}$ and note that $C_0<\infty$. By passing to a subsequence (not relabeled) we may assume that $C_0=\liminf\limits_{\eps\to 0}\|u\e\|_{L^p(Q)}$.
Fix $\omega_0\in\Omega_0$. Define linear functionals $U_\eps\in\mbox{\rm Lin}(\sD)$ via
    \begin{equation*}
      U_\eps(\varphi):=\int_Qu\e(x)(\unf^*\varphi)(\omega_0,x)\,dx.
    \end{equation*}
    Note that for all $\varphi\in\sA$, $(U\e(\varphi))$ is a bounded sequence in $\R$. Indeed, by H\"older's inequality and Lemma~\ref{L:admis},
    \begin{equation}\label{eq:x1}
      \limsup\limits_{\eps\to0}|U\e(\varphi)|\leq      \limsup\limits_{\eps\to0}\|u\e\|_{L^p(Q)}\|\unf^*\varphi(\omega_0,\cdot)\|_{L^q(Q)}\leq C_0\|\varphi\|_{\sB^q}.
    \end{equation}
    Since $\sA$ is countable we can pass to a subsequence (not relabeled) such that $U\e(\varphi)$ converges for all $\varphi\in\sA$. By linearity and since $\sD=\mbox{span}(\sA)$, we conclude that $U\e(\varphi)$ converges for all $\varphi\in\sD$, and $U(\varphi):=\lim\limits_{\eps\to0}U\e(\varphi)$ defines a linear functional on $\sD$. In view of \eqref{eq:x1} we have $|U(\varphi)|\leq C_0\|\varphi\|_{\sB^q}$, and thus $U$ admits a unique extension to a linear functional in $(\sB^q)^*$. Let $u\in\sB^p$ denote its Riesz-representation. Then $u\e\tsq{\omega_0} u$, and
    \begin{equation*}
      \|u\|_{\sB^p}=\|U\|_{(\sB^q)^*}\leq C_0=\liminf\limits_{\eps\to 0}\|u\e\|_{L^p(Q)}.
    \end{equation*}
    Since $\mathbf 1_{\Omega}\in\sD_\Omega$ we conclude that for all $\varphi_Q\in\sD_Q$ we have
    \begin{equation*}
      \int_Qu\e(x)\varphi_Q(x)\,dx=U\e(\mathbf 1_{\Omega}\varphi_Q)\to U(\mathbf 1_{\Omega}\varphi_Q)=\ex{\int_Q u(\omega,x)\varphi_Q(x)\,dx}=\int_Q \ex{u(x)}\varphi_Q(x)\,dx.
    \end{equation*}
    Since $(u\e)$ is bounded in $L^p(Q)$, and $\sD_Q\subset L^p(Q)$ is dense, we conclude that $u\e\weakto\ex{u}$ weakly in $L^p(Q)$.
\end{proof}

\begin{proof}[Proof of Lemma~\ref{L:metric-char}]
  \begin{enumerate}[(i)]
  \item Argument for completeness: If $(U_j)$ is a Cauchy sequence in $\mbox{\rm Lin}(\sD)$, then for all $\varphi\in\sA_1$, 
    $(U_j(\varphi))$ is a Cauchy sequence in $\R$. By linearity of the $U_j$'s this implies that $(U_j(\varphi))$ is Cauchy in $\R$ for all $\varphi\in\sD$. Hence, $U_j\to U$ pointwise in $\sD$ and it is easy to check that $U$ is linear. Furthermore, $U_j\to U$ pointwise in $\sA_1$ implies $U_j\to U$ in the metric space.
    
Argument for separability: Consider the (injective) map $J:(\sB^q)^*\to\mbox{\rm Lin}(\sD)$ where $J(U)$ denotes the restriction of $U$ to $\sD$. The map $J$ is continuous, since for all $U,V\in(\sB^q)^*$ and $\varphi\in\sA_1$ we have $|(JU)(\varphi)-(JV)(\varphi)|\leq \|U-V\|_{(\sB^q)^*}\|\varphi\|_{\sB^q}=\|U-V\|_{(\sB^q)^*}$ (recall that the test functions in $\sA_1$ are normalized). Since $(\sB^q)^*$ is separable (as a consequence of the assumption that $\mathcal F$ is countably generated), it suffices to show that the range $\mathcal R(J)$ of $J$ is dense in $\mbox{\rm Lin}(\sD)$. To that end let $U\in\mbox{\rm Lin}(\sD)$. For $k\in\N$ we denote by $U_k\in(\sB^q)^*$ the unique linear functional that is equal to $U$ on the the finite dimensional (and thus closed) subspace $\mbox{span}\{\varphi_1,\ldots,\varphi_k\}\subset \sB^q$ (where $\{\varphi_j\}$ denotes the enumeration of $\sA_1$), and zero on the orthogonal complement in $\sB^q$. Then a direct calculation shows that $d(U,J(U_k);\mbox{\rm Lin}(\sD))\leq\sum_{j>k}2^{-j}=2^{-k}$. Since $k\in\N$ is arbitrary, we conclude that $\mathcal R(J)\subset\mbox{\rm Lin}(\sD)$ is dense.
  \item Let $u\e$ denote a bounded sequence in $L^p(Q)$ and $u\in \sB^p$. Then by definition, $u\e\tsq{\omega_0}u$ is equivalent to $J^{\omega_0}\e u\e\to J_0u$ pointwise in $\sD$, and the latter is equivalent to convergence in the metric space $\mbox{\rm Lin}(\sD)$.
  \end{enumerate}
\end{proof}

\begin{proof}[Lemma~\ref{L:strong}]
This follows from H{\"o}lder's inequality and Lemma~\ref{L:admis}, which imply for all $\varphi\in\sA$ the estimate
\begin{multline*}
  \limsup\limits_{\eps\to 0}\int_Q|(u\e(x)-u(x))\unf^*\varphi(\omega_0,x)|\,dx\\ 
	\leq    \limsup\limits_{\eps\to 0}\Big(\|u\e-u\|_{L^p(Q)}\left(\int_Q|\unf^*\varphi(\omega_0,x)|^q\,dx\right)^\frac1q\Big)=0.
\end{multline*}
\end{proof}

\begin{proof}[Proof of Lemma~\ref{lem:Every-v-is-a-fB-limit}]
  Since $\sD$ (defined as in Lemma~\ref{L:metric-char}) is dense in $\sB^p$, for any $\delta>0$ there exists $v_\delta\in\sD_0$ with $\|u-v_\delta\|_{\sB^p}\leq \delta$. Define $v_{\delta,\eps}(x):=\unf^*v_\delta(\omega_0,x)$. Let $\varphi\in\sD$. Since $v_\delta$ and $\varphi$ (resp. $v_\delta\varphi$) are by definition linear combinations of functions (resp. products of functions) in $\sA$, we deduce from Lemma~\ref{L:admis} that $(v_{\delta,\eps})_{\eps}$ is bounded in $L^p(Q)$, and that 
  \begin{equation*}
    \int_Qv_{\delta,\eps}\unf^*\varphi(\omega_0,x)=\int_Q\unf^*(v_\delta\varphi)(\omega_0,x)\to \ex{\int_Qv_\delta\varphi}.
  \end{equation*}
  By appealing to the metric characterization, we can rephrase the last convergence statement as $d(J^{\omega_0}\e v_{\delta,\eps},J_0v_\delta;\mbox{\rm Lin}(\sD))\to 0$. By the triangle inequality we have
  \begin{eqnarray*}
    d(J^{\omega_0}\e v_{\delta,\eps},J_0u;\mbox{\rm Lin}(\sD))&\leq& d(J^{\omega_0}\e v_{\delta,\eps},J_0v_\delta;\mbox{\rm Lin}(\sD))+d(J_0v_\delta,J_0u;\mbox{\rm Lin}(\sD)).
  \end{eqnarray*}
  The second term is bounded by $\|v_\delta-u\|_{\sB^p}\leq \delta$, while the first term vanishes for $\eps\downarrow 0$. Hence, there exists a diagonal sequence $u\e:=v_{\delta(\eps),\eps}$ (bounded in $L^p(Q)$) such that there holds $d(J^{\omega_0}\e u_{\eps},J_0u;\mbox{\rm Lin}(\sD))\to 0$. The latter implies $u_{\eps}\tsq{\omega_0}u$ by Lemma~\ref{L:metric-char}.
\end{proof}

\subsection{Comparison to stochastic two-scale convergence in the mean via Young measures}\label{Section:4:2}
In this paragraph we establish a relation between quenched two-scale
convergence and two-scale convergence in the mean (in the sense of Definition
\ref{def46}). The relation is established by Young measures: We show that any bounded sequence $u\e$ in $\sB^p$ -- viewed as a functional acting on test-functions of the form $\unf^*\varphi$ -- generates (up to a subsequence) a Young measure on $\sB^p$ that (a) concentrates on the quenched two-scale cluster points of $u\e$, and (b) allows to represent the two-scale limit (in the mean) of $u\e$. 
\begin{defn}
  We say $\boldsymbol{\nu}:=\left\{ \nu_{\omega}\right\} _{\omega\in\Omega}$ is a Young measure on $\sB^p$, if for all $\omega\in\Omega$, $\nu_\omega$ is a Borel probability measure on $\sB^p$  (equipped with the weak topology) and 
  \[
  \omega\mapsto\nu_{\omega}(B)\quad\mbox{is }\mbox{measurable for all }B\in\cB(\sB^p),
  \]
  where $\cB(\sB^p)$ denotes the Borel-$\sigma$-algebra on $\sB^p$  (equipped with the weak topology).
\end{defn}
\begin{thm}
\label{thm:Balder-Thm-two-scale}Let $u_{\eps}$ denote a bounded sequence in $\sB^p$.
Then there exists a subsequence (still denoted by $\eps$) and a Young measure  $\boldsymbol{\nu}$ on $\sB^p$
such that for all $\omega_0\in\Omega_0$,
\[
\nu_{\omega_0}\mbox{ is concentrated on }\CL\left(\omega_0,\big(u_{\eps}(\omega_0,\cdot)\big)\right),
  \]
  and 
  \[
  \liminf_{\eps\to0}\Vert u_{\eps}\Vert_{\sB^p}^{p}\geq \int_{\Omega}\left(\int_{\sB^p}\left\Vert v\right\Vert _{\sB^p}^{p}\,d\nu_{\omega}(v)\right)\,dP(\omega).
  \]
  Moreover, we have 
  \[
  u_{\eps}\ts u\qquad\text{where }u:=\int_{\Omega}\int_{\sB^p}v\, d\nu_{\omega}(v)dP(\omega).
  \]
  Finally, if there exists $v:\Omega\to\sB^p$ measurable and $\nu_{\omega}=\delta_{v(\omega)}$ for $P$-a.e. $\omega\in\Omega$,
  then up to extraction of a further subsequence (still denoted by $\eps$) we have
  \[
  u_{\eps}(\omega)\tsq{\omega}v(\omega)\qquad\text{for $P$-a.e.~$\omega\in\Omega$}.
  \]
\end{thm}
{\it (For the proof see Section~\ref{S:4:p2}).}\smallskip

In the opposite direction we observe that quenched two-scale convergence implies two-scale convergence in the mean in the following sense:
  \begin{lemma}\label{L:fromquenchedtomean}
    Consider a family $\{(u\e^\omega)\}_{\omega\in\Omega}$ of sequences $(u^\omega\e)$ in $L^p(Q)$ and suppose that:
    \begin{enumerate}[(a)]
    \item There exists $u\in\sB^p$ s.t.~for $P$-a.e. $\omega\in\Omega$ we have $u\e^\omega\tsq{\omega}u$.
    \item There exists a sequence $(\tilde u\e)$ s.t.~$u\e^\omega(x)=\tilde u\e(\omega,x)$ for a.e.~$(\omega,x)\in\Omega\times Q$.
    \item There exists a bounded sequence $(\chi\e)$ in $L^p(\Omega)$ such that $\|u^\omega\e\|_{L^p(Q)}\leq\chi\e(\omega)$ for a.e.~$\omega\in\Omega$.
    \end{enumerate}
    Then $\tilde u\e\wt u$ weakly two-scale (in the mean).
  \end{lemma}
{\it (For the proof see Section~\ref{S:4:p2}).}\smallskip

To compare homogenization of convex integral functionals w.r.t.~stochastic two-scale convergence in the mean and in the quenched sense, we appeal to the following result.
\begin{defn}[Quenched two-scale normal integrand]\label{D:qtscnormal}
A function $h:\,\Omega\times Q\times\R^{d}\to\R$ is called a quenched two-scale normal integrand, if for all $\xi\in\R^d$, $h(\cdot,\cdot,\xi)$ is $\mathcal F\otimes\cB(\R^{d})$-measurable, and for a.e.~$(\omega,x)\in\Omega\times Q$, $h(\omega,x,\cdot)$ is lower semicontinuous, and for $P$-a.e.~$\omega_0\in\Omega_0$ and sequence $(u\e)$ in $L^p(Q)$ the following implication holds:
\[
u\e\tsq{\omega_0}u\qquad\Rightarrow\qquad
\liminf_{\eps\to0}\int_{Q}h(\tau_{\frac{x}{\eps}}\omega_0,x,u\e(x))dx\geq\int_{\Omega}\int_{Q}h(\omega,x,u(\omega,x))\,dx\,dP(\omega).
\]
\end{defn}
\begin{lemma}
\label{lem:Balder-Lem-two-scale}
Let $h$ denote a quenched two-scale normal integrand, let $(u\e)$ denote a bounded sequence in $\sB^p$ that generates a Young measure $\boldsymbol{\nu}$ on $\sB^p$ in the sense of Theorem \ref{thm:Balder-Thm-two-scale}. 
Suppose that $h\e:\Omega\to\R$, $h\e(\omega):=-\int_Q\min\big\{0,h(\tau_{\frac{x}{\eps}}\omega,x,u_{\eps}(\omega,x))\big\}\,dx$ is uniformly integrable. Then 
\begin{multline}
  \liminf_{\eps\to 0}\int_{\Omega}\int_{Q}h(\tau_{\frac{x}{\eps}}\omega,x,u_{\eps}(\omega,x))\,dx\,dP(\omega)\\ 
	\geq\int_{\Omega}\int_{\sB^p}\left(\int_\Omega\int_Qh(\tomega,x,v(\tomega,x))\,dx\,dP(\tomega)\right)\,d\nu_{\omega}(v)\,dP(\omega).\label{eq:liminf-balder-ts-lower-semic}
\end{multline}
\end{lemma}
{\it (For the proof see Section~\ref{S:4:p2}).}

\subsubsection{Proof of Theorem~\ref{thm:Balder-Thm-two-scale} and Lemmas~\ref{lem:Balder-Lem-two-scale} and \ref{L:fromquenchedtomean}}\label{S:4:p2}
We first recall some notions and results of Balder's theory for Young measures \cite{Balder1984}. Throughout this section $\sM$ is assumed to be a separable, complete  metric space with metric $d(\cdot,\cdot;\sM)$. 

\begin{defn}
  \begin{itemize}
  \item We say a function $s:\Omega\to\sM$ is measurable, if it is $\mathcal F-\mathcal B(\sM)$-measurable where $\mathcal B(\sM)$ denotes the Borel-$\sigma$-algebra in $\sM$.
  \item A function $h:\Omega\times\sM\to(-\infty,+\infty]$ is called a \textit{normal} integrand, if $h$ is $\mathcal F\otimes\mathcal B(\sM)$-measurable, and for all $\omega\in\Omega$ the function $h(\omega,\cdot):\sM\to(-\infty,+\infty]$ is lower semicontinuous.
  \item A sequence $s\e$ of measurable functions $s\e:\Omega\to\sM$ is called \textit{tight}, if there exists a normal integrand $h$ such that for every $\omega\in\Omega$ the function $h(\omega,\cdot)$ has compact sublevels in $\sM$ and $\limsup_{\eps\to 0}\int_\Omega h(\omega,s\e(\omega))\,dP(\omega)<\infty$.
  \item A Young measure in $\sM$ is a family $\boldsymbol{\mu}:=\left\{ \mu_{\omega}\right\} _{\omega\in\Omega}$
    of Borel probability measures on $\sM$ such that for all $B\in\mathcal B(\sM)$ the map $\Omega\ni \omega\mapsto \mu_\omega(B)\in\R$ is $\mathcal F$-measurable.
  \end{itemize}
\end{defn}
\begin{thm}[\mbox{\cite[Theorem I]{Balder1984}}]\label{thm:Balder} Let $s\e:\,\Omega\to\sM$ denote a tight sequence of measurable functions. Then there exists a subsequence, still indexed by $\eps$, and a Young measure ${\boldsymbol\mu}:\Omega\to\sM$ such that for every normal integrand $h:\,\Omega\times \sM\rightarrow(-\infty,+\infty]$ we have
\begin{equation}\label{eq:Balders-ineq}
\liminf_{\eps\to0}\int_{\Omega}h(\omega,s\e(\omega))\,dP(\omega)\geq\int_{\Omega}\int_{\sM}h(\omega,\xi) d\mu_{\omega}(\xi)dP(\omega)\,,
\end{equation}
provided that the negative part $h^-_\eps(\cdot)=|\min\{0,h(\cdot,s\e(\cdot))\}|$ is uniformly integrable.
Moreover, for $P$-a.e. $\omega\in\Omega_0$ the measure $\mu_\omega$ is supported in the set of all cluster points of $s\e(\omega)$, i.e.~in $\bigcup_{k=1}^\infty\overline{\{s\e(\omega)\,:\,\eps<\frac{1}{k}\}}$ (where the closure is taken in $\sM$).
\end{thm}
In order to apply the above theorem we require an appropriate  metric space in which two-scale convergent sequences and their limits embed:
\begin{lemma}\label{L:metric-struct}
  \begin{enumerate}[(i)]
  \item   We denote by $\sM$ the set of all triples $(U,\eps,r)$ with $U\in\mbox{\rm Lin}(\sD)$, $\eps\geq 0$, $r\geq 0$.  $\sM$ endowed with the metric
  \begin{equation*}
    d((U_1,\eps_1,r_1),(U_2,\eps_2,r_2);\sM):=d(U_1,U_2;\mbox{\rm Lin}(\sD))+|\eps_1-\eps_2|+|r_1-r_2|
  \end{equation*}
  is a complete, separable metric space. 
\item For $\omega_0\in\Omega_0$ we denote by $\sM^{\omega_0}$ the set of all triples $(U,\eps,r)\in\sM$ such that
  \begin{equation}\label{eq:repr}
    U=
    \begin{cases}
      J^{\omega_0}_\eps u&\text{for some }u\in L^p(Q)\text{ with }\|u\|_{L^p(Q)}\leq r\text{ in the case }\eps>0,\\
      J_0 u&\text{for some }u\in\sB^p\text{ with }\|u\|_{\sB^p}\leq r\text{ in the case }\eps=0.
    \end{cases}
  \end{equation}
  Then $\sM^{\omega_0}$ is a closed subspace of $\sM$.
\item Let $\omega_0\in\Omega_0$, and $(U,\eps,r)\in\sM^{\omega_0}$. Then the function $u$ in the representation \eqref{eq:repr} of $U$ is unique, and
  \begin{equation}\label{eq:x6}
    \begin{cases}\displaystyle
      \|u\|_{L^p(Q)}=\sup\limits_{\varphi\in\overline{\sD},\ \|\varphi\|_{\sB^q}\leq 1}|U(\varphi)|&\text{if }\eps>0,\\
      \|u\|_{\sB^p}=\sup\limits_{\varphi\in \sD,\ \|\varphi\|_{\sB^q}\leq 1}|U(\varphi)|&\text{if }\eps=0.
    \end{cases}
  \end{equation}
\item For $\omega_0\in\Omega_0$ the function $\|\cdot\|_{\omega_0}:\sM^{\omega_0}\to[0,\infty)$, 
  \begin{equation*}
    \|(U,\eps,r)\|_{\omega_0}:=
    \begin{cases}
      \big(\sup\limits_{\varphi\in\overline{\sD},\ \|\varphi\|_{\sB^q}\leq 1}|U(\varphi)|^p+\eps+r^p\big)^{\frac{1}{p}}&\text{if }(U,\eps,r)\in\sM^{\omega_0},\,\eps>0,\\
      \big(\sup\limits_{\varphi\in \sD,\ \|\varphi\|_{\sB^q}\leq 1}|U(\varphi)|^p+r^p\big)^\frac1p&\text{if }(U,\eps,r)\in\sM^{\omega_0},\,\eps=0,\\
    \end{cases}
  \end{equation*}
  is lower semicontinuous on $\sM^{\omega_0}$. 
\item For all $(u,\eps)$ with $u\in L^p(Q)$ and $\eps>0$ we have $s:=(J^{\omega_0}_\eps u,\eps,\|u\|_{L^p(Q)})\in\sM^{\omega_0}$ and $\|s\|_{\omega_0}=\big(2\|u\|_{L^p(Q)}^p+\eps\big)^\frac1p$. Likewise, for all $(u,\eps)$ with $u\in\sB^p$ and $\eps=0$ we have $s=(J_0u,\eps,\|u\|_{\sB^p})$ and $\|s\|_{\omega_0}=2^\frac1p\|u\|_{\sB^p}$.
\item For all $R<\infty$ the set $\{(U,\eps,r)\in\sM^{\omega_0}\,:\,\|(U,\eps,r)\|_{\omega_0}\leq R\}$ is compact in $\sM$.
\item Let $\omega_0\in\Omega_0$ and let $u\e$ denote a bounded sequence in $L^p(Q)$. Then the triple $s\e:=(J^{\omega_0}\e u\e,\eps,\|u_{\varepsilon}\|_{L^p(Q)})$ defines a sequence in $\sM^{\omega_0}$. Moreover, we have $s\e\to s_0$ in $\sM$ as $\eps\to0$ if and only if $s_0=(J_0u,0,r)$ for some $u\in\sB^p$, $r\geq\|u\|_{\sB^p}$, and $u\e\tsq{\omega_0}u$.
  \end{enumerate}
\end{lemma}
\begin{proof}
  \begin{enumerate}[(i)]
  \item This is a direct consequence of Lemma~\ref{L:metric-char} (i) and the fact that the product of complete, separable metric spaces remains complete and separable.
  \item Let $s_k:=(U_k,\eps_k,r_k)$ denote a sequence in $\sM^{\omega_0}$ that converges in $\sM$ to some $s_0=(U_0,\eps_0,r_0)$. We need to show that $s_0\in\sM^{\omega_0}$. By passing to a subsequence, it suffices to study the following three cases: $\eps_k>0$ for all $k\in\N_0$, $\eps_k=0$ for all $k\in\N_0$, and $\eps_0=0$ while $\eps_k>0$ for all $k\in\N$.

    Case 1: $\eps_k>0$ for all $k\in\N_0$.\\
    W.l.o.g.~we may assume that $\inf_k\eps_k>0$. Hence, there exist  $u_k\in L^p(Q)$ with $U_k=J^{\omega_0}_{\eps_k}u_k$. Since $r_k\to r$, we conclude that $(u_k)$ is bounded in $L^p(\Omega)$. We thus may pass to a subsequence (not relabeled) such that $u_k\weakto u_0$ weakly in $L^p(Q)$, and
    \begin{equation}\label{eq:x3}
      \|u_0\|_{L^p(Q)}\leq \liminf\limits_{k}\|u_k\|_{L^p(Q)}\leq \lim_k r_k=r_0.
    \end{equation}
    Moreover, $U_k\to U$ in the metric space $\mbox{\rm Lin}(\sD)$ implies pointwise convergence on $\sD$, and thus for all $\varphi_Q\in\sD_Q$ we have $U_k(\mathbf 1_{\Omega}\varphi_Q)=\int_Qu_k\varphi_Q\to \int_Qu_0\varphi_Q$. We thus conclude that $U_0(\mathbf 1_{\Omega}\varphi_Q)=\int_Q u_0\varphi_Q$. Since $\sD_Q\subset L^q(Q)$ dense, we deduce that $u_k\weakto u_0$ weakly in $L^p(Q)$ for the entire sequence.
    On the other hand the properties of the shift $\tau$ imply that for any $\varphi_\Omega\in\sD_\Omega$ we have $\varphi_\Omega(\tau_{\frac{\cdot}{\eps_k}}\omega_0)\to\varphi_\Omega(\tau_{\frac{\cdot}{\eps_0}}\omega_0)$ in $L^q(Q)$. Hence, for any $\varphi_\Omega\in\sD_\Omega$ and $\varphi_Q\in\sD_Q$ we have
    \begin{equation*}
      U_k(\varphi_\Omega\varphi_Q)=\int_Q u_k(x)\varphi_Q(x)\varphi_\Omega(\tau_{\frac{x}{\eps_k}}\omega_0)\,dx\to
      \int_Q u_0(x)\varphi_Q(x)\varphi_\Omega(\tau_{\frac{x}{\eps_0}}\omega_0)\,dx=J^{\omega_0}_{\eps_0}(\varphi_\Omega\varphi_Q)
    \end{equation*}
    and thus (by linearity) $U_0=J^{\omega_0}_{\eps_0}u_0$.

    Case 2: $\eps_k=0$ for all $k\in\N_0$.\\
    In this case there exist a bounded sequence $u_k$ in $\sB^p$ with $U_k=J_0u_k$ for $k\in\N$. By passing to a subsequence we may assume that $u_k\weakto u_0$ weakly in $\sB^p$ for some $u_0\in\sB^p$ with 
    \begin{equation}\label{eq:x4}
      \|u_0\|_{\sB^p}\leq \liminf\limits_{k}\|u_{\eps_k}\|_{\sB^p}\leq r_0.
    \end{equation}
    This implies that $U_k=J_0u_k\to J_0u_0$ in $\mbox{\rm Lin}(\sD)$. Hence, $U_0=J_0u_0$ and we conclude that $s_0\in\sM$.

    Case 3: $\eps_k>0$ for all $k\in\N$ and $\eps_0=0$.\\
    There exists a bounded sequence $u_k$ in $L^p(Q)$. Thanks to Lemma~\ref{lem:two-scale-limit}, by passing to a subsequence we may assume that $u_k\tsq{\omega_0} u_0$ for some $u\in \sB^p$ with 
    \begin{equation}\label{eq:x5}
      \|u_0\|_{\sB^p}\leq \liminf\limits_{k}\|u_k\|_{L^p(Q)}\leq r_0.
    \end{equation}
    Furthermore, Lemma~\ref{L:metric-char} implies that $J^{\omega_0}_{\eps_k}u_k\to J_0u_0$ in $\mbox{\rm Lin}(\sD)$, and thus $U_0=J_0u_0$. We conclude that $s_0\in\sM$.
\item We first argue that the representation \eqref{eq:repr} of $U$ by $u$ is unique. In the case $\eps>0$ suppose that $u,v\in L^p(Q)$ satisfy $U=J^{\omega_0}_{\eps}u=J^{\omega_0}_{\eps}v$. Then for all $\varphi_Q\in\sD_Q$ we have $\int_Q(u-v)\varphi_Q=J^{\omega_0}_\eps u(\mathbf 1_\Omega\varphi_Q)-J^{\omega_0}_\eps v(\mathbf 1_\Omega\varphi_Q)=U(\mathbf 1_\Omega\varphi_Q)-U(\mathbf 1_\Omega\varphi_Q)=0$, and since $\sD_Q\subset L^q(Q)$ dense, we conclude that $u=v$. In the case $\eps=0$ the statement follows by a similar argument from the fact that $\sD$ is dense $\sB^q$.
  To see \eqref{eq:x6} let $u$ and $U$ be related by \eqref{eq:repr}. Since $\overline\sD$ (resp. $\sD$) is dense in $L^q(Q)$ (resp. $\sB^q$), we have
  \begin{equation*}
    \begin{cases}
      \|u\|_{L^p(Q)}=\sup\limits_{\varphi\in\overline{\sD},\ \|\varphi\|_{\sB^q}\leq 1}|\int_Qu\varphi\,dx\,dP|=\sup\limits_{\varphi\in\overline{\sD},\ \|\varphi\|_{\sB^q}\leq 1}|U(\varphi)|&\text{if }\eps>0,\\
      \|u\|_{\sB^p}=\sup\limits_{\varphi\in\sD,\ \|\varphi\|_{\sB^q}\leq 1}|\int_{\Omega}\int_{Q}u\varphi\,dx\,dP|=\sup\limits_{\varphi\in\sD,\ \|\varphi\|_{\sB^q}\leq 1}|U(\varphi)|&\text{if }\eps=0.
    \end{cases}    
  \end{equation*}

  \item Let $s_k=(U_k,\eps_k,r_k)$ denote a sequence in $\sM^{\omega_0}$ that converges in $\sM$ to a limit $s_0=(U_0,\eps_0,r_0)$. By (ii) we have $s_0\in\sM^{\omega_0}$. For $k\in\N_0$ let $u_k$ in $L^p(Q)$ or $\sB^p$ denote the representation of $U_k$ in the sense of \eqref{eq:repr}. We may pass to a subsequence such that one of the three cases in (ii) applies and (as in (ii)) either $u_k$ weakly converges to $u_0$ (in $L^p(Q)$ or $\sB^p$), or $u_k\tsq{\omega_0}u_0$. In any of these cases the claimed lower semicontinuity of $\|\cdot\|_{\omega_0}$ follows from $\eps_k\to\eps_0$, $r_k\to r_0$, and \eqref{eq:x6} in connection with one of the lower semicontinuity estimates \eqref{eq:x3} -- \eqref{eq:x5}.
\item This follows from the definition and duality argument \eqref{eq:x6}.
\item Let $s_k$ denote a sequence in $\sM^{\omega_0}$. Let $u_k$ in $L^p(Q)$ or $\sB^p$ denote the (unique) representative of $U_k$ in the sense of \eqref{eq:repr}. Suppose that  $\|s_k\|_{\omega_0}\leq R$. Then  $(r_k)$ and $(\eps_k)$ are bounded sequences in $\R_{\geq 0}$, and $\sup_{k}\|u_k\|\leq \sup_kr_k<\infty$ (where $\|\cdot\|$ stands short for either $\|\cdot\|_{L^p(Q)}$ or $\|\cdot\|_{\sB^p}$). Thus we may pass to a subsequence such that~$r_k\to r_0$,  $\eps_k\to \eps_0$, and one of the following three cases applies:
  \begin{itemize}
  \item Case 1: $\inf_{k\in\N_0}\eps_k>0$. In that case we conclude (after passing to a further subsequence) that $u_k\weakto u_0$ weakly in $L^p(Q)$, and thus $U_k\to U_0=J^{\omega_0}_{\eps_0}u_0$ in $\mbox{Lin}(\sD)$.
  \item Case 2: $\eps_k=0$ for all $k\in\N_0$. In that case we conclude (after passing to a further subsequence) that $u_k\weakto u_0$ weakly in $\sB^p(Q)$, and thus $U_k\to U_0=J_0u_0$ in $\mbox{Lin}(\sD)$.
  \item Case 3: $\eps_k>0$ for all $k\in\N$ and $\eps_0=0$.  In that case we conclude (after passing to a further subsequence) that $u_k\tsq{\omega_0}u_0$, and thus $U_k\to U_0=J_0u_0$ in $\mbox{Lin}(\sD)$.
  \end{itemize}
  In all of these cases we deduce that $s_0=(U_0,\eps_0,r_0)\in\sM^{\omega_0}$, and $s_k\to s_0$ in $\sM$.
\item This is a direct consequence of (ii) -- (vi), and Lemma~\ref{L:metric-char}.
\end{enumerate}
\end{proof}

Now we are in position to prove Theorem~\ref{thm:Balder-Thm-two-scale}
\begin{proof}[Proof of Theorem~\ref{thm:Balder-Thm-two-scale}]
Let $\sM$, $\sM^{\omega_0}$, $J^{\omega_0}_\eps$ etc.~be defined as in Lemma~\ref{L:metric-struct}. 
\smallskip

{\it Step 1. (Identification of $(u\e)$ with a tight $\sM$-valued sequence).}
Since $u\e\in\sB^p$, by Fubini's theorem, we have $u\e(\omega,\cdot)\in L^p(Q)$ for $P$-a.e.~$\omega\in\Omega$. By modifying $u\e$ on a null-set in $\Omega\times Q$ (which does not alter two-scale limits in the mean), we may assume w.l.o.g.~that $u\e(\omega,\cdot)\in\ L^p(Q)$ for all $\omega\in\Omega$. Consider the measurable function $s\e:\Omega\to\sM$ defined as
\begin{equation*}
  s\e(\omega):=    \begin{cases}
      \big(J^{\omega}\e u\e(\omega,\cdot),\eps,\|u\e(\omega,\cdot)\|_{L^p(Q)}\big)&\text{if }\omega\in\Omega_0\\
      (0,0,0)&\text{else.}
    \end{cases}
\end{equation*}
We claim that $(s\e)$ is tight. To that end consider the integrand $h:\Omega\times\sM\to(-\infty,+\infty]$ defined by
\begin{equation*}
  h(\omega,(U,\eps,r)):=
  \begin{cases}
    \|(U,\eps,r)\|_{\omega}^p&\text{if }\omega\in\Omega_0\text{ and }(U,\eps,r)\in\sM^{\omega},\\
    +\infty&\text{else.}
  \end{cases}
\end{equation*}
From Lemma~\ref{L:metric-struct} we deduce that $h$ is a normal integrand and $h(\omega,\cdot)$ has compact sublevels for all $\omega\in\Omega$. Moreover, for all $\omega_0\in\Omega_0$ we have $s\e(\omega_0)\in\sM^{\omega_0}$ and thus $h(\omega_0,s\e(\omega_0))=2\|u\e(\omega_0,\cdot)\|^p_{L^p(Q)}+\eps$. Hence,
\begin{equation*}
  \int_\Omega h(\omega,s\e(\omega))\,dP(\omega)=2\|u\e\|^p_{\sB^p}+\eps.
\end{equation*}
We conclude that $(s\e)$ is tight.

{\it Step 2. (Compactness and  definition of $\boldsymbol{\nu}$)}. By appealing to Theorem~\ref{thm:Balder} there exists a subsequence (still denoted by $\eps$) and a Young measure $\boldsymbol{\mu}$ that is generated by $(s\e)$. Let $\boldsymbol{\mu_1}$ denote the first component of $\boldsymbol{\mu}$, i.e.~the Young measure on $\mbox{\rm Lin}(\sD)$ characterized for $\omega\in\Omega$ by
\begin{equation*}
  \int_{\mbox{\rm Lin}(\sD)}f(\xi)\,d\mu_{1,\omega}(\xi)=\int_{\sM}f(\xi_1)\,d\mu_\omega(\xi),
\end{equation*}
for all $f:\mbox{\rm Lin}(\sD)\to\R$ continuous and bounded, where $\sM\ni\xi=(\xi_1,\xi_2,\xi_3)\to\xi_1\in\mbox{\rm Lin}(\sD)$ denotes the projection to the first component.
By Balder's theorem, $\mu_\omega$ is concentrated on the limit points of $(s\e(\omega))$. By Lemma~\ref{L:metric-struct} we deduce that for all $\omega\in\Omega_0$ any limit point $s_0(\omega)$ of $s\e(\omega)$ has the form $s_0(\omega)=(J_0u,0,r)$ where $0\leq r<\infty$ and $u\in\sB^p$ is a $\omega$-two-scale limit of a subsequence of $u\e(\omega,\cdot)$. Thus, $\mu_{1,\omega}$ is supported on $\{J_0u\,:\,u\in \CL(\omega,(u\e(\omega,\cdot))\}$ which in particular is a subset of $(\sB^q)^*$. Since $J_0:\sB^p\to (\sB^q)^*$ is an isometric isomorphism (by the Riesz-Frech\'et theorem), we conclude that $\boldsymbol{\nu}=\{\nu_\omega\}_{\omega\in\Omega}$, $\nu_\omega(B):=\mu_{1,\omega}(J_0B)$ (for all Borel sets $B\subset\sB^p$ where $\sB^p$ is equipped with the weak topology) defines a Young measure on $\sB^p$, and for all $\omega\in\Omega_0$, $\nu_\omega$ is supported on $\CL(\omega,(u\e(\omega,\cdot))$. 
\smallskip

{\it Step 3. (Lower semicontinuity estimate).} Note that $h:\Omega\times\sM\to[0,+\infty]$,
\begin{equation*}
  h(\omega,(U,\eps,r)):=
  \begin{cases}
    \sup_{\varphi\in\overline\sD,\,\|\varphi\|_{\sB^q}\leq 1}|U(\varphi)|^p&\text{if }\omega\in\Omega_0\text{ and }(U,\eps,r)\in\sM^{\omega},\eps>0,\\
    \sup_{\varphi\in\sD,\,\|\varphi\|_{\sB^q}\leq 1}|U(\varphi)|^p&\text{if }\omega\in\Omega_0\text{ and }(U,\eps,r)\in\sM^{\omega},\eps=0,\\
    +\infty&\text{else.}
  \end{cases}
\end{equation*}
defines a normal integrand, as can be seen as in the proof of Lemma~\ref{L:metric-struct}. Thus Theorem~\ref{thm:Balder} implies that
\begin{equation*}
  \liminf_{\eps\to 0}\int_\Omega h(\omega,s\e(\omega))\,dP(\omega)\geq \int_\Omega\int_{\sM}h(\omega,\xi)\,d\mu_\omega(\xi)dP(\omega).
\end{equation*}
In view of Lemma~\ref{L:metric-struct} we have $    \sup_{\varphi\in\overline\sD,\,\|\varphi\|_{\sB^q}\leq 1}|(J^\omega_\eps u_\eps)(\omega,\cdot))(\varphi)|=\|u\e(\omega,\cdot)\|_{L^p(Q)}$ for $\omega\in\Omega_0$, and thus the left-hand side turns into $\liminf_{\eps\to 0}\|u\e\|^p_{\sB^p}$. Thanks to the definition of $\boldsymbol{\nu}$ the right-hand side turns into $\int_\Omega \int_{\sB^p}\|v\|_{\sB^p}^p\,d\nu_\omega(v)dP(\omega)$.
\smallskip

{\it Step 4. (Identification of the two-scale limit in the mean)}.
Let $\varphi\in\sD_0$. Then $h:\Omega\times\sM\to[0,+\infty]$,
\begin{equation*}
  h(\omega,(U,\eps,r)):=
  \begin{cases}
    U(\varphi)&\text{if }\omega\in\Omega_0,\,(U,\eps,r)\in\sM^\omega,\\
    +\infty&\text{else.}
  \end{cases}
\end{equation*}
defines a normal integrand. Since $h(\omega,s\e(\omega))=\int_Qu\e(\omega,x)\unf^*\varphi(\omega,x)\,dx$ for $P$-a.e.~$\omega\in\Omega$, we deduce that $|h(\cdot, s\e(\cdot))|$ is uniformly integrable. Thus, \eqref{eq:Balders-ineq} applied to $\pm h$ and the definition of $\boldsymbol{\nu}$ imply that
\begin{eqnarray*}
  \lim\limits_{\eps\to 0}\int_\Omega\int_Qu\e(\omega,x)(\unf^*\varphi)(\omega,x)\,dx\,dP(\omega)&=&  \lim\limits_{\eps\to 0}\int_\Omega h(\omega,s\e(\omega))\,dP(\omega)\\
  &=&\int_\Omega\int_{\sB^p}h(\omega,v)\,d\nu_\omega(v)\,dP(\omega)\\
  &=&\int_\Omega\int_{\sB^p}\ex{\int_Qv\varphi}\,d\nu_\omega(v)\,dP(\omega).
\end{eqnarray*}
Set $u:=\int_\Omega\int_{\sB^p}v\,d\nu_\omega(v)dP(\omega)\in\sB^p$. Then Fubini's theorem yields
\begin{eqnarray*}
  \lim\limits_{\eps\to 0}\int_\Omega\int_Qu\e(\omega,x)(\unf^*\varphi)(\omega,x)\,dx\,dP(\omega)&=& \ex{\int_Qu\varphi}.
\end{eqnarray*}
Since $\mbox{span}(\sD_0)\subset \sB^q$ dense, we conclude that $u\e\ts u$.
\smallskip

{\it Step 5. Recovery of quenched two-scale convergence}. Suppose that $\nu_\omega$ is a delta distribution on $\sB^p$, say $\nu_\omega=\delta_{v(\omega)}$ for some measurable $v:\Omega\to\sB^p$. Note that $h:\Omega\times\sM\to[0,+\infty]$,
\begin{equation*}
  h(\omega,(U,\eps,r)):=-d(U,J_0v(\omega);\mbox{\rm Lin}(\sD))
\end{equation*}
is a normal integrand and $|h(\cdot, s\e(\cdot))|$ is uniformly integrable. Thus,  \eqref{eq:Balders-ineq} yields
\begin{eqnarray*}
  &&\limsup\limits_{\eps\to 0}\int_{\Omega} d(J^\omega\e u\e(\omega,\cdot),J_0v(\omega);\mbox{\rm Lin}(\sD))\,dP(\omega)\\
  &=&-\liminf\limits_{\eps\to 0}\int_\Omega h(\omega,s\e(\omega))\,dP(\omega)\\
  &\leq&-\int_\Omega\int_{\sB^p}h(\omega,J_0v)\,d\nu_\omega(v)\,dP(\omega)=-\int_\Omega h(\omega,J_0v(\omega))\,dP(\omega)=0.
\end{eqnarray*}
Thus, there exists a subsequence (not relabeled) such that $d(J^\omega\e u\e(\omega,\cdot),J_0v(\omega);\mbox{\rm Lin}(\sD))\to 0$ for a.e.~$\omega\in\Omega_0$. In view of Lemma~\ref{L:metric-char} this implies that $u\e\tsq{\omega}v(\omega)$ for a.e.~$\omega\in\Omega_0$.
\end{proof}

\begin{proof}[Proof of Lemma~\ref{lem:Balder-Lem-two-scale}]
  {\it Step 1. Representation of the functional by a lower semicontinuous integrand on $\sM$.}\\
  For all $\omega_0\in\Omega_0$ and $s=(U,\eps,r)\in\sM^{\omega_0}$ we write $\pi^{\omega_0}(s)$ for the unique representation $u$ in $\sB^p$ (resp. $L^p(Q)$) of $U$ in the sense of \eqref{eq:repr}. We thus may define for $\omega\in\Omega_0$ and $s\in\sM^{\omega_0}$ the integrand
\begin{equation*}
  \overline h(\omega_0,s):=
  \begin{cases}
    \int_Qh(\tau_{\frac{x}{\eps}}\omega,x,(\pi^{\omega_0}s)(x))\,dx&\text{if }s=(U,\eps,s)\text{ with }\eps>0,\\
    \int_\Omega\int_Qh(\omega,x,(\pi^{\omega_0}s)(x))\,dx\,dP(\omega)&\text{if }s=(U,\eps,s)\text{ with }\eps=0.
  \end{cases}
\end{equation*}

We extend $\overline h(\omega_0,\cdot)$ to $\sM$ by $+\infty$, and define $\overline h(\omega,\cdot)\equiv 0$ for $\omega\in\Omega\setminus\Omega_0$. We claim that $\overline h(\omega,\cdot):\sM\to(-\infty,+\infty]$ is lower semicontinuous for all $\omega\in\Omega$. It suffices to consider $\omega_0\in\Omega_0$ and a convergent sequence $s_k=(U_k,\eps_k,r_k)$ in $\sM^{\omega_0}$. For brevity we only consider the (interesting) case when $\eps_k\downarrow \eps_0=0$. Set $u_k:=\pi^{\omega_0}(s_k)$. By construction we have
\begin{equation*}
  \overline h(\omega_0,s_k)=\int_Q h(\tau_{\frac{x}{\eps_k}}\omega_0,u_k(\omega_0,x))\,dx,
\end{equation*}
and
\begin{equation*}
  \overline h(\omega_0,s_0)=\int_{\Omega}\int_Q h(\omega,x,u_0(\omega,x))\,dx\,dP(\omega).
\end{equation*}
Since $s_k\to s_0$ and $\eps_k\to 0$, Lemma~\ref{L:metric-struct} (vi) implies that $u_k\tsq{\omega_0}u_0$, and since $h$ is assumed to be a quenched two-scale normal integrand, we conclude that $\liminf\limits_{k}  \overline h(\omega_0,s_k)\geq   \overline h(\omega_0,s_0)$, and thus $\overline h$ is a normal integrand.

{\it Step 2. Conclusion.}\\
As in Step~1 of the proof of Theorem~\ref{thm:Balder-Thm-two-scale} we may associate with the sequence $(u\e)$ a sequence of measurable functions $s\e:\Omega\to\sM$ that (after passing to a subsequence that we do not relabel) generates a Young measure $\boldsymbol{\mu}$ on $\sM$. Since by assumption $u\e$ generates the Young measure ${\boldsymbol \nu}$ on $\sB^p$, we deduce that the first component $\boldsymbol{\mu_1}$ satisfies $\nu_\omega(B)=\mu_{\omega}(J_0B)$ for any Borel set $B$. Applying \eqref{eq:Balders-ineq} to the integrand $\overline h$ of Step~1, yields
\begin{eqnarray*}
  &&\liminf\limits_{\e\to 0}\int_\Omega\int_Q h(\tau_{\frac{x}{\eps}}\omega_0,u\e(\omega_0,x))\,dx\,dP(\omega)\\
  &=&\liminf\limits_{\e\to 0}\int_\Omega\overline h(\omega,s\e(\omega))\,dP(\omega)\\
  &\geq&\int_\Omega\int_{\sM}\overline h(\omega,\xi)\,d\mu_\omega(\xi)\,dP(\omega)\\
  &=&\int_\Omega\int_{\sB^p}\Big(\int_\Omega\int_Qh(\tomega,x,v(\tomega,x))\,dx\,dP(\tomega)\Big)\,d\nu_\omega(v)\,dP(\omega).
\end{eqnarray*}
\end{proof}

\begin{proof}[Proof of Lemma~\ref{L:fromquenchedtomean}]
  By (b) and (c) the sequence $(\tilde u\e)$ is bounded in $\sB^p$ and thus we can pass to a subsequence such that $(\tilde u\e)$ generates a Young measure $\boldsymbol \nu$. Set $\tilde u:=\int_\Omega\int_{\sB^p}v\,d\nu_\omega(v)\,dP(\omega)$ and note that Theorem~\ref{thm:Balder-Thm-two-scale} implies that $\tilde u\e\wt \tilde u$ weakly two-scale in the mean. On the other hand the theorem implies that $\nu_\omega$ concentrates on the quenched two-scale cluster points of $(u^\omega\e)$ (for a.e.~$\omega\in\Omega$). Hence, in view of (a) we conclude that for a.e.~$\omega\in\Omega$ the measure $\nu_\omega$ is a Dirac measure concentrated on $u$, and thus $\tilde u=u$ a.e.~in $\Omega\times Q$.  
\end{proof}

\subsection{Quenched homogenization of convex functionals}\label{Section:4:3}
In this section we demonstrate how to lift homogenization results w.r.t.~two-scale convergence in the mean to quenched statements at the example of a convex minimization problem. Throughout this section we assume that $V:\Omega\times Q\times\R^{d\times d}\to\R$ is a convex integrand satisfying the assumptions $(A1)-(A3)$ of Section \ref{Section_Convex}. For $\omega\in\Omega$ we define $\cE^{\omega}\e: W^{1,p}_0(Q)\to \re{}$, 
\begin{equation*}
\cE^{\omega}_{\eps}(u):=\int_{Q}V\left(\tau_{\frac{x}{\eps}}\omega, x,\nabla^s u(x)\right)\,dx,
\end{equation*}
and recall from Section~\ref{Section_Convex} the definition \eqref{energy} of the averaged energy $\cE_{\eps}$ and the definition \eqref{energy_hom} of the two-scale limit energy
$\cE_{0}$. 
The goal of this section is to relate two-scale limits of ``mean''-minimizers, i.e.~functions $u\e\in L^p(\Omega)\otimes W^{1,p}_0(Q)$ that minimize $\cE_{\eps}$, with limits of ``quenched''-minimizers, i.e.~families  $\{u\e(\omega)\}_{\omega\in\Omega}$ of minimizers to $\cE^{\omega}\e$ in $W^{1,p}_0(Q)$.
\begin{thm}
  \label{thm:Quenched-hom-convex-grad} Let $u\e\in L^{p}(\Omega)\otimes W_{0}^{1,p}(Q)$
  be a minimizer of $\cE_{\eps}$. Then there exists a subsequence such that $(u\e,\nabla u\e)$ generates a Young measure $\boldsymbol{\nu}$ in $\sB:=(\sB^p)^{d+d^2}$ in the sense of Theorem~\ref{thm:Balder-Thm-two-scale}, and for $P$-a.e.~$\omega\in\Omega$, $\nu_{\omega}$ concentrates on the set $    \big\{\,(u,\nabla u+\chi)\,:\,\cE_0(u,\chi)=\min\cE_0\,\big\}$ of minimizers of the limit functional.
  Moreover, if  $V(\omega,x,\cdot)$ is strictly convex for all $x\in Q$ and $P$-a.e.~$\omega\in\Omega$, then the minimizer $u\e$ of $\cE_{\eps}$ and the minimizer $(u,\chi)$ of $\cE_0$ are  unique, and for $P$-a.e.~$\omega\in\Omega$ we have (for a not relabeled subsequence)
  \begin{gather*}
    u\e\weakto u\text{ weakly in }W^{1,p}(Q),\qquad u\e(\omega,\cdot)\tsq{\omega}u,\qquad\nabla u\e(\omega,\cdot)\tsq{\omega}\nabla u+\chi,\\
    \text{and }\min\cE^\omega_\eps=\cE^\omega_\eps(u\e(\omega,\cdot))\to \cE_0(u,\chi)=\min\cE_0.
  \end{gather*}
\end{thm}
\begin{remark}[Identification of quenched two-scale cluster points]
  If we combine Theorem~\ref{thm:Quenched-hom-convex-grad} with the identification of the support of the Young measure in Theorem~\ref{thm:Balder-Thm-two-scale} we conclude the following: There exists a subsequence such that
$(u\e,\nabla u\e)$ two-scale converges in the mean to a limit of the form $(u_0,\nabla u_0+\chi_0)$ with $\cE_0(u_0,\chi_0)=\min\cE_0$, and for a.e.~$\omega\in\Omega$ the set of quenched $\omega$-two-scale cluster points $\CL(\omega, (u\e(\omega,\cdot),\nabla u\e(\omega,\cdot)))$ is contained in $\big\{\,(u,\nabla u+\chi)\,:\,\cE_0(u,\chi)=\min\cE_0\,\big\}$. In the strictly convex case we further obtain that $\CL(\omega, (u\e(\omega,\cdot),\nabla u\e(\omega,\cdot)))=\{(u,\nabla u+\chi)\}$ where $(u,\chi)$ is the unique minimizer to $\cE_0$. Note, however, that our argument (that extracts quenched two-scale limits from the sequence of ``mean'' minimizers) involves an exceptional $P$-null-set that a priori depends on the selected subsequence. This is in contrast to the classical result in \cite{DalMaso1986} which is based on a subadditive ergodic theorem and states that there exists a set of full measure $\Omega'$ such that for all $\omega\in\Omega'$ the minimizer $u\e^\omega$ to $\cE^{\omega}_\eps$ weakly converges in $W^{1,p}(Q)$ to the deterministic minimizer $u$ of the reduced functional $\cE_{\hom}$ for any sequence $\eps\to 0$.
\end{remark}
In the proof of Theorem~\ref{thm:Quenched-hom-convex-grad} we combine homogenization in the mean in form of Theorem~\ref{thm1}, the connection to quenched two-scale limits via Young measures in form of Theorem~\ref{thm:Balder-Thm-two-scale}, and a recent result by Nesenenko and the first author that states that $V$ is a quenched two-scale normal integrand:
\begin{lemma}[\mbox{\cite[Lemma~5.1]{HeidaNesenenko2017monotone}}]\label{lem:General-Hom-Convex}
$V$ is a quenched two-scale normal integrand in the sense of Definition~\ref{D:qtscnormal}.
\end{lemma}

\begin{proof}[Proof of Theorem~\ref{thm:Quenched-hom-convex-grad}]  
  {\it Step 1. (Identification of the support of $\boldsymbol{\nu}$).}
  
  Since $u\e$ is a sequence of minimizers, by Corollary~\ref{C:thm1} there exists a subsequence (not relabeled) and minimizers $(u,\chi)\in W^{1,p}_0(Q)\times (L^p_{\pot}(\Omega)\otimes L^p(Q)^d)$ of $\cE_0$ such that
  that $u\e \wt u \text{ in }\ltp^d$, $\nabla u\e \wt \nabla u+\chi  \text{ in }\ltp^{d\times d}$, and 
  \begin{equation}\label{eq:conv-minima}
    \lim\limits_{\eps\to 0}\min\cE\e=    \lim\limits_{\eps\to 0}\cE\e(u\e)=\cE_0(u,\chi)=\min\cE_0.
  \end{equation}
  In particular, the sequence $(u\e,\nabla u\e)$ is bounded in $\sB$.  By Theorem~\ref{thm:Balder-Thm-two-scale} we may pass to a further subsequence (not relabeled) such that $(u\e,\nabla u\e)$ generates a Young measure $\boldsymbol{\nu}$ on $\sB$.   Since $\nu_\omega$ is supported on the set of quenched $\omega$-two-scale cluster points of $(u\e(\omega,\cdot),\nabla u\e(\omega,\cdot))$, we deduce from Lemma~\ref{lem:sto-conver-grad} that the support of $\nu_\omega$ is contained in $\sB_0:=\{\xi=(\xi_1,\xi_2)=(u',\nabla u'+\chi')\,:\,u'\in W^{1,p}_0(Q),\,\chi\in L^p_{\pot}(\Omega)\otimes L^p(Q)^d\}$ which is a closed subspace of $\sB$. Moreover, thanks to the relation of the generated Young measure and stochastic two-scale convergence in the mean, we have $(u,\chi)=\int_\Omega \int_{\sB_0}(\xi_1,\xi_2-\nabla\xi_1)\,\nu_\omega(d\xi)\,dP(\omega)$. By Lemma~\ref{lem:General-Hom-Convex}, $V$ is a quenched two-scale normal integrand and thus Lemma~\ref{lem:Balder-Lem-two-scale} implies that
  \begin{equation*}
    \lim\limits_{\eps\to 0}\cE_\eps(u\e)\geq \int_\Omega\int_{\sB}\Big(\int_\Omega\int_Q V(\tomega,x,\xi_2)\,dx\,dP(\tomega)\Big)\,\nu_\omega(d\xi)\,dP(\omega).
  \end{equation*}
  In view of \eqref{eq:conv-minima} and the fact that $\nu_\omega$ is supported in $\sB_0$, we conclude that
  \begin{equation*}
    \min\cE_0\geq \int_\Omega\int_{\sB_0}\cE_0(\xi_1,\xi_2-\nabla\xi_1)\,\nu_\omega(d\xi)\,dP(\omega)\geq \min\cE_0\int_\Omega\int_{\sB_0}\nu_\omega(d\xi)dP(\omega).
  \end{equation*}
  Since $\int_\Omega\int_{\sB_0}\nu_\omega(d\xi)dP(\omega)=1$, we have $\int_\Omega\int_{\sB_0}|\cE_0(\xi_1,\xi_2-\nabla\xi_1)-\min\cE_0|\,\nu_\omega(d\xi)\,dP(\omega)= 0$, and thus we conclude that for $P$-a.e.~$\omega\in\Omega_0$,  $\nu_\omega$ concentrates on $\{(u,\nabla u+\chi)\,:\,\cE_0(u,\chi)=\min\cE_0\}$.
  \smallskip

  {\it Step 2. (The strictly convex case).}

  The uniqueness of $u\e$ and $(u,\chi)$ is clear. From Step~1 we thus conclude that $\nu_\omega=\delta_{\xi}$ where $\xi=(u,\nabla u+\chi)$. Theorem~\ref{thm:Balder-Thm-two-scale} implies that $(u\e(\omega,\cdot),\nabla u\e(\omega,\cdot))\tsq{\omega}(u,\nabla u+\chi)$ (for $P$-a.e.~$\omega\in\Omega$). 
  By Lemma~\ref{lem:General-Hom-Convex}, $V$ is a quenched two-scale normal integrand and thus for $P$-a.e.~$\omega\in\Omega$,
  \begin{equation*}
    \liminf\limits_{\eps\to 0}\cE^\omega_\eps(u\e(\omega,\cdot))\geq \cE_0(u,\chi)=\min\cE_0.
  \end{equation*}
  On the other hand, since $u\e(\omega,\cdot)$ minimizes $\cE^\omega_\eps$, we deduce by a standard argument that for $P$-a.e.~$\omega\in\Omega$,
  \begin{equation*}
    \lim\limits_{\eps\to 0}\min\cE^\omega_\eps=\lim\limits_{\eps\to 0}\cE^\omega_\eps(u\e(\omega,\cdot))=\cE_0(u,\chi)=\min\cE_0.
  \end{equation*}
\end{proof}
\begin{appendix}
\section{Proof of Lemma \ref{lem8}}\label{appendix:1}
\begin{proof}
We prove the claim for $\chi=D\varphi$ for $\varphi\in \sob(\Omega)^d$ and the general case follows by density.

Recall, the stationary extension of $D\varphi$ is given by $S{D\varphi}(\omega,x)=D\varphi(\tau_x\omega)$ and we have $\nabla S{\varphi}(\omega,x)=S{D\varphi}(\omega,x)$. Let $R>0$, $K>0$ and $\eta_{R}\in C^{\infty}_c(B_{R+K})$ be a cut-off function satisfying $\eta=1$ in $B_R$, $0 \leq \eta\leq 1$ and $|\nabla \eta_{R}|\leq \frac{2}{K}$.
Using stationarity of $P$, we obtain
\begin{equation*}
\ex{|D\varphi|^p}=\ex{\fint_{B_R}|\nabla S{\varphi}|^p dx}=\ex{\fint_{B_R}|\nabla(\eta_R S{\varphi})|^p dx}\leq \ex{\frac{1}{|B_R|}\int_{\re d}|\nabla(\eta_R S{\varphi})|^p dx}.
\end{equation*}
Using this and Korn's inequality in $L^p(\R^d)$,
\begin{align*}
\ex{|D\varphi|^p}&\leq 2 \ex{\frac{1}{|B_R|}\int_{\re d}|\nabla^s(\eta_R S{\varphi})|^p}\\ &=2\ex{\fint_{B_R}|\nabla^s S{\varphi}|^p dx}+\frac{2}{|B_R|}\ex{\int_{B_{R+K}\setminus B_R}|\nabla^s(\eta_R S{\varphi})|^p dx}.
\end{align*}
The first term on the right-hand side of the above inequality equals $2 \ex{|D^s\varphi|^p}$ and therefore to conclude the proof, it is sufficient to show that the second term vanishes in the limit $R\to \infty$.
We have
\begin{align}\label{ineq47}
\begin{split}
&\frac{1}{|B_R|}\ex{\int_{B_{R+K}\setminus B_R}|\nabla^s(\eta_R S{\varphi})|^p dx}\leq  \frac{1}{|B_R|}\ex{\int_{B_{R+K}\setminus B_R}|\nabla(\eta_R S{\varphi})|^p dx}\\
&\qquad\qquad \leq  \frac{C}{|B_R|} \ex{\int_{B_{R+K}\setminus B_R}|\eta_R|^p|\nabla S{\varphi}|^p +|\nabla \eta_R|^p|S{\varphi}|^p dx}  \\
&\qquad\qquad \leq  \frac{C}{|B_R|}\ex{\int_{B_{R+K}\setminus B_R}|\nabla S{\varphi}|^pdx}+\frac{C}{|B_R|K^p}\ex{\int_{B_{R+K}\setminus B_R}|S{\varphi}|^p dx}.
\end{split}
\end{align}
For the first term on the right-hand side, we have
\begin{align*}
\frac{C}{|B_R|}\ex{\int_{B_{R+K}\setminus B_R}|\nabla S{\varphi}|^pdx}= & \frac{C|B_{R+K}|}{|B_R|}\ex{\fint_{B_{R+K}}|\nabla S{\varphi}|^pdx}-C\ex{\fint_{B_R}|\nabla S{\varphi}|^pdx}\\  = & C \ex{|D\varphi|^p}\brac{\frac{|B_{R+K}|}{|B_R|}-1}
\end{align*}
and as $R\to \infty$ the last expression vanishes. Similarly, the second term on the right-hand side of (\ref{ineq47}) vanishes as $R\to \infty$.
\end{proof}
\end{appendix}
\section*{Acknowledgments}
The authors thank Alexander Mielke for fruitful discussions and valuable comments. SN and MV thank Goro Akagi for useful discussions and valuable comments.
MH has been funded by Deutsche Forschungsgemeinschaft (DFG) through grant CRC 1114 ``Scaling Cascades
in Complex Systems'', Project C05 ``Effective models for interfaces with many scales''. SN and MV were supported by the DFG
in the context of TU Dresden's Institutional Strategy ``The Synergetic University''.

\nocite{*}
\bibliographystyle{abbrv}
\bibliography{ref}

\end{document}